\documentclass[aos]{imsart}
\usepackage{amsthm,amsmath,amssymb}
\usepackage{natbib}
\usepackage{multirow}
\usepackage[pdftex]{graphicx}
\usepackage{caption}
\usepackage{subcaption}
\usepackage{float}
\usepackage{makecell}
\usepackage{booktabs}
\usepackage{array}
\usepackage{url}
\usepackage{algorithm}
\usepackage{algorithmic}
\usepackage{bm}
\usepackage{cancel}
\usepackage{bbm}
\usepackage{smile}
\usepackage{mathtools}
\usepackage{relsize}
\usepackage{lipsum}
\usepackage{mathrsfs}
\usepackage{xr}
\usepackage{dsfont}
\usepackage{epstopdf}
\usepackage{caption}
\usepackage{marvosym}
\usepackage{diagbox}
\usepackage{tkz-berge}

\makeatletter
\newcommand*{\addFileDependency}[1]{
  \typeout{(#1)}
  \@addtofilelist{#1}
  \IfFileExists{#1}{}{\typeout{No file #1.}}
}
\makeatother

\newcommand{\atanh}{\operatorname{atanh}}

\usepackage{natbib}
\usepackage{multirow}

\usepackage[colorlinks=true,
            linkcolor=blue,
            urlcolor=blue,
            citecolor=blue]{hyperref}

\usepackage[utf8]{inputenc}
\definecolor{fondpaille}{cmyk}{0,0,0.1,0}

\newcommand{\maxdeg}{\operatorname{maxdeg}}

\newcommand*{\TV}{\operatorname{TV}}

\renewcommand*{\liminf}{\displaystyle \operatornamewithlimits{liminf}}
\renewcommand*{\limsup}{\displaystyle \operatornamewithlimits{limsup}}

\begin{document}

\begin{frontmatter}

\title{Property Testing in High Dimensional Ising Models}
\runtitle{Property Testing in Ising models}

\begin{aug}
 \author{\fnms{Matey} \snm{Neykov}\corref{}\ead[label=e1]{mneykov@stat.cmu.edu}}\thanksref{}
 \and
 \author{\fnms{Han} \snm{Liu}\corref{}\ead[label=e2]{hanliu@northwestern.edu}}\thanksref{t2}
   \address[a]{Department of Statistics \& Data Science,\\Carnegie Mellon University, \\ Pittsburgh, PA 15213, USA  \\ \printead{e1}\\}
  \address[b]{Department of Electrical Engineering\\ and Computer Science\\ Department of Statistics\\ Northwestern University, \\ Evanston, Il 60208, USA \\\printead{e2}\\}
\end{aug}

\affiliation{Carnegie Mellon University and Northwestern University}

 \thankstext{t2}{Research partially supported by NSF DMS1454377-CAREER; NSF IIS 1546482-BIGDATA; NIH R01MH102339; NSF IIS1408910; NIH R01GM083084.}
 
 \runauthor{Neykov and Liu}

\begin{abstract}
This paper explores the information-theoretic limitations of graph property testing in zero-field Ising models. Instead of learning the entire graph structure, sometimes testing a basic graph property such as connectivity, cycle presence or maximum clique size is a more relevant and attainable objective. Since property testing is more fundamental than graph recovery, any necessary conditions for property testing imply corresponding conditions for graph recovery, while custom property tests can be statistically and/or computationally more efficient than graph recovery based algorithms. Understanding the statistical complexity of property testing requires the distinction of \textit{ferromagnetic} (i.e., positive interactions only) and general Ising models. Using combinatorial constructs such as graph packing and strong monotonicity, we characterize how target properties affect the corresponding minimax upper and lower bounds within the realm of ferromagnets. On the other hand, by studying the detection of an \textit{antiferromagnetic} (i.e., negative interactions only) Curie-Weiss model buried in Rademacher noise, we show that property testing is strictly more challenging over general Ising models. In terms of methodological development, we propose two types of correlation based tests: computationally efficient screening for ferromagnets, and score type tests for general models, including a fast cycle presence test. Our correlation screening tests match the information-theoretic bounds for property testing in ferromagnets in certain regimes.
\end{abstract}
 
\begin{keyword}
\kwd{Ising models}
\kwd{minimax testing}
\kwd{Dobrushin's comparison theorem}
\kwd{antiferromagnetic Curie-Weiss detection}
\kwd{two-point function bounds}
\kwd{correlation tests}
\end{keyword}

\end{frontmatter}

\section{Introduction}

The Ising model is a pairwise binary model introduced by statistical physicists as a model for spin systems with the goal of understanding spontaneous magnetization and phase transitions \citep{ising1925beitrag}. More recently the model has found applications in diverse areas such as image analysis \citep{geman1984stochastic}, bioinformatics and social networks \citep{ahmed2009recovering}. In statistics the model is an archetypal example of an undirected graphical model. A central topic of interest in graphical models research is estimating the structure (also known as \textit{structure learning}) of, or inferring questions about, the underlying graph based on a sample of observations. Substantial progress has been made towards understanding structure learning. Popular procedures developed for high dimensional graph estimation include $\ell_1$-regularization methods \citep{yuan2007model, rothman08spice, Liu2009Nonparanormal, Ravikumar2011High, Cai2011Constrained}, neighborhood selection \citep{Meinshausen2006High, bresler2008reconstruction} and thresholding \citep{montanari2009graphical}. In this paper, instead of focusing on learning the structure of the entire graph, we study the weaker inferential problem of \textit{property testing}, i.e., testing whether the graph structure obeys certain properties based on a sample of $n$ observations. Specifically, we study the \textit{zero-field} Ising model. 

Formally, a zero-field Ising model is a collection of $d$ binary $\pm 1$ valued random variables $\bX = (X_1,X_2,\ldots X_d)$, hereto referred to as spins, which are distributed according to the law
\begin{align}\label{general:ising}
\PP_{\theta, G_\wb}(\bX = \xb) \varpropto \exp\bigg(\theta \sum_{(u,v) \in E(G)} w_{uv} x_u x_v\bigg),
\end{align}
where $\theta \geq 0$, $G_\wb = (G, \wb)$ where $G = ([d], E)$ is a simple graph, and $\wb \in \RR^{d \choose 2}$ are weights on the graph's edges (i.e., for each edge $(u,v) \in E(G)$, $\wb$ specifies the edge weight $w_{uv}$, and for any $(u,v) \not \in E(G)$: $w_{uv} = 0$). Using \eqref{general:ising}, it is easily seen that the vector $\bX$ is Markov to the graph $G$, or in other words, any two non-adjacent spins $X_u$ and $X_v$ ($(u,v) \not \in E(G)$) are independent given the values of all the remaining spins. 

Note that model (\ref{general:ising}) is overparametrized. However, when $w_{uv}$ are viewed as fixed constants, this specification allows one to study the behavior of $\bX$ for different values of $\theta$. In statistical physics the parameter $\theta = \frac{1}{T}$ where $T$ stands for temperature, and is often referred to as the \textit{inverse temperature} of the system. The temperature plays an important role in changing the ``balance'' of the distribution of the spins, and is the main cause for the system to undergo phase transitions. The complicated behavior of the Ising model at different temperatures suggests that the difficulty of property testing is related to $\theta$. The main focus of this paper is uncovering necessary and sufficient conditions on the temperature, sample size, dimensionality and graph properties, allowing one to conduct property tests even when the data is sampled from the most challenging models. Understanding such limitations is practically useful, since necessary conditions can provide a benchmark for algorithm comparisons, while mismatches between sufficient and necessary conditions can prompt to searching for better algorithms.

To elaborate on the type of problems we study, let $[d] = \{1,\ldots,d\}$ be a vertex set of cardinality $d$ and let $\cG_d$ be the set of all graphs over the vertex set $[d]$. A binary graph property $\cP$ is a map $\cP : \cG_d \mapsto \{0,1\}$. Given a sample of $n$ observations from a zero-field Ising model with an underlying simple graph $G$, the goal of property testing is to test the hypotheses 
\begin{align}\label{property:testing:intro:def}
\Hb_0: \cP(G) = 0 \mbox{  versus  } \Hb_1: \cP(G) = 1.
\end{align}
Below we give three specific instances of property tests. We furthermore give informal summaries of our findings, which are presented more rigorously in Sections \ref{generic:lower:bounds:sec} and \ref{ferromagnetic:algorithms}. \vspace{.2 cm}

\noindent{\bf Connectivity.}  A graph is connected if and only if each pair of its vertices is connected via a path. Define $\cP$ as $\cP(G) = 0$ if $G$ is disconnected and $\cP(G) = 1$ otherwise. Testing for connectivity is equivalent to testing whether the variables can be partitioned into two independent sets. It turns out that in \textit{simple ferromagnets} (that is, models whose spin-spin interactions satisfy $w_{uv} \in \{0,1\}$ for all $(u,v)$) connectivity testing is possible iff
$$
\sqrt{\frac{\log d}{n}} \lesssim \theta,
$$
where $\lesssim$ is inequality up to constants. Note that there is no upper bound on $\theta$ and as long as $\theta$ is large enough connectivity testing is always possible. 

\vspace{.2 cm}

\noindent{\bf Cycle Presence.} If a graph is a forest, i.e., a graph containing no cycles, its structure can be estimated efficiently using a graph selection procedure based on a maximum spanning tree construction proposed by \cite{chow68approximating}. It is therefore sometimes of interest to test whether the underlying graph is a forest. In this example $\cP$ satisfies $\cP(G) = 0$ if $G$ is a forest and $\cP(G) = 1$ otherwise. We will also refer to forest testing as cycle testing, since it is equivalent to testing whether the graph contains cycles. In simple ferromagnets, cycle testing is possible iff
$$
\sqrt{\frac{\log d}{n}} \lesssim \theta \lesssim \log \frac{n}{\log d}.
$$
In contrast to connectivity, there appears to be an upper bound on the temperature when one tests for cycle presence. 
\vspace{.2 cm}

\noindent{\bf Clique Size.} Another relevant question is to test whether the size of a maximum clique (i.e., a maximum complete subgraph) contained in the graph is less than or equal to some integer $m-1$ versus the alternative that a maximum clique is of size at least $m$. This is a relevant question since Hammersley-Clifford's theorem \citep{grimmett2018probability} ensures that the Ising distribution can be factorized over the cliques in the graph, and hence knowing the maximal size of any clique puts a restriction on this factorization. In this example set $\cP(G) = 0$ if $G$ contains no $m$-clique, and $\cP(G) = 1$ otherwise. Let the maximum degree\footnote{The largest number of neighbors of any vertex of $G$.} of the graph $G$ be $s$. It turns out that testing the clique size is impossible in simple ferromagnets unless 
$$
\sqrt{\frac{\log d}{n}} \lesssim \theta \lesssim \frac{\log \frac{n}{\log d}}{s}.
$$
Different from before, the maximum degree appears in the upper bound on $\theta$. We will show that testing the maximum clique size is possible when 
$$
  \sqrt{\frac{\log d}{n}} \lesssim\theta\lesssim \frac{1}{s}, \mbox{ and when } m = s+1 \mbox { and } \frac{\log s}{s} \ll \theta \lesssim  \frac{\log \frac{n}{\log d}}{s},
$$
where $\ll$ is used in the sense ``much larger than'' (for a precise definition see the notation section below). This matches the previous two bounds up to constants.

\vspace{.2 cm}

\noindent By definition, property testing is a statistically simpler task compared to learning the entire graph structure, since if a graph estimate is available, property testing can be done via a deterministic procedure (although possibly a computationally challenging one). An important implication of this observation is that any quantification on how hard testing a particular graph property is, immediately implies that estimating the entire graph is at least as hard. Conversely, any algorithm capable of learning the graph structure with high confidence can be applied to test any property while preserving the same confidence. Importantly however, there could exist tests geared towards particular graph properties which can statistically and/or computationally outperform generic graph learning methods. 

Under the assumption that the maximum degree of $G$ is at most $s$, foundational results on the limitations of structure learning of Ising models were given by \cite{santhanam2012information}. In view of the relationship between property tests and structure learning, our work can be seen as a generalization of necessary conditions for structure learning. Our results also help to paint a more complete picture of the statistical complexity of testing in Ising models. Unlike in structure learning, understanding property testing requires the distinction of ferromagnetic and general Ising models, of which the latter exhibit strictly stronger limitations. In terms of methodological development, we formalize correlation based property tests which can be customized to target any graph property. We now outline the three major contributions of this work.

\subsection{Summary of Contributions}

Our first contribution is to provide necessary conditions for property testing in ferromagnets. We give a generic lower bound on the inverse temperature (Theorem \ref{ising:single:edge}), demonstrating that property testing is difficult in  high temperature regimes. A key role in the proof is played by Dobrushin's comparison theorem \citep{follmer1988random}, which a is powerful tool for comparing discrepancies between Gibbs measures based on their local specifications. We further formalize the class of \textit{strongly monotone} graph properties, and show that when the temperature drops below a certain property dependent threshold, testing strongly monotone properties becomes challenging (Theorem \ref{biclques:theorem}). We also provide an analogue of Theorem \ref{ising:single:edge} specialized for strongly monotone properties (Proposition \ref{simple:lower:bound:monotone}). Our general results are applied to obtain bounds on testing connectivity, cycles and maximum clique size. 

Our second contribution is to design several correlation based tests and understand their limitations. First, we formalize and study a generic correlation screening algorithm for ferromagnets. We show that this algorithm works well at high temperature regimes (Remark \ref{remark:on:corr:test} and Corollary \ref{generic:property:test:cor}), and could be successful even beyond this regime for some properties (Section \ref{ferromagnetic:corr:screening:examples}). To analyze the algorithms at low temperature regimes we develop a novel ``no-edge'' correlation bound for graphs of bounded degree (see Proposition \ref{generic:property:test:cor}), which may be of independent interest. We apply those algorithms to testing connectivity, cycles and maximum clique size and discover that they match the derived lower bounds in certain regimes. Second, we adapt the correlation decoders of \cite{santhanam2012information} to property testing for general Ising models, and we develop a computationally tractable cycle test (Section \ref{comp:eff:cycle:test}). 

Our third contribution is to study necessary conditions for general Ising models, i.e., models including both \textit{ferromagnetic and antiferromagnetic}\footnote{Inspired by statistical physics jargon, throughout the paper we use the terms ferromagnetic and antiferromagnetic to refer to positive and negative interactions respectively.} interactions. Specifically we argue that testing strongly monotone properties over general models requires more stringent conditions than performing the same tests over ferromagnets (Theorem \ref{scaling:theorem} and Proposition \ref{connectivity:low:temp}). In order to prove this result we demonstrate that it is very difficult to detect the presence of an antiferromagnetic Curie-Weiss\footnote{i.e., an antiferromagnetic model with a complete graph.} \citep[e.g., see][]{kochmanski2013curie} model buried in Rademacher noise, which to the best of our knowledge is the first attempt to analyze this problem. The detection problem we consider is in part inspired by the works \cite{addario2010comb, arias2012detection, arias2015detectingpositive, arias2015detecting}. 

\subsection{Related Work}

Recent works on Ising models related to the Curie-Weiss model include \cite{berthet2016exact, mukherjee2016global}. An interesting paper on testing goodness-of-fit in Ising models by \cite{daskalakis2018testing}, uses tests based on minimal pairwise correlations which are similar in spirit to some of the tests we consider. In a related work \cite{gheissari2017concentration} demonstrated that sums of pairwise correlations concentrate for general Ising models. Pseudo-likelihood parameter estimation and inference of the inverse temperature for Ising models of given structures was studied by \cite{bhattacharya2015inference}. Property testing is a fundamentally different problem, and our work is in part inspired by  \cite{neykov2016combinatorial}. We show that the graph packing constructions introduced by \cite{neykov2016combinatorial} for Gaussian models, can also be used to give upper bounds on the temperature for property testing in Ising models (Theorem \ref{ising:single:edge}). Unlike \cite{neykov2016combinatorial} however, we do not restrict our study to graphs of bounded degree, and we give a more complete picture of the complicated landscape of property testing in Ising models, by distinguishing ferromagnetic from general models (see Theorems \ref{biclques:theorem} and \ref{scaling:theorem}, and Propositions \ref{simple:lower:bound:monotone} and \ref{connectivity:low:temp}). 

Structure learning is very relevant to property testing. Restricted to the class of ferromagnetic models, \cite{tandon2014information} related structural conditions of the graph with information-theoretic bounds. \cite{santhanam2012information} suggested correlation decoders, which are computationally inefficient but to the best of our knowledge have the smallest sample size requirements for general models. \cite{anandkumar2012high} gave a polynomial time neighborhood selection method for models whose graphs obey special properties. The first polynomial time algorithm which works for general Ising models was given by \cite{bresler2015efficiently} and was motivated by earlier works on structure recovery \citep{bresler2008reconstruction, bresler2014structure}. Inspired by the simplicity of the correlation algorithms studied by \cite{montanari2009graphical, santhanam2012information} we use similar ideas to develop property tests, and demonstrate that for some properties our tests work in vastly different regimes compared to graph recovery. 

\subsection{Notation}\label{notation:section}

For convenience of the reader we summarize the notation used throughout the paper. For a vector $\vb = (v_1, \ldots, [d])^T\in \RR^d$, let $\|\vb\|_q = (\sum_{i = 1}^d v_i^q)^{1/q},  1 \leq q < \infty$ with the usual extension for $q = \infty$: $\|\vb\|_{\infty} = \max_{i} |v_i|$. Moreover, for a matrix $\Ab \in \RR^{d \times d}$ we denote $\|\Ab\|_p = \max_{\|\vb\|_p = 1} \|\Ab \vb\|_p$ for $p \geq 1$. For any $n \in \NN$ we use the shorthand $[n] = \{1,\ldots, n\}$. We denote $\NN_0 = \NN \cup \{0\}$. For a set $N \subset \NN$ we define ${N \choose 2} = \{(u,v) ~|~ u<v, ~ u,v \in N\}$ to be the set of ordered pairs of numbers in $N$. 
For a graph $G = (V,E)$ we use $V(G) = V$, $E(G) = E$, $\maxdeg(G)$ to refer to the vertex set, edge set and maximum degree of $G$ respectively. For two graphs $G, G'$ we use $G' \trianglelefteq G$ if $G'$ is a \textit{spanning subgraph} of $G$, i.e., $V(G') = V(G)$ and $E(G') \subseteq E(G)$; we use $G' \subseteq G$ if $G'$ is a subgraph of $G$ but not necessarily a spanning one, i.e., $V(G') \subseteq V(G)$ and $E(G') \subseteq E(G)$. For a graph $G = (V,E)$ and an edge $e$ will write $e \in G$, $e \in E$ or $e \in E(G)$ interchangeably whenver this does not cause confusion.

For a probability measure $\PP$, the notation $\PP^{\otimes n}$ means the product measure of $n$ independent and identically distributed (i.i.d.) samples from $\PP$.
For two functions $f(x)$ and $g(x)$, we use the notation $f(x) \approx g(x)$ in the sense that $\lim_{x \downarrow 0} \frac{f(x)}{g(x)} = 1$.
Given two sequences $\{a_n\}, \{b_n\}$ we write $a_n = O(b_n)$ if for large enough $n$ there exists a constant $C < \infty$ such that $a_n \leq C b_n$; $a_n = \Omega(b_n)$ if there exists a positive constant $c > 0$ such that $a_n \geq c b_n$; $a_n = o(b_n)$ if $a_n/b_n \rightarrow 0$, and $a_n \asymp b_n$ if there exists positive constants $c$ and $C$ such that $c < a_n/b_n < C$; $a_n \gtrsim b_n$ if there exists an absolute constant $c > 0$ so that $a_n \geq c b_n$. Finally we use $\wedge$ and $\vee$ for $\min$ and $\max$ of two numbers respectively. For positive sequences $a_n$ and $b_n$ we denote $a_n \gg b_n$ if $b_n/a_n = o(1)$.

\subsection{Organization}

The remainder of the paper is structured as follows. Minimax bounds for ferromagnetic models are given in Section \ref{generic:lower:bounds:sec}. Section \ref{ferromagnetic:algorithms} is dedicated to correlation screening algorithms for  testing in ferromagnets. Section \ref{antiferromagnetic:bounds} provides minimax bounds for general models and studies correlation based algorithms for general models. The proofs of two results on strongly monotone properties --- Theorem \ref{biclques:theorem} and Proposition \ref{simple:lower:bound:monotone}, are given in Section \ref{proofs:from:section:2}. Discussion is postponed to the final Section \ref{discussion:sec}. Most proofs are relegated to the appendices. 

\section{Bounds for Ferromagnets}\label{generic:lower:bounds:sec}

This section discusses lower and upper bounds on the temperature for ferromagnetic models. We begin by formally introducing the \textit{simple} \textit{zero-field} \textit{ferromagnetic} Ising models. Given a $\theta \geq 0$, the simple zero-field ferromagnetic Ising model with signal $\theta$ is given by
\begin{align}\label{ising:model:measure}
\PP_{\theta, G}(\bX = \xb) =  \frac{1}{Z_{\theta, G}} \exp\Big( \theta \sum_{(u,v) \in E(G)}  x_u x_v\Big),
\end{align}
where the vector of spins $\xb \in \{\pm 1\}^d$ and
$$
Z_{\theta, G} = \sum_{\xb \in \{\pm 1\}^d} \exp\Big(\theta \sum_{(u,v) \in E(G)} x_u x_v\Big),
$$
denotes the normalizing constant, also known as \textit{partition function}. Model (\ref{ising:model:measure}) is equivalent to (\ref{general:ising}), where the spin-spin interactions $w_{uv}$ are either equal to $0$ or $1$; hence the term ``simple''. The term ``zero-field'' refers to the fact that all ``main-effects'' parameters of the spins $x_u$ have been set to zero, and ``ferromagnetic'' refers to the fact that all spin-spin interactions are non-negative. As discussed in the introduction, the parameter $\theta$ is the inverse temperature but will also be referred to as \textit{signal strength} interchangeably. 

\subsection{General Results}

A key concept allowing us to quantify the difficulty of testing a graph property $\cP$ under the worst possible scenario is the minimax risk. Formally, given data generated from model (\ref{ising:model:measure}) and a property $\cP$, testing (\ref{property:testing:intro:def}) is equivalent to testing $\Hb_0: G \in \cG_0(\cP)$ versus $\Hb_1: G \in \cG_1(\cP)$ where 
\begin{align}\label{null:alternative:def}
\cG_0(\cP) := \{G \in \cG_d ~|~ \cP(G) = 0 \}, ~~~~ \cG_1(\cP) := \{G \in \cG_d ~|~ \cP(G) = 1 \}. 
\end{align}
The minimax risk of testing $\cP$ is defined as
\begin{align}\label{testing:risk:def}
R_n(\cP,\theta) :=\inf_{\psi}\bigr[ \sup_{G \in \cG_0(\cP)} \PP^{\otimes n}_{\theta, G}(\psi = 1) + \sup_{G \in \cG_1(\cP)} \PP^{\otimes n}_{\theta, G}(\psi = 0)\bigr],
\end{align}
where the infimum is taken over all measurable binary valued test functions $\psi$, and recall the notation $^{\otimes n}$ for a product measure of $n$ i.i.d. observations. Criteria (\ref{testing:risk:def}) evaluates the sum of the worst possible type I and type II errors under the best possible test function $\psi$. One can always generate $\psi \sim Ber(\frac{1}{2})$ independently of the data, which yields a minimax risk equal to $1$. In the remainder of this section we derive upper and lower bounds on the temperature beyond which $R_n(\cP,\theta)$ asymptotically equals $1$, which implies that asymptotically the best test of $\cP$ would be as good as a random guess. Importantly, here and throughout the manuscript we implicitly assume the high dimensional regime $d := d(n)$, so that asymptotically $d \rightarrow \infty$ as  $n\rightarrow\infty$.

To formalize our general signal strength bound for combinatorial properties in Ising models, we need several definitions. Similar definitions were previously used by \cite{neykov2016combinatorial} to understand the limitations of combinatorial inference in Gaussian graphical models. The first definition allows us to measure a graph based pre-distance between edges.

\begin{definition}[Edge Geodesic Pre-distance] Let $G$ be a graph and $\{e, e'\}$ be a pair of edges which need not belong to $G$. The edge geodesic pre-distance is given by
$$
d_G(e, e') := \min_{u \in e, v \in e'} d_{G}(u, v),
$$
where $d_{G}(u, v)$ denotes the geodesic distance\footnotemark\, between vertices $u$ and $v$ on $G$. If such a path does not exist $d_G(e,e') = \infty$.
\end{definition}
\footnotetext{The geodesic distance between $u$ and $v$ is the number of edges on the shortest path connecting $u$ and $v$.}
Here we use the term \textit{pre-distance} since $d_G(e,e')$ does not obey the triangle inequality. Having defined a pre-distance we can define edge packing sets and packing numbers.

\begin{definition}[Packing Number]\label{packing:number} Given a graph $G = (V,E)$ and a collection of edges $\cC$ with vertices in $V$, an $r$-packing of $\cC$ is any subset of edges $S$, i.e., $ S \subseteq \cC$ such that each pair of edges $e, e' \in S$ satisfy $d_{G}(e,e') \geq r$. We define the $r$-packing number:
$$
N(\cC, d_G, r) = \max \{|S| ~|~ S \subseteq \cC, ~ S \mbox{ is } r\mbox{-packing}\}, \mbox{i.e.,}
$$
$N(\cC, d_G, r)$ is the maximum cardinality of an $r$-packing set.
\end{definition}
A large $r$-packing number implies that the set $\cC$ has a large collection of edges that are far away from each other. Hence the packing number can be understood as a complexity measure of an edge set. The final definition before we state our first result formalizes constructions of graphs belonging to the null and alternative hypothesis and differing in a single edge. 

\begin{definition}[Null-Alternative Divider]\label{single:edge:null:alt:sep} For a binary graph property $\cP$, let $G_0 = ([d], E_0) \in \cG_0(\cP)$. We refer to an edge set
$$
\cC = \{e_1, \ldots, e_m\},
$$ 
as a \textit{null-alternative divider} (or simply \textit{divider} for short) with a null base $G_0$ if for any $e \in \cC$ the graphs $G_e := ([d], E_0 \cup \{e\}) \in \cG_1(\cP)$.
\end{definition}
Intuitively, a large divider set $\cC$ implies that testing $\cP$ is difficult since there exist multiple graphs $G_e$ with which one can confuse the graph $G_0$. We make this intuition precise in
\begin{theorem}[Signal Strength General Lower Bound]\label{ising:single:edge} Given a binary graph property $\cP$, let $G_0 \in \cG_0(\cP)$, and the set $\cC$ be a divider set with a null base $G_0$. Suppose that $|\cC| \rightarrow \infty$ asymptotically. If we have
$$
\textstyle \theta \leq \frac{1}{2}\sqrt{\frac{\log N(\cC, d_{G_0}, \log \log |\cC|)}{n}} \wedge \atanh\bigr(\frac{e^{-2}}{\maxdeg(G_0) + 1}\bigr),
$$
then $\liminf_{n \rightarrow \infty} R_n(\cP, \theta) = 1.$
\end{theorem}

Theorem \ref{ising:single:edge} gives a strategy for obtaining lower bounds on $\theta$ using purely combinatorial constructions. Its proof utilizes Dobrushin's comparison theorem \citep{follmer1988random} to bound the $\chi^2$ divergence between Ising measures deferring in a single edge. The second inequality on $\theta$ is required to ensure that the system is in a ``high temperature regime'' which is where Dobrushin's theorem holds. If one can select a graph $G_0$ of constant maximum degree, the real obstruction on $\theta$ will be given by the entropy term. Theorem \ref{ising:single:edge} is reminiscent of Theorem 2.1 of \cite{neykov2016combinatorial}; remarkably, similar constructions can be used to give lower bounds on the signal strength in both the Gaussian and Ising models. Even though the statements of the two results are related, their proofs are vastly different. The proof in the Gaussian case heavily relies on the fact that the partition functions can be evaluated in closed form, which is generally impossible in Ising models. We demonstrate the usefulness of Theorem \ref{ising:single:edge} in Section \ref{lower:bound:examples} where we apply it to a connectivity testing example. 

We complement Theorem \ref{ising:single:edge} by an upper bound on the inverse temperature $\theta$ above which the minimax risk cannot be controlled. The need for such bounds arises due to identifiability issues in Ising models at low temperatures. In such regimes the model develops long range correlations, i.e., even spins which are not neighbors on the graph can become highly correlated. A simple implication of this fact for instance is that it is challenging to tell apart a triangle graph from a vertex with its two disconnected neighbors at low temperatures (see Figure \ref{biclique:example:cycle}). To formalize the statement we first define a class of graph properties. To this end recall the distinction between the spanning subgraph and subgraph inclusions $\trianglelefteq, \subseteq$ introduced in Section \ref{notation:section}.

\begin{definition}[Monotone and Strongly Monotone Properties] A binary graph property $\cP : \cG_d \mapsto \{0,1\}$, is called \textit{monotone} if for any two graphs $G' \trianglelefteq G$ we have $\cP(G') \leq \cP(G)$. A binary property $\cP$ is called \textit{strongly monotone} if for any two graphs $G' \subseteq G$ we have $\cP(G') \leq \cP(G)$.
\end{definition}
By definition any strongly monotone property is monotone, however the converse is not true. An example of a strongly monotone property is the size of the largest clique in a graph. On the other hand, an example of a monotone property which is not strongly monotone is graph connectivity. We now state our result giving an upper bound on $\theta$ when testing strongly monotone properties. We have

\begin{theorem}[Strongly Monotone Properties Upper Bound]\label{biclques:theorem} Let $\cP$ be a strongly monotone property, and $H_0 \in \cG_0(\cP)$. Assume there exists an $l \times r$ biclique\footnote{A complete bipartite graph.} $B$ with $r \geq 2$ such that $B \trianglelefteq H_0$. Suppose there are two vertices $u,v$ belonging to the right side of $B$, so that adding $(u,v)$ to $H_0$ gives a graph $H_1 \in \cG_1(\cP)$. Let $\theta$ satisfy $\theta \geq \frac{2}{l}$ and $\theta \geq \frac{3}{r - 2}$ when $r > 2$ or $\theta \geq \log 2$ for $r = 2$. Then if for some $\kappa > 1$ we have
\begin{align}\label{theta:scaling:upper:bound:max:monotone}
\theta \geq \frac{\log \frac{2\kappa n r}{\log \lfloor d/(l + r) \rfloor}}{l},
\end{align}
it holds that $\liminf_{n \rightarrow \infty} R_n(\cP, \theta) = 1$.
\end{theorem}

Theorem \ref{biclques:theorem} shows how to prove upper bounds on $\theta$ using graph constructions. One needs to find a graph $H_0$ containing a large biclique $B$, so that adding edges to $H_0$ transfers it to an alternative graph. The number of ``left'' vertices $l$ of $B$ appears in (\ref{theta:scaling:upper:bound:max:monotone}), and therefore the larger $B$ is the harder it is to test $\cP$ in the worst case. The intuition behind this is as follows. The existence of the biclique $B$ is a measure of how dense $H_0$ is. The denser $H_0$ is the harder it is to tell it apart from $H_1$ when $\theta$ is large. On the other hand the strong monotonicity of $\cP$ ensures that if a subgraph $H_1$ of $G$ satisfies $\cP(H_1) = 1$ then $\cP(G) = 1$. Therefore if $G$ contains $H_0$ as a subgraph it becomes hard to test for $\cP$ when the value of $\theta$ is large. 

We end this section with a result, which shows a simple lower bound on $\theta$ for strongly monotone properties. One may use this result in place of Theorem \ref{ising:single:edge}, when handling strongly monotone properties. 

\begin{proposition}[Strongly Monotone Properties Lower Bound]\label{simple:lower:bound:monotone} Let $\cP$ be a strongly monotone property, and the graph $H_0 = ([m], E_0) \in \cG_0(\cP)$, be such that if one adds the edge $e$ to $H_0$ the resulting graph $H_1 = ([m], E_0 \cup \{e\}) \in \cG_1(\cP)$. Suppose $\log \lfloor d/m \rfloor \leq n$. Then if 
\begin{align}\label{new:condition:prop:27}
\theta < \atanh \bigg(\sqrt{\frac{\log \lfloor d/m \rfloor}{n}}\bigg),
\end{align}
we have $\liminf_{n \rightarrow \infty} R_n(\cP, \theta) = 1$. Furthermore $\liminf_{n \rightarrow \infty} R_n(\cP, \theta) = 1$ if $\log \lfloor d/m \rfloor \gtrsim n$ for a sufficiently large constant. 
\end{proposition}

Notice that for positive $\theta$ one has $\theta > \tanh(\theta)$, and therefore \eqref{new:condition:prop:27} implies that $\theta > \textstyle \sqrt{\frac{\log \lfloor d/m \rfloor}{n}}$ in order for cycle testing to be possible. Examples \ref{cycle:pres:ex} and \ref{cliq:size:ex} of the following section illustrate how to apply Theorem \ref{biclques:theorem} and Proposition \ref{simple:lower:bound:monotone} in practice. 

\subsection{Examples} \label{lower:bound:examples}

In this section we apply Theorems \ref{ising:single:edge}, \ref{biclques:theorem} and Proposition \ref{simple:lower:bound:monotone} to establish necessary conditions on $\theta$ for the three examples discussed in the introduction. In the first example we derive a lower bound on $\theta$ for graph connectivity testing. 

\begin{example}[Connectivity]\label{conn:example} Define ``graph connectivity'' $\cP$ as $\cP(G) = 0$ if $G$ is disconnected and $\cP(G) = 1$ otherwise. Then if 
\begin{align}\label{connectivity:lower:bound}
\theta < \kappa \sqrt{\log d/n} \wedge \atanh(1/(3e^2)),
\end{align} we have $\liminf_{n \rightarrow \infty} R_n(\cP, \theta) = 1$. Furthermore, if $\log d \gtrsim n$ for a sufficiently large absolute constant and if
\begin{align}\label{theta:bound:logd:big}
\tanh(\theta) < 1,
\end{align}
we have $\liminf_{n \rightarrow \infty} R_n(\cP, \theta) = 1$.
\end{example}

\begin{proof}[Proof of Example \ref{conn:example}]
Note that since connectivity is not a strongly monotone property we cannot apply Proposition \ref{simple:lower:bound:monotone}, and will use Theorem \ref{ising:single:edge} instead. Construct a base graph $G_0 := ([d], E_0)$ where 
$$E_0 := \{(j, j+1)_{j = 1}^{\lfloor d/2\rfloor - 1}, (\lfloor d/2\rfloor,1), (j, j +1)_{j = \lfloor d/2\rfloor + 1}^{d}, (\lfloor d/2\rfloor + 1, d)\},$$
and let 
$$\cC := \{(j, \lfloor d/2\rfloor + j)_{j = 1}^{\lfloor d/2\rfloor}\} \mbox{ (see Figure \ref{conn:graph:ex2}).}$$
Adding any edge from $\cC$ to $G_0$ results in a connected graph, so $\cC$ is a divider with a null base $G_0$. To construct a packing set of $\cC$, we collect all edges $(j, \lfloor d/2\rfloor + j)$ for $j \leq \lfloor d/2\rfloor - \lceil \log \log |\cC|\rceil$ satisfying $\lceil \log \log |\cC|\rceil$ divides $j$. This procedure results in a packing set with radius at least $\lceil \log \log |\cC|\rceil $ and cardinality of at least $\Big \lfloor \frac{|\cC|}{\lceil \log \log |\cC|\rceil}\Big\rfloor -1 $. Therefore 
$$\log N(\cC, d_{G_0}, \log \log |\cC|) \geq \log \Big[\Big \lfloor \frac{|\cC|}{\lceil \log \log |\cC|\rceil}\Big\rfloor -1\Big] \asymp \log |\cC| \asymp \log d.$$ 
By Theorem \ref{ising:single:edge} we conclude that the asymptotic risk of connectivity testing is $1$ if for some absolute constant $\kappa > 0$ we have that \eqref{connectivity:lower:bound} holds. The second conclusion of this example does not follow directly from our general results. Its proof is deferred to Appendix \ref{bounds:ferromagnets}. 

\begin{figure}[H] 
\centering
\begin{tikzpicture}[scale=.5]
\SetVertexNormal[Shape  = circle,
                  MinSize = 11pt,
                  FillColor  = white,
                  InnerSep=0pt,
                  LineWidth = .5pt]
   \SetVertexNoLabel
   \tikzset{EdgeStyle/.style= {thin,
                                double          = black,
                                double distance = .5pt}}
                                \begin{scope}[rotate=90]
                                \grEmptyCycle[prefix=a,RA=3]{5}{1}%
                                \begin{scope}\grEmptyCycle[prefix=b,RA=1.5]{5}{1}\end{scope} \end{scope}
       \tikzset{EdgeStyle/.style= {dashed,thin,
                                double          = black,
                                double distance = .5pt}}
                                                                       \tikzset{LabelStyle/.style = {below, fill = white, text = black, fill opacity=0, text opacity = 1}}
                                    \Edge[label=$e\protect\vphantom{'}$](a1)(b1)
    \Edge[label=$e'$](a4)(b4)
    \Edge(a2)(b2)
    \Edge(a3)(b3)
    \Edge(a0)(b0)
            \tikzset{EdgeStyle/.style= {thin,
                                double          = black,
                                double distance = .5pt}}
     \Edge(b0)(b1)
     \Edge(b1)(b2)
     \Edge(b2)(b3)
     \Edge(b3)(b4)
     \Edge(b0)(b4)
     \Edge(a0)(a1)
     \Edge(a1)(a2)
     \Edge(a2)(a3)
     \Edge(a3)(a4)
     \Edge(a0)(a4)
\end{tikzpicture}
\caption{The graph $G_0$ with two edges $e,e' \in \cC: d_{G_0}(e, e') = 2$, $d = 10$.}\label{conn:graph:ex2}
\end{figure}
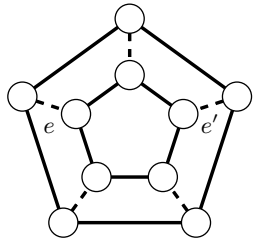

\end{proof}

\begin{example}[Cycle Presence]\label{cycle:pres:ex} Consider testing the property $\cP$ ``cycle presence'', i.e. $\cP(G) = 0$ if $G$ is a forest and $\cP(G) = 1$ otherwise. Suppose $\log \lfloor d/3 \rfloor \leq n$. If either
\vskip -.4cm
\noindent\begin{minipage}{0.6\linewidth}
\begin{align}\label{lower:cycle:testing}
\theta < \atanh(\sqrt{\log \lfloor d/3\rfloor /n}) \mbox{  	      or }
\end{align}
\end{minipage}%
\begin{minipage}{0.4\linewidth}
\begin{align}\label{upper:bound:cycle:testing}
\theta \geq 2 \vee \log\frac{4\kappa n}{\log \lfloor d/3 \rfloor},
\end{align}
\end{minipage}
for some absolute constant $\kappa > 0$, we have $\liminf_{n \rightarrow \infty} R_n(\cP, \theta ) = 1$. Furthermore, $\liminf_{n \rightarrow \infty} R_n(\cP, \theta ) = 1$ if for some sufficiently large constant we have $\log d \gtrsim n$.

\end{example}

\begin{proof}[Proof of Example \ref{cycle:pres:ex}]

By definition cycle presence is a strongly monotone property.  Figure \ref{biclique:example:cycle} shows an example of a graph $H_0$ satisfying the conditions of Theorem \ref{biclques:theorem} and Proposition \ref{simple:lower:bound:monotone}. 
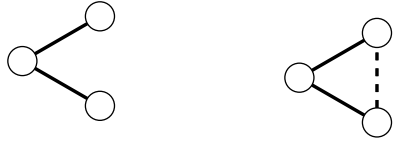
\begin{figure}[H] 
\centering
\begin{tikzpicture}[scale=.55]
\SetVertexNormal[Shape  = circle,
                  FillColor  = white,
                  MinSize = 11pt,
                  InnerSep=0pt,
                  LineWidth = .5pt]
   \SetVertexNoLabel
   \tikzset{EdgeStyle/.style= {thin,
                                double          = black,
                                double distance = .5pt}}
                                \begin{scope}[rotate=60, shift={(0cm, 0cm)}]
                                \grEmptyCycle[prefix=a,RA=1.25]{3}{1}\end{scope}
                                \begin{scope}[rotate=60, shift={(-3cm, 6cm)}]\grEmptyCycle[prefix=b,RA=1.25]{3}{1}\end{scope} 
                                                                    \Edge(a1)(a2)
                                                                    \Edge(a1)(a0)
                                                                    \Edge(b1)(b2)
                                                                    \Edge(b1)(b0)

    \tikzset{EdgeStyle/.style= {dashed,thin,
                                double          = black,
                                double distance = .5pt}}
\Edge(a2)(a0)
    \tikzset{LabelStyle/.style = {below, fill = white, text = black, fill opacity=0, text opacity = 1}}
\end{tikzpicture}
\caption{The graph $H_0$ in the left panel is a triangle with a missing edge.  $H_0$ contains a biclique with no cycles. On the other hand, if we add the dashed edge we obtain a triangle graph $H_1$, which has a cycle. In terms of the notation of Theorem \ref{biclques:theorem} we have $l = 1$ and $r = 2$. }\label{biclique:example:cycle}
\end{figure}
Concretely, $H_0$ is a $1 \times 2$ biclique which contains no cycle and has the property that adding one edge on its right side gives a graph with a cycle. By Proposition \ref{simple:lower:bound:monotone} we immediately confirm \eqref{lower:cycle:testing} and the final conclusion.  Furthermore, by a direct application of Theorem \ref{biclques:theorem} it follows that if there exists a constant $\kappa > 1$ such that if \eqref{upper:bound:cycle:testing} holds the minimax risk of cycle testing is asymptotically $1$.

\end{proof}

\begin{example}[Clique Size]\label{cliq:size:ex} In our final example we consider testing the ``maximum clique size'' property $\cP$, where $\cP$ is such that $\cP(G) = 0$ if $G$ has no $m$-clique and $\cP(G) = 1$ otherwise. Suppose that the maximum degree of $G$ satisfies $\maxdeg(G) \leq s$ where $s$ is a known integer such that $m \leq s + 1$. Let $\log \lfloor d/m \rfloor \leq n$. If either
\vskip -.4 cm
\noindent\begin{minipage}{0.5\linewidth}
\begin{align}\label{clique:lower:bound}
\theta < \atanh( \sqrt{\log \lfloor d/m \rfloor/n}), \mbox{  	    or }
\end{align}
\end{minipage}%
\begin{minipage}{0.5\linewidth}
\begin{align}\label{reader:friendly:upper:bound:clique}
 \theta \gtrsim \frac{12}{s-9} \vee \frac{\log\frac{\kappa n s}{\log \lfloor 2 d/s \rfloor}}{(s-1)/4},
\end{align}
\end{minipage}
for some absolute constant $\kappa > 0$, we have $\liminf_{n \rightarrow \infty} R_n(\cP, \theta ) = 1$. Furthermore $\liminf_{n \rightarrow \infty} R_n(\cP, \theta) = 1$ if $\log \lfloor d/m \rfloor \gtrsim n$ for a sufficiently large constant. 

  \end{example}

\begin{proof}[Proof of Example \ref{cliq:size:ex}]

Since $\cP$ is a strongly monotone property, we can apply Theorem \ref{biclques:theorem} and Proposition \ref{simple:lower:bound:monotone} to upper and lower bound $\theta$ respectively. We start first with the lower bound. Construct $H_0$ as an $m$-clique with a missing edge, as shown in Figure \ref{clique:example}. By Proposition \ref{simple:lower:bound:monotone} we immediately deduce \eqref{clique:lower:bound} and the final conclusion of the statement.

 \begin{figure}[H] 
\centering
\begin{tikzpicture}[scale=.55]
\SetVertexNormal[Shape  = circle,
                  FillColor  = white,
                  MinSize = 11pt,
                  InnerSep=0pt,
                  LineWidth = .5pt]
   \SetVertexNoLabel
   \tikzset{EdgeStyle/.style= {thin,
                                double          = black,
                                double distance = .5pt}}
                                \begin{scope}[rotate=0, shift={(0cm, 0cm)}]\grEmptyCycle[prefix=a,RA=1.5]{4}\end{scope}
                              \begin{scope}[rotate=0, shift={(6cm, 0cm)}]\grEmptyCycle[prefix=b,RA=1.5]{4}\end{scope}
                                \Edge(a1)(a2)
                                \Edge(a2)(a0)
                                \Edge(a3)(a1)
                                \Edge(a3)(a2)
                                \Edge(a3)(a0)

                                \Edge(b1)(b2)
                                \Edge(b2)(b0)
                                \Edge(b3)(b1)
                                \Edge(b3)(b2)
                                \Edge(b3)(b0)
    \tikzset{EdgeStyle/.style= {dashed,thin,
                                double          = black,
                                double distance = .5pt}}
                                                                       \tikzset{LabelStyle/.style = {above right, fill = white, text = black, fill opacity=0, text opacity = 1}}

                                                                    \Edge(b1)(b0)
                                                                    
    \tikzset{LabelStyle/.style = {below, fill = white, text = black, fill opacity=0, text opacity = 1}}
\end{tikzpicture}
\caption{For this figure let $m = 4$. In the left panel we show an example of a graph $H_0$, while on the right panel we add one edge to transfer $H_0$ to $H_1$ which satisfies the property $\cP$.}\label{clique:example}
\end{figure}
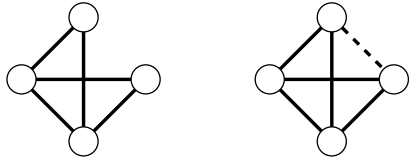

The following construction of the graph $H_0$ from the statement of Theorem \ref{biclques:theorem} is inspired by Tur\'{a}n's Theorem \citep[e.g.,][]{bollobas2004extremal}. We build $H_0$ by taking $\lfloor\frac{s-1}{m-2}\rfloor (m-1) + 1$ vertices, splitting them in $m-1$ approximately equally sized groups (one group will have 1 more vertex than the others) and connecting any two vertices belonging to different groups; see Figure \ref{clique:example2} for a visualization of $H_0$. It is simple to check that $H_0$ does not contain an $m$-clique, and adding certain edges to $H_0$, gives a graph containing an $m$-clique with maximum degree bounded by $s$.  Furthermore, $H_0$ contains a $\lfloor\frac{s-1}{m-2}\rfloor \lfloor\frac{m-1}{2}\rfloor \times \bigr(\lfloor\frac{s-1}{m-2}\rfloor \lceil\frac{m-1}{2}\rceil + 1\bigr)$ biclique; to see this split the $m-1$ vertex groups into $2$ vertex groups: one with all vertices in the first $\lfloor\frac{m-1}{2}\rfloor$ groups, and the other with all remaining vertices. 

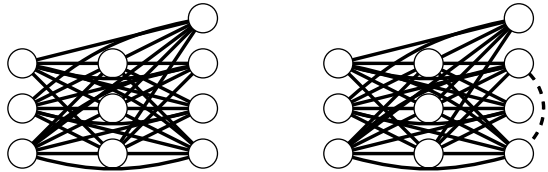
\begin{figure}[H]
\centering
\begin{tikzpicture}[scale=.6]
\SetVertexNormal[Shape      = circle,
                  FillColor  = white,
                  MinSize = 11pt,
                  InnerSep=0pt,
                  LineWidth = .5pt]
   \SetVertexNoLabel
   \tikzset{LabelStyle/.style = {below, fill = white, text = black, fill opacity=0, text opacity = 1}}
   \tikzset{EdgeStyle/.style= {thin,
                                double          = black,
                                double distance = .5pt}}
    \begin{scope}\grEmptyPath[prefix=a,RA=2]{3}\end{scope}
    \begin{scope}[shift={(0,1)}]\grEmptyPath[prefix=b,RA=2]{3}\end{scope}
    \begin{scope}[shift={(0,2)}]\grEmptyPath[prefix=c,RA=2]{3}\end{scope}
    
    \begin{scope}[shift={(7,0)}]\grEmptyPath[prefix=d,RA=2]{3}\end{scope}
    \begin{scope}[shift={(7,1)}]\grEmptyPath[prefix=e,RA=2]{3}\end{scope}
    \begin{scope}[shift={(7,2)}]\grEmptyPath[prefix=f,RA=2]{3}\end{scope}
        \begin{scope}[shift={(4,3)}]\grEmptyPath[prefix=g,RA=2]{1}\end{scope}
        \begin{scope}[shift={(11,3)}]\grEmptyPath[prefix=h,RA=2]{1}\end{scope}

\Edge(a0)(a1)
\Edge(a0)(b1)
\Edge(a0)(c1)
\Edge(a0)(b2)

\Edge(g0)(a1)
\Edge(g0)(b1)
\Edge(g0)(c1)
\Edge(g0)(a0)
\Edge(g0)(c0)

\Edge(h0)(d1)
\Edge(h0)(e1)
\Edge(h0)(f1)
\Edge(h0)(d0)
\Edge(h0)(f0)

\Edge(b0)(a1)
\Edge(b0)(b1)
\Edge(b0)(c1)
\Edge(b0)(c2)
\Edge(b0)(a2)

\Edge(c0)(c1)
\Edge(c0)(a1)
\Edge(c0)(b1)
\Edge(c0)(b2)

\Edge(a2)(a1)
\Edge(a2)(b1)
\Edge(a2)(c1)

\Edge(b2)(a1)
\Edge(b2)(b1)
\Edge(b2)(c1)

\Edge(c2)(c1)
\Edge(c2)(a1)
\Edge(c2)(b1)

\Edge(d0)(d1)
\Edge(d0)(e1)
\Edge(d0)(f1)
\Edge(d0)(e2)

\Edge(e0)(d1)
\Edge(e0)(e1)
\Edge(e0)(f1)
\Edge(e0)(f2)
\Edge(e0)(d2)

\Edge(f0)(f1)
\Edge(f0)(d1)
\Edge(f0)(e1)
\Edge(f0)(e2)

\Edge(d2)(d1)
\Edge(d2)(e1)
\Edge(d2)(f1)

\Edge(e2)(d1)
\Edge(e2)(e1)
\Edge(e2)(f1)

\Edge(f2)(f1)
\Edge(f2)(d1)
\Edge(f2)(e1)

\tikzset{EdgeStyle/.append style = {thin, bend right=15}}

\Edge(a0)(a2)
\Edge(c0)(c2)
\Edge(b0)(b2)

\Edge(d0)(d2)
\Edge(f0)(f2)
\Edge(e0)(e2)
\Edge(g0)(b0)
\Edge(h0)(e0)

\tikzset{EdgeStyle/.append style = {thin, bend left=15}}
\Edge(a0)(c2)
\Edge(c0)(a2)
\Edge(d0)(f2)
\Edge(f0)(d2)

\tikzset{EdgeStyle/.append style = {dashed, thin,bend left=45}}
\Edge(f2)(d2)
\end{tikzpicture}

\caption{In the left panel we show an example of the graph $H_0$. The concrete values of $s$ and $m$ used are $s=7$ and $m = 4$. $H_0$ contains no $4$-clique, and its maximum degree is $7$. On the other hand adding any edge on the rightmost side (such as the dashed edge in the figure on the right) to $H_0$ results in a graph $H_1$ which contains a $4$-clique, and whose maximum degree remains bounded by $7$. The graph $H_0$ contains a $3 \times 7$ biclique, whose left side consists of taking the three leftmost vertices.} \label{clique:example2}
\end{figure}

We are now in a position to apply Theorem \ref{biclques:theorem}. To render bound (\ref{theta:scaling:upper:bound:max:monotone}) in a reader friendly form, we use that the terms $\lfloor\frac{s-1}{m-2}\rfloor \lfloor\frac{m-1}{2}\rfloor \geq \frac{s-1}{4}$ and $\lfloor\frac{s-1}{m-2}\rfloor \lceil\frac{m-1}{2}\rceil + 1\geq \frac{s+3}{4}$. We have that the minimax risk is asymptotically $1$ if for any $\kappa > 1/2$ \eqref{reader:friendly:upper:bound:clique} holds.

\end{proof}
Notably, the maximum degree  $s$ of the graph appears in (\ref{reader:friendly:upper:bound:clique}) unlike in the previous two examples. The bigger the maximum degree is allowed to be, the smaller the signal $\theta$ has to be in order for meaningful clique size tests to exist. 

\section{Correlation Screening for Ferromagnets}\label{ferromagnetic:algorithms}

In this section we formulate and study the limitations of a greedy correlation screening algorithm on monotone property testing problems. We pay special attention to the examples discussed in Section \ref{lower:bound:examples}. Unlike correlation based decoders, such as the ones studied by \cite{santhanam2012information}, this algorithm is designed to directly target the graph property of interest, and also has polynomial runtime for many instances. Moreover, for different properties, the regimes in which the algorithm works differ vastly from graph recovery algorithms. For generality we expand model class (\ref{ising:model:measure}) to include all zero-field ferromagnetic models such that:
\begin{align}\label{ferromagnet:def}
\PP_{\theta, G_\wb}(\bX = \xb) \varpropto \exp\Big(\theta \sum_{(u,v) \in E(G)} w_{uv} x_u x_v\Big),
\end{align}
where $\xb \in \{\pm 1\}^d$, $\theta \geq 0$, $G_{\wb} := (G,{\wb})$ is a weighted graph, and for $(u,v) \in E(G)$: $w_{uv} > 0$. In contrast to (\ref{ising:model:measure}), in (\ref{ferromagnet:def}) the weights $\wb$ allow for the interactions to have different magnitude. 

\subsection{General Correlation Screening Algorithm}

We now define a class of graphs ``witnessing'' the alternative. For a monotone property $\cP$ define the collection of graphs
$$
\cW(\cP) := \{G \in \cG_d ~|~ \cP(G) = 1\}.\footnote{This is simply a re-definition of the set $\cG_1(\cP)$ from \eqref{null:alternative:def}. }
$$
We refer to graphs in $\cW(\cP)$ as \textit{witnesses} of $\cP$. It is clear by the monotonicity of $\cP$ for two graphs $G$ and $G'$ such that  $\cP(G) = 0$ and $G' \in \cW(\cP)$ that the set $E(G') \setminus E(G) \neq \varnothing$. Define the sets of weighted graphs
\begin{align*}
\cG_0(\cP) &:= \bigr\{G_{\wb}  ~\bigr|~ \wb \in \RR^{+{d \choose 2}}, \cP(G) = 0 \bigr\},\\  \cG_1(\cP)  &:= \bigr\{G_{\wb}  ~\bigr|~ \wb \in \RR^{+{d \choose 2}}, \cP(G) = 1, \max_{G' \in \cW(\cP), G' \trianglelefteq G} \min_{(u,v) \in E(G')}  w_{uv} \geq 1 \bigr\}.
\end{align*}
The set $\cG_0(\cP)$ imposes no signal strength restrictions, while $\cG_1(\cP)$ requires the existence of at least one witness of $\cP$, each edge of which corresponds to an interaction with magnitude at least $\theta$. For future reference we omit the dependency on $\cP$ if this does not cause confusion. In Section \ref{lower:bound:examples} we saw that some property tests, such as cycle testing and clique size testing, necessitate further restrictions on their parameters (see (\ref{upper:bound:cycle:testing}) and (\ref{reader:friendly:upper:bound:clique})). Let $\cR$ be an appropriately chosen for the property $\cP$ restriction set on the weighted graph pair $G_{\wb}$. For instance, an appropriate set $\cR$ for cycle testing could be $\cR = \{G_{\wb} ~|~ \| \wb\|_{\infty} \leq \Theta/\theta\}$ for some $\Theta \geq \theta$. 

To this end, it is useful to first define the extremal correlation
\begin{align}\label{min:corr:under:alt}
\cT := \cT(\cP,\cR,\theta) = \min_{G_{\wb} \in \cG_1 \cap \cR} \max_{G' \in \cW} \min_{(u,v) \in E(G')} \EE_{\theta, G_\wb} X_u X_v.
\end{align}
$\cT$ is the maximal smallest possible correlation between neighboring vertices in a witness graph given any model from the alternative. In the following we give a simple universal lower bound on $\cT$.

\begin{lemma}\label{universal:bound:cT:lemma} For any monotone property $\cP$ we have:
\begin{align*}
\cT \geq \tanh(\theta).
\end{align*}
\end{lemma}
\begin{proof}[Proof of Lemma \ref{universal:bound:cT:lemma}] 
Observe that by Griffith's inequality (see Theorem \ref{FKG:ineq} in Appendix \ref{aux:app:sec}) deleting any edge can only reduce the correlation between a pair of vertices. Therefore one can prune the graph $G$ without increasing $\cT$, until it becomes a minimal witness $W$, i.e., if we delete any edge from $W$ the resulting graph does not satisfy $\cP$. On the graph $W$, we have $\cT \geq \min_{(u,v) \in E(W)} \EE_{\theta, W_{\wb}}X_uX_v$. Next, one can prune further edges from $W$ until only the minimum edge remains. Since the correlation of a pair of vertices with a graph consisting of the single edge between them is precisely $\tanh(\theta)$ (see Lemma \ref{chain:graph} in Appendix \ref{aux:app:sec}) the inequality follows.
\end{proof}

Since in practice $\cT$ might be hard to estimate, we assume that we have a lower bound on $\cT$: $\underline{\cT}$ in closed form (we allow for $\underline \cT = \cT$). Provided that we have sufficiently many samples, and the data is generated under an alternative model, many empirical correlations between neighboring vertices should be approximately at least $\cT$ (and hence at least $\underline{\cT}$). To formally define the empirical correlations, let $\bX^{(1)}, \bX^{(2)}, \ldots, \bX^{(n)} \sim \PP_{\theta, G_{\wb}}$ be $n$ i.i.d. samples from the ferromagnetic Ising model  (\ref{ferromagnet:def}). Define the empirical measure $\hat \PP$, so that for any Borel set $A\subset \RR^d$: $\hat \PP(A) = n^{-1}\sum_{i = 1}^n\mathbbm{1}(\bX^{(i)} \in A)$. Put $\hat \EE$ for the expectation under $\hat \PP$. To this end, for a given $\delta > 0$ define the universal threshold
\begin{align}\label{tau:univ:thresh}
\tau := \tau(n,d,\delta) = \sqrt{\frac{4 \log d + \log \delta^{-1}}{n}},
\end{align}
and consider the following correlation screening Meta-Algorithm \ref{generic:property:test} for monotone property testing in ferromagnetic Ising models. 

\begin{algorithm}
\caption{Correlation Screening Test}\label{generic:property:test}
\begin{algorithmic}
\STATE \textbf{Input:} $\{\bX^{(i)}\}_{i \in [n]}, \theta, \cR, \cP$
\STATE Set $\psi = 0$
\STATE Calculate the matrix $\Mb := \{\hat \EE X_u X_v\}_{u,v \in [d]}$
\STATE Solve 
\begin{align}\label{corr:screening:witness}\hat G = \argmax_{G' \in \cW} \min_{(u,v) \in E(G')} M_{uv}\end{align} 
\STATE Set $\psi = 1$ if $\min_{e \in \hat G} M_e > \underline{\cT} - \tau$.
\RETURN{$\psi$}
\end{algorithmic}
\end{algorithm}
The only potentially computationally intensive task in Algorithm \ref{generic:property:test} is optimization (\ref{corr:screening:witness}), which aims to find a witness whose smallest empirical correlation is the largest. However, for many properties solving (\ref{corr:screening:witness}) can be done in polynomial time via greedy procedures. We remark that step (\ref{corr:screening:witness}) treats $M_{uv}$ as a surrogate of $\theta w_{uv}$. Instead, one could opt to substitute $M_{uv}$ with an estimate of the parameter $\theta w_{uv}$, which can be obtained via a procedure such as $\ell_1$-regularized vertex-wise logistic regressions \citep{ravikumar09high}, e.g. Here we prefer to focus on correlation screening due to its simplicity, while we recognize that the estimate $M_{uv}$ may not be a good proxy of $\theta w_{uv}$ in models at low temperature regimes, which are known to develop long range correlations. To this end define the extremal quantity

$$
\cQ(\cP,\cR,\theta) := \max_{G_{\wb} \in \cG_0 \cap \cR} \max_{G' \in \cW} \min_{(u,v) \in E(G')} \EE_{\theta, G_{\wb}} X_u X_v 
$$
The term $\cQ$, selects a weighted graph $G_{\wb}$ under the null and a witness $G'$, which yields the largest possible minimal correlation on any of the edges of $G'$. The following result holds regarding the performance of Algorithm \ref{generic:property:test}. 

\begin{theorem}[Correlation Screening Sufficient Conditions]\label{generic:property:test:thm} Suppose that $(\theta, n, d)$ satisfy
\begin{align}\label{abstract:corr:algo:assumption}
\underline{\cT} - \cQ > 2\tau.
\end{align}
Then Algorithm \ref{generic:property:test} satisfies
\begin{align}\label{abstract:corr:algo}
\sup_{G_{\wb} \in \cG_0\cap \cR} \PP_{\theta, G_{\wb}}(\psi = 1) \leq \delta ~~~\mbox{and}~~~  \sup_{G_{\wb} \in \cG_1 \cap \cR} \PP_{\theta, G_{\wb}}(\psi = 0) \leq \delta.
\end{align}
\end{theorem}

Condition (\ref{abstract:corr:algo:assumption}) ensures that the gap between the minimal correlations in models under the null and alternative hypothesis is sufficiently large even in worst case situations. Theorem \ref{generic:property:test:thm} is a straightforward consequence of Hoeffding's inequality, and the real difficulty when applying it is controlling the quantities $\cT$ and $\cQ$.  Recall that Lemma \ref{universal:bound:cT:lemma} showed a simple universal lower bound on $\cT$. Below we give two general upper bounds on $\cQ$. Given a sparsity level $s$ and a real number $\Theta$ define the ratio
$$
\textstyle R(s,\Theta) := \frac{\cosh(2s \Theta) + 2s e^{-2(s-1)\Theta}\cosh(2(s-1)\Theta)}{ 2s e^{-2(s-1)\Theta}\cosh(2\Theta) + 1}.
$$
The following holds

\begin{proposition}[No Edge Correlation Upper Bounds]\label{generic:property:test:cor} Assume that the graph $G_{\wb} \in \cR$, where the restriction set $\cR$ is $
\cR = \{G_{\wb}  ~|~ \maxdeg(G) \leq s, \|\wb\|_{\infty} \leq \Theta/\theta\}.$ Then the following two results hold.
\begin{itemize}
\item[i.] Let $s \geq 3$.\footnote{Similar bound on $\cQ$ holds for the case $s = 2$. For details refer to the proof.} Then 
$$\cQ \leq \frac{R(s,\Theta) - 1}{R(s,\Theta) + 1}.$$

\item[ii.] Let $(s-1) \tanh(\Theta) < 1$. Then $$\cQ \leq \frac{s \tanh^2(\Theta)}{1 - (s-1) \tanh(\Theta)}.$$
\end{itemize}
\end{proposition}
\begin{remark}\label{remark:on:corr:test} We will now argue that Proposition \ref{generic:property:test:cor} ii. and Lemma \ref{universal:bound:cT:lemma} ensure that Algorithm \ref{generic:property:test} satisfies (\ref{abstract:corr:algo}) in the high temperature regime $s \tanh(\Theta)\allowbreak \lesssim 1$ when the entries of $\wb$ are approximately equal. By Lemma \ref{universal:bound:cT:lemma} we have
$$\cT \geq \tanh(\theta).$$ 
Suppose now that $\theta = \Theta$ (equivalently $w_{uv} = 1$ for all non-zero weights). If $s \tanh(\theta) < 1/3$ and $\tanh(\theta) > 4 \tau$, by ii.
$$
\cT - \cQ \geq \tanh(\theta) - \tanh(\theta)/2 > 2 \tau.
$$
Hence when $\tanh(\theta) > 4 \tau$ we have that \eqref{abstract:corr:algo:assumption} and consequently \eqref{abstract:corr:algo} hold. More generally if $\theta \asymp \Theta$, $\tanh(\theta) \geq \Omega(\tau)$ and $s\tanh(\Theta)$ is sufficiently small, implies that Algorithm \ref{generic:property:test} controls the type I and type II errors. This fact, coupled with the results of Section \ref{generic:lower:bounds:sec} suggests that correlation screening is optimal up to scalars for many properties in the high temperature regime (where $s\tanh(\theta)$ is small). 
\end{remark}

Below we study three specific instances of Algorithm \ref{generic:property:test} to obtain better understanding of its limitations. Importantly, we observe that the correlation screening test can be constant optimal beyond the high temperature regime for some properties. 
\subsection{Examples}\label{ferromagnetic:corr:screening:examples}

We now revisit the three examples of Section \ref{lower:bound:examples}.

\begin{example}[Connectivity]

Here we implement the correlation screening algorithm for connectivity testing (see Algorithm \ref{connectivity:test}), and we take the opportunity to contrast property testing to graph recovery. We will argue that correlation screening can test graph connectivity even in graphs of unbounded degree. In contrast, correlation based algorithms fail to learn the structure even in unconnected graphs when the signal strength $\theta \geq \Omega(\frac{1}{s})$, where $s$ denotes the maximum degree of the graph, as argued by \cite{montanari2009graphical}.

\begin{algorithm}
\caption{Connectivity Test}\label{connectivity:test}
\begin{algorithmic}
\STATE \textbf{Input:} $\{\bX^{(i)}\}_{i \in [n]}$
\STATE Set $\psi = 0$
\STATE Calculate the matrix $\Mb := \{\hat \EE X_u X_v\}_{u,v \in [d]}$
\STATE Estimate $\hat T$ the maximum spanning tree (MST) on $\Mb$\footnotemark // Equivalent to solving (\ref{corr:screening:witness})
\STATE Set $\psi = 1$ if $\min_{e \in \hat T} M_e > \tanh(\theta) - \tau$
\RETURN{$\psi$}
\end{algorithmic}
\end{algorithm}
\footnotetext{Finding an MST can be done efficiently.}
It is simple to see that $\cG_0, \cG_1$ reduce to
$$
\cG_0 := \{G_{\wb} ~|~ G \mbox{ is disconnected} \},  ~~~ \cG_1 := \bigr\{G_{\wb} ~|~ \max_{\underset{\mbox{\tiny tree}}{T} \trianglelefteq G} \min_{(u,v) \in T}  w_{uv} \geq 1 \bigr \},
$$
and there are no further parameter restrictions, i.e., $\cR$ is all weighted graphs. We have

\begin{corollary}[Connectivity]\label{conn:test:consistency} Assume that $\tanh(\theta) > 2\tau$. Then Algorithm \ref{connectivity:test} satisfies (\ref{abstract:corr:algo}).
\end{corollary}

Corollary \ref{conn:test:consistency} underscores the difference between property testing and structure learning. \cite{montanari2009graphical} and \cite{santhanam2012information}, showed that one cannot recover the graph structure  in a ferromagnetic model when the parameter $\theta$ exceeds a critical threshold. We also note that the condition $\tanh(\theta) \geq 2\tau$ matches the lower bound prediction (\ref{connectivity:lower:bound}) up to constant terms when $\tau$ is sufficiently small. 

It is worth mentioning that Algorithm \ref{connectivity:test} is no longer optimal when $\log d \gtrsim n$ due to lack of concentration. If $\log d \gtrsim n$ for a sufficiently large constant, by \eqref{theta:bound:logd:big} $\tanh(\theta)$ has to equal $1$ asymptotically. It is simple to devise a test that works when $\tanh(\theta) = 1$, namely: reject the null hypothesis if all spins have the same signs through each of the $n$ trials. If the graph is connected this will happen with probability $1$; if the graph is disconnected this event happens with probability at most $1/2^n$. Finally we remark that whether one can devise a finite sample connectivity test when $\log d \gtrsim n$ and $\tanh(\theta) < 1$ remains an open question which merits further investigation.
\end{example}

\begin{example}[Cycle Presence] \label{cycle:testing:sec}
Here we revisit cycle testing. The sets $\cG_0$ and $\cG_1$ reduce to
\begin{align*}
\cG_0  := \{G_{\wb} ~|~ G \mbox{ is a forest}\}, ~~~ \cG_1 := \bigr\{G_{\wb} ~|~ \max_{\underset{\mbox{\tiny cycle}}{C} \subseteq G }\min_{(u,v) \in C} \theta_{uv} \geq 1 \bigr\}.
\end{align*}
Motivated by (\ref{upper:bound:cycle:testing}) we take the restriction set as $\cR = \{\wb ~|~ \|\wb\|_{\infty} \leq \Theta/\theta\}$. The correlation screening algorithm for cycle testing is given in Algorithm \ref{cycle:test}. We have the following corollary of Theorem \ref{generic:property:test:thm}.

\begin{algorithm}
\caption{Cycle Test}\label{cycle:test}
\begin{algorithmic}
\STATE \textbf{Input:} $\{\bX^{(i)}\}_{i \in [n]}, \theta$
\STATE Set $\psi = 0$
\STATE Calculate the matrix $\Mb := \{\hat \EE X_u X_v\}_{u,v \in [d]}$
\STATE Add edges with weights from $\Mb$ from high to low until a cycle $\hat C$ emerges\footnotemark // i.e., solve (\ref{corr:screening:witness})
\STATE Set $\psi = 1$ if $\min_{e \in \hat C} M_e > \tanh(\theta) - \tau$
\RETURN{$\psi$}
\end{algorithmic}
\end{algorithm}
\footnotetext{Finding the cycle $\hat C$ takes at most $d$ steps, and can be done efficiently.}

\begin{corollary}[Cycle Presence]\label{cycle:test:consistency} Assume that $\tanh(\theta)  - \tanh^2(\Theta)> 2\tau$. Then Algorithm \ref{cycle:test} satisfies (\ref{abstract:corr:algo}).
\end{corollary}

Below we derive a more direct result for the special case when $\theta \equiv \Theta$.

\begin{corollary}[Cycle Presence $\theta = \Theta$]\label{cycle:test:cor:2} Suppose $\theta = \Theta$. When $\tau$ is sufficiently small, if 
\begin{align}\label{cycle:testing:bounds}
\tau \lesssim \theta \lesssim \log(1/\tau),
\end{align}
Algorithm \ref{cycle:test} satisfies \eqref{abstract:corr:algo}.
\end{corollary}

\begin{proof}[Proof of Corollary \ref{cycle:test:cor:2}] In this setting, the condition of Corollary \ref{cycle:test:consistency} reduces to the quadratic inequality $t - t^2 > 2\tau$,  where we put $t := \tanh(\theta)$ for brevity. Equivalently, the feasible values of $\theta$ satisfy
$$
\textstyle \frac{1 - \sqrt{1 - 8 \tau}}{2} \leq t \leq \frac{1 + \sqrt{1 - 8 \tau}}{2}. 
$$
To make the calculation more accessible we will now use the notation $f(x) \approx g(x)$ in the sense that $\lim_{x \downarrow 0} \frac{f(x)}{g(x)} = 1$. When $\tau$ is sufficiently small, it is simple to check that $\frac{1 - \sqrt{1 - 8 \tau}}{2} \approx 2\tau$ and $\frac{1 + \sqrt{1 - 8 \tau}}{2} \approx 1 - 2\tau$. It therefore follows that Algorithm \ref{cycle:test} is successful when 
$$ 2\tau \approx \atanh(2\tau) \leq \theta \leq \atanh(1 - 2\tau) \approx \log(1/\tau)/2.$$
\end{proof}
Up to scalars \eqref{cycle:testing:bounds} agrees with bounds (\ref{lower:cycle:testing}) and (\ref{upper:bound:cycle:testing}) given in Section \ref{lower:bound:examples}. An alternative correlation based cycle test which works for general models is given in Section \ref{comp:eff:cycle:test}. 

\end{example}

\begin{example}[Clique Size] Finally we revisit clique size testing. The parameter sets reduce to
\begin{align*}
\cG_0  := \{G_{\wb} ~|~ G \mbox{ has no $m$-clique}\}, ~~~ \cG_1  := \{G_{\wb} ~|~ \max_{\underset{\mbox{\tiny $m$-clique}}{C} \subseteq G}\min_{(u,v) \in C} w_{uv} \geq 1 \},
\end{align*}
and let the restriction set $\cR = \{G_{\wb} ~|~ \|\wb\|_{\infty} \leq \Theta/\theta,  \maxdeg(G) \leq s\}$, where $2 \leq m \leq s + 1$. We summarize the correlation screening implementation for clique size testing in Algorithm \ref{clique:size:test}. To this end for a $Z \sim N(0,1)$ define 
\begin{align*}
\textstyle r(m,\theta) := \frac{e^{2\theta}\EE \cosh^{m-2}(\sqrt{\theta}Z + 2\theta)}{\EE \cosh^{m-2}(\sqrt{\theta}Z) }. 
\end{align*}
\begin{algorithm}
\caption{Clique Size Test}\label{clique:size:test}
\begin{algorithmic}
\STATE \textbf{Input:} $\{\bX^{(i)}\}_{i \in [n]}, \theta$
\STATE Set $\psi = 0$
\STATE Calculate the matrix $\Mb := \{\hat \EE X_u X_v\}_{u,v \in [d]}$
\STATE Add edges with weights from $\Mb$ from high to low until an $m$-clique $\hat C$ emerges\footnotemark // i.e., solve (\ref{corr:screening:witness})

\STATE Set $\psi = 1$ if $\min_{e \in \hat C} M_e > \frac{r(m,\theta) - 1}{r(m, \theta) + 1} - \tau$
\RETURN{$\psi$}
\end{algorithmic}
\end{algorithm}
\footnotetext{In this footnote we show an example of an algorithm for checking for $m$-clique presence. When a new edge $(u,v)$ is added, walk over common neighbors of both $u$ and $v$, and check for an $m$-clique. There are at most $\frac{ds}{2}$ steps and at each step we have to check at most $s - 1 \choose m - 2$  $m$-cliques, giving a runtime bound of $O\bigr(ds \bigr(s + m^2{s - 1 \choose m - 2}\bigr)\bigr)$.}
The following holds
\begin{corollary}[Clique Size]\label{clique:test:consistency} For $G_{\wb} \in \cG_1\cap\cR$ we have $\cT = \frac{r(m,\theta) - 1}{r(m, \theta) + 1}$. Hence if $(s,m,d,n,\theta,\Theta)$ are such that either
$$
\cT - \frac{R(s,\Theta) - 1}{R(s,\Theta) + 1} \geq 2\tau,
$$
or
$$
\cT - \frac{s \tanh^2(\Theta)}{1- (s-1)\tanh(\Theta)} \geq 2\tau, \mbox{ and } (s-1)\tanh(\Theta) < 1,
$$
Algorithm \ref{clique:size:test} satisfies (\ref{abstract:corr:algo}).
\end{corollary}

Below we derive a more direct result for the special case when $\theta \equiv \Theta$.

\begin{corollary}[Clique Size $\theta = \Theta$]\label{clique:test:consistency:theta:eq:theta} Suppose $\theta = \Theta$ and that $\tau$ and $\frac{1}{s}$ are sufficiently small. Then if 
\begin{align}\label{first:theta:bound}
\tau \lesssim \theta \lesssim \frac{1}{s},
\end{align}
Algorithm \ref{clique:size:test} satisfies (\ref{abstract:corr:algo}). Next, suppose that $m = s+1$. If $e^{s\theta} \gg s$ and
\begin{align}\label{low:temp:clique:testing}
\theta \leq \frac{\log{(2/\tau})}{4(s-1)},
\end{align}
Algorithm \ref{clique:size:test} satisfies (\ref{abstract:corr:algo}).
\end{corollary}

\begin{proof}[Proof of Corollary \ref{clique:test:consistency:theta:eq:theta}] 
In Remark \ref{remark:on:corr:test} we already argued that if $4 \tau \leq \tanh(\theta) \leq \frac{1}{3 s}$, Algorithm \ref{clique:size:test} controls the type I and type II errors. Since $\tanh(x) \approx x$ when $x$ is small inequality \eqref{first:theta:bound} follows. 
To show the second part, we resort to the first bound of Corollary \ref{clique:test:consistency}. Since the proof is more involved we show it in Remark \ref{clique:testing:remark} of Appendix \ref{corr:screening:proofs}.
\end{proof}

From \eqref{first:theta:bound} it follows that in the high temperature regime Algorithm \ref{clique:size:test} matches the lower bound (\ref{clique:lower:bound}). Furthermore note that (\ref{low:temp:clique:testing}) matches bound (\ref{reader:friendly:upper:bound:clique}) up to scalars. Hence the special case of clique size testing when $m = s + 1$ is yet another example confirming that correlation screening can be useful for property testing even at low temperatures. 

\end{example}

\section{Results for General Models}\label{antiferromagnetic:bounds}

So far we studied model classes admitting only ferromagnetic interactions. The situation drastically changes if one considers the general class of zero-field Ising models, which includes models with antiferromagnetic, i.e., negative interactions. 

\subsection{Minimax Bounds} \label{antiferromagnetic:bounds:sub}The main result of this section is an impossibility theorem, which shows that testing strongly monotone properties over the general class of models requires boundedness of a certain maximum functional of the property. We also argue more specifically, that unlike in the ferromagnetic case, connectivity testing is not feasible at low temperatures over the general case unless the degree of the graph is bounded. Both of these results sharply contrast what we have seen in the previous sections of the paper. 

Concretely, we will work with the simple zero-field models specified by the parameters $\theta > 0$ and $\wb \in \{\pm 1\}^{{d \choose 2}}$ as 
\begin{align}\label{ising:model:measure:anti}
\PP_{\theta, G_{\wb}}(\bX = \xb) \varpropto \exp\Big(\theta \sum_{(u,v) \in E(G)} w_{uv} x_u x_v\Big).
\end{align}
Expression (\ref{ising:model:measure:anti}) has more degrees of freedom compared to (\ref{ising:model:measure}) since the spin-spin interactions in (\ref{ising:model:measure:anti}) are allowed to be negative. Intuitively, interactions corresponding to $w_{uv}=-1$ have a ``repelling'' effect on the corresponding spins $u$ and $v$, whereas interactions with $w_{uv}=1$ have an ``attracting'' effect. 
 
Given a monotone property $\cP$, recall definition (\ref{null:alternative:def}) of the collections of graphs $\cG_0(\cP), \cG_1(\cP)$ (below we omit the dependence on $\cP$). Let $\cR$ be a suitable restriction set on the graph $G$. 
We redefine the minimax risk to reflect the model class expansion as follows. Let
\begin{align}\label{testing:risk:def:antifer}
\textstyle R_n(\cP,\cR,\theta) := \displaystyle \inf_{\psi}\bigr[\sup_{\wb} \sup_{G \in \cG_0 \cap \cR} \PP^{\otimes n}_{\theta,  G_{\wb}}(\psi = 1) + \sup_{\wb} \sup_{G \in \cG_1 \cap \cR} \PP^{\otimes n}_{\theta, G_{\wb}}(\psi = 0)\bigr],
\end{align}
where $\PP^{\otimes n}_{\theta,G_{\wb}}$ denotes the product measure of $n$ i.i.d. observations of (\ref{ising:model:measure:anti}) and the supremum on $\wb$ is taken over the set $\{\pm 1\}^{{d \choose 2}}$. Armed with this new definition we have
\begin{theorem}[Strongly Monotone Properties General Lower Bound]\label{scaling:theorem} Assume $G_{\wb}$ belongs to the restriction set $\cR_s = \{G_{\wb}   ~|~\allowbreak \maxdeg(G) \leq s\}$, where $s = o(\sqrt{d})$. Suppose that the strongly monotone property $\cP$ satisfies $\cP(\varnothing) = 0$ and $\cP(C_s) = 1$, where $C_s$ denotes an $s$-clique graph. Then if for some small $\varepsilon > 0$,
\begin{align}\label{sufficient:conditions:clique:detection}
\frac{s \log d/s^2}{n} > 2 + \varepsilon ~~~~~~\mbox{       and       } ~~~~~~ \frac{s \log{d/(2s)}}{n \log \sqrt{2s}} \geq 1 + \varepsilon,
\end{align}
we have $
\liminf_{n \rightarrow \infty} \inf_{\theta \geq 0} R_n(\cP, \cR_s, \theta) = 1$.
\end{theorem}
\begin{remark} 

We would like to contrast our result with similar known bounds such Theorem 1 of \cite{santhanam2012information} and Theorem 1 of \cite{bresler2008reconstruction}. The key differences between our result and these known bounds are that, first (\ref{sufficient:conditions:clique:detection}) is valid for property testing and even more generally for a certain detection problem (see the proof in Section \ref{bounds:with:ferromagnets} of the supplement for more details), while previous results are valid only for structure recovery; and therefore second --- the worst cases are very different. In fact both previously known bounds remain valid in the smaller class of ferromagnetic models, while as we saw in Section \ref{ferromagnetic:algorithms}, some strongly monotone property tests such as cycle presence do not exhibit such limitations.
\end{remark}

Note that for any non-constant strongly monotone $\cP$ one has $\cP(\varnothing) = 0$. Further, the requirement that $\cP$ holds true on $C_s$ is mild, since for any non-zero strongly monotone $\cP$ one can always find a sufficiently large $s$ for which $\cP$ is satisfied. The only true restriction of Theorem \ref{scaling:theorem} on $\cP$ is thus that one has to be able to find $s$ in the sparse regime $s \ll \sqrt{d}$.

Loosely speaking Theorem \ref{scaling:theorem} shows that when the quantity $\frac{s \log d/s}{n}$ (up to a log factor) is large, strongly monotone property testing over the model class (\ref{ising:model:measure:anti}) is very difficult in the sparse regime when $s \ll \sqrt{d}$. What is more, this statement remains valid regardless of the magnitude of the signal strength parameter $\theta \geq 0$. This contrasts sharply with our results in the ferromagnetic case, where we have already seen an example which did not require such a condition. Take the cycle testing example in Section \ref{cycle:testing:sec}. In this example if $s$ denotes the maximum degree of the graph, we can always take an $s$-clique $C_s$ (which certainly contains a cycle), and hence $s$ has to satisfy (\ref{sufficient:conditions:clique:detection}) in order for tests with reasonable minimax risk (\ref{testing:risk:def:antifer}) to exist. In contrast Corollary \ref{cycle:test:consistency} shows that controlling the minimax risk is possible without requirements on the maximum degree of the graph. Theorem \ref{scaling:theorem} shows that this is no longer the case over the broader model class (\ref{ising:model:measure:anti}). 

Theorem \ref{scaling:theorem} sheds some light on the complexity involved in testing within model class (\ref{ising:model:measure:anti}). However, it also leaves something to be desired, namely it does not quantify the effect $\theta$ has, and it does not address specific properties which may potentially exhibit different complexity. Moreover, it only applies to strongly monotone properties and not to all monotone properties, and thus in particular it does not apply to connectivity testing. Below we give an explicit upper bound on the parameter $\theta$ for connectivity testing within the model class (\ref{ising:model:measure:anti}). We show a particularly hard case for connectivity testing in Figure \ref{connectivity:antiferro:figure} and include a brief explanation in its caption. 
\begin{figure}
\centering
\begin{tikzpicture}[scale=.65]
\SetVertexNormal[Shape      = circle,
                  FillColor  = white,
                  MinSize = 11pt,
                  InnerSep=0pt,
                  LineWidth = .5pt]
   \SetVertexNoLabel
   \tikzset{LabelStyle/.style = {below, fill = white, text = black, fill opacity=0, text opacity = 1}}
   \tikzset{EdgeStyle/.style= {thin,
                                double          = black,
                                double distance = .5pt}}
    \begin{scope}\grComplete[RA=1.5]{7}\end{scope}
    \begin{scope}[shift={(4,0)}]\grComplete[prefix=a,RA=1.5]{7}\end{scope}
    \begin{scope}[shift={(8,0)}]\grComplete[prefix=b,RA=1.5]{7}\end{scope}
    \begin{scope}[shift={(0,0)}]\grComplete[prefix=c,RA=1.5]{7}\end{scope}
    \begin{scope}[shift={(-.25,-3)}]\grPath[Math,prefix=u,RA=1.5,RS=0]{7}\end{scope}
        \Edge(a0)(b3)
        \Edge(c0)(a3)
                                                                           \tikzset{LabelStyle/.style = {right, fill = white, text = black, fill opacity=0, text opacity = 1}}
        \tikzset{EdgeStyle/.append style = {thin,  bend right=15}}
    \Edge[label=$\theta$](u0)(c6)
    \tikzset{EdgeStyle/.append style = {thin, dashed, bend left=15}}
                                                                               \tikzset{LabelStyle/.style = {below left, fill = white, text = black, fill opacity=0, text opacity = 1}}
    \Edge[label=$-\theta$](u0)(c4)
\end{tikzpicture}\caption{Testing connectivity on a model with a connected graph as above is difficult. Solid edges correspond to positive interactions with magnitude $\theta$, while the dashed edge corresponds to a negative interaction with magnitude $-\theta$. The cliques are of size $s$ so that the total degree remains at most $s$. When the value of $\theta$ is large, the majority of the spins in each clique tend to have the same sign. Hence the two interactions of the leftmost clique with the leftmost node in the path graph are ``likely'' to cancel out, which will make it hard to tell this graph from the disconnected graph consisting of the connected cliques and the path graph. We exploit this construction in the proof of Proposition \ref{connectivity:low:temp}.}\label{connectivity:antiferro:figure}
\end{figure}
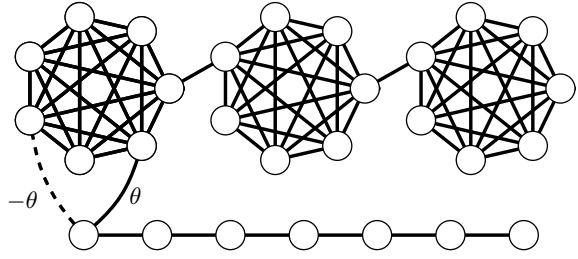

\begin{proposition}[Connectivity Testing General Upper Bound]\label{connectivity:low:temp} Let  $\cP$ be graph connectivity. Assume $G_{\wb}$ belongs to the restriction set $\cR_s = \{G_{\wb}  ~|~ \maxdeg(G) \leq s\}$. Let $s,n,d$ be sufficiently large, and suppose $\theta \geq \frac{3}{2\lfloor s/4\rfloor - 2}$ and there exists a $\kappa > 1$ so that
\begin{align}\label{conn:testing:cond:anti}
\theta >  \frac{2\log\frac{\kappa sn}{\log(ds)}}{s - 16}.
\end{align} 
Then the minimax risk (\ref{testing:risk:def:antifer}) satisfies $
\liminf_{n \rightarrow \infty} R_n(\cP, \cR_s, \theta) = 1.$
\end{proposition}

Importantly, condition (\ref{conn:testing:cond:anti}) implies that the maximum degree cannot be too large with respect to the other parameters if we hope for a connectivity test with a good control over both type I and type II errors to exist. Recall that no such conditions were needed in Corollary \ref{conn:test:consistency} when testing connectivity in ferromagnets. 

\subsection{Correlation Testing for General Models}\label{corr:testing:general:models}

Section \ref{antiferromagnetic:bounds:sub} made it apparent that property testing is more challenging in the enlarged model space (\ref{ising:model:measure:anti}). In this section we work with an even larger class of zero-field models compared to (\ref{ising:model:measure:anti}) which are specified as: 
\begin{align}\label{ising:model:measure:anti:full:generality}
\PP_{\theta, G_{\wb}}(\bX = \xb) \varpropto \exp\Big( \theta \sum_{(u,v) \in E(G)} w_{uv} x_u x_v\Big),
\end{align}
where $1 \leq |w_{uv}| \leq \Theta/\theta, (u,v) \in E(G)$ are unknown parameters and $\xb \in \{\pm 1\}^d$. For a monotone property $\cP$  define the sets of weighted graphs
\begin{align*}
\cG_0(\cP, \theta,\Theta) &:= \{G_{\wb}  ~|~ 1 \leq |w_{uv}| \leq \Theta/\theta, \cP(G) = 0\}, \\\cG_1(\cP, \theta, \Theta) &:= \{G_{\wb}  ~|~ 1 \leq |w_{uv}| \leq \Theta/\theta, \cP(G) = 1\}.
\end{align*}
Different from Section \ref{ferromagnetic:algorithms}, here we impose signal strength restrictions even in the null set $\cG_0(\cP, \Theta) $ and leave the more general setting for future work. We note that one can no longer rely on the screening algorithms of Section \ref{ferromagnetic:algorithms} to perform a property test for data generated by (\ref{ising:model:measure:anti:full:generality}); the success of correlation screening hinges on the fact that in ferromagnets deleting edges reduces correlations, which no longer holds in the model class (\ref{ising:model:measure:anti:full:generality}). One alternative would be to perform exact structure recovery, and check whether the graph property in question holds on the desired graph. Possibilities of exact graph recovery include methods developed in \citep{bresler2008reconstruction,Ravikumar2011High, santhanam2012information, anandkumar2012high, bresler2015efficiently}. 

Below we take a different route, and modify the correlation decoders of \cite{santhanam2012information} by specializing them to property testing. Specifically, we consider a score test type of approach, which only involves model fitting assuming the null hypothesis holds. Suppose there exists an algorithm $\cA$ mapping the data input as
\begin{align}\label{data:mapping}
\cA (\{\bX^{(i)}\}_{i \in [n]}, \theta, \Theta, \cP) \mapsto  \tilde G_{\tilde \wb},
\end{align}
so that the output $\tilde G_{\tilde \wb} \in \cG_0(\cP, \theta, \Theta)$ and in addition if the true underlying graph $G$ satisfies $\cP(G) = 0$ then
\begin{align}\label{expectation:bound}
\max_{u,v \in [d]}|\hat \EE X_u X_v - \EE_{\theta, \tilde G_{\tilde \wb}} X_u X_v| \leq \varepsilon(\delta),\footnotemark
\end{align}
\footnotetext{Here $\hat \EE$ is the empirical expectation defined in Section \ref{ferromagnetic:algorithms}, while $\EE_{\theta, \tilde G_{\tilde \wb}}$ is the expectation with respect to the measure $\PP_{ \theta, \tilde G_{\tilde \wb}}$ from (\ref{ising:model:measure:anti:full:generality}).}
holds with probability at least $1 - \delta$. Define the test 
\begin{align}\label{combinatorial:test:ising:antiferro}
\psi_{\rho}(\{\bX^{(i)}\}_{i \in [n]}, \theta, \tilde G_{\tilde \wb}) := \mathbbm{1}(\max_{u, v} |\hat \EE X_u X_v - \EE_{\theta, \tilde G_{\tilde \wb}} X_u X_v| \geq \rho).
\end{align}
Recall the definition of the threshold $\tau$ (\ref{tau:univ:thresh}) and let  
\begin{align}\label{new:Tau:def}
\cT := \cT(\theta, \Theta, s) = \frac{\sinh^2(\theta/4)}{2s\Theta(3 \exp(2s\Theta) + 1)}.
\end{align}
The following holds
\begin{proposition}[General Tests Sufficient Conditions] \label{mean:comb:inf:testing:general} Suppose that $\cA$  is an algorithm satisfying (\ref{expectation:bound}), and $(s,n,d,\theta, \Theta)$ are such that
\begin{align}\label{sample:size:cond}
\cT \geq \tau + \varepsilon(\delta),
\end{align}
for a small $\delta > 0$. Then the test $\psi_{\varepsilon(\delta)}$ given in (\ref{combinatorial:test:ising:antiferro}) satisfies
\begin{align}\label{typeItypeII:control}
\sup_{G_{\wb} \in \cG_0(\cP, \theta, \Theta)} \PP(\psi_{\varepsilon(\delta)} = 1) \leq \delta ~~~\mbox{and}~~~ \sup_{G_{\wb} \in \cG_1(\cP, \theta, \Theta)} \PP(\psi_{\varepsilon(\delta)} = 0) \leq \delta.
\end{align}
\end{proposition}

A key component for the existence of a successful test (\ref{combinatorial:test:ising:antiferro}) is the algorithm $\cA$ satisfying condition (\ref{expectation:bound}). One example how to construct such an algorithm, is to solve the following optimization problem
\begin{align}\label{combinatorial:algo:ising:antiferro}
\cA(\{\bX^{(i)}\}_{i \in [n]}, \theta, \Theta, \cP)  = \argmin_{G_{\wb} \in \cG_0(\cP, \theta, \Theta)} \max_{u,v \in [d]}|\hat \EE X_u X_v - \EE_{\theta, G_{\wb}}X_u X_v|.
\end{align}
In the following lemma we show that (\ref{expectation:bound}) indeed holds.
\begin{lemma}[Algorithm \eqref{combinatorial:algo:ising:antiferro} Sufficient Condition] \label{simple:lemma:F:scan} If the algorithm $\cA$ is defined by (\ref{combinatorial:algo:ising:antiferro}), then (\ref{expectation:bound}) holds with $\varepsilon(\delta) = \tau.$
\end{lemma}

The proof of Lemma \ref{simple:lemma:F:scan} is a direct consequence of Hoeffding's inequality. By combining the statements of Proposition \ref{mean:comb:inf:testing:general} and Lemma \ref{simple:lemma:F:scan}, we arrive at an abstract generic property test summarized in Algorithm \ref{generic:test}.

\begin{algorithm}
\caption{Generic Property Test}\label{generic:test}
\begin{algorithmic}
\STATE \textbf{Input:} $\{\bX^{(i)}\}_{i \in [n]}, \theta, \Theta, \cP$
\STATE Calculate the matrix $\{\hat \EE X_u X_v\}_{u,v \in [d]}$
\STATE Solve $\tilde G_{\tilde \wb} = \argmin_{G_{\wb} \in \cG_0(\cP, \theta,\Theta)} \max_{u,v \in [d]}|\hat \EE X_u X_v - \EE_{\theta, G_{\wb}}X_u X_v|$
\STATE Output $\psi_{\tau}(\{\bX^{(i)}\}_{i \in [n]}, \theta, \tilde G_{\tilde \wb})$ 
\end{algorithmic}
\end{algorithm}

A sufficient condition for Algorithm \ref{generic:test} to satisfy (\ref{typeItypeII:control}) is $
\cT \geq 2 \tau$, where $\cT$ is defined in (\ref{new:Tau:def}). One potential problem with Algorithm \ref{generic:test} is that solving (\ref{combinatorial:algo:ising:antiferro}) in general requires combinatorial optimization, which will likely result in non-polynomial runtime complexity for most properties. However, unlike the structure learning procedure of \cite{santhanam2012information} which also requires combinatorial optimization, Algorithm \ref{generic:test} has the advantage that it does not need to optimize over the entire set of graphs $\cG_d$ but only over the smaller set $\{G \in \cG_d ~|~ \cP(G) = 0\}$. We conclude this section by proposing a custom  variant of this algorithm specialized to cycle testing which uses a different algorithm $\cA$ and can be ran in polynomial time. 

\subsection{A Computationally Efficient Cycle Test} \label{comp:eff:cycle:test}

In this section we propose an efficient algorithm $\cA$ satisfying (\ref{expectation:bound}) for cycle testing. Having computationally efficient algorithms for cycle testing is beneficial in practice, since if we have enough evidence that the graph is a forest, we can recover its structure efficiently \citep{chow68approximating}. We summarize the algorithm $\cA$ called ``cycle test map'' in Algorithm \ref{cycle:test:step1}.  
\begin{algorithm}
\caption{Cycle Test Map}\label{cycle:test:step1}
\begin{algorithmic}
\STATE \textbf{Input:} $\{\bX^{(i)}\}_{i \in [n]}, \theta, \Theta$
\STATE Calculate the matrix $\Mb := \{\hat \EE X_u X_v\}_{u,v \in [d]}$
\STATE Find a MST $\tilde T$ with edge weights $|M_{uv}|$
\FOR{$1 \leq u < v \leq d$}
\IF{$(u,v) \not \in E(\tilde T)$ or $|M_{uv}| < \tanh(\theta) - \tau$} 
\STATE $\tilde w_{uv} \leftarrow 0$; $E(\tilde T) \leftarrow E(\tilde T) \setminus \{(u,v)\}$
\ELSE
\STATE $\tilde w_{uv} \leftarrow \sign(M_{uv})((\atanh(|M_{uv}|) \wedge \Theta) \vee \theta)/\theta$
\ENDIF
\ENDFOR
\RETURN $\tilde T_{\tilde \wb}$
\end{algorithmic}
\end{algorithm}

Given the output $\tilde T_{\tilde \wb}$ of Algorithm \ref{cycle:test:step1}, evaluating the expectations $\EE_{\theta, \tilde T_{\tilde \wb}} X_u X_v$ needed in (\ref{combinatorial:test:ising:antiferro}) can be done in polynomial time via the simple formula
$$\textstyle\EE_{\theta, \tilde T_{\tilde \wb}} X_u X_v = \prod_{(k,\ell) \in \cP^{\tilde T}_{u \rightarrow v}}\EE_{\theta, \tilde T_{\tilde \wb}} X_kX_\ell= \prod_{(k,\ell) \in \cP^{\tilde T}_{u \rightarrow v}} \tanh(\theta \tilde w_{k \ell}),\footnotemark$$\footnotetext{The validity of this formula follows  by Proposition \ref{restriction:prop} and Lemma \ref{chain:graph} which can be found in the supplement.}where $\cP^{\tilde T}_{u \rightarrow v}$ denotes the path between vertices $u$ and $v$ in the forest $\tilde T$. Next, we show the validity of the test in (\ref{combinatorial:test:ising:antiferro}).

\begin{proposition}[Fast Cycle Test Sufficient Conditions]\label{test:cycle:generic:proof} Suppose that $\tanh(\theta)(1- \tanh(\Theta))> 2\tau$. Then the output of Algorithm \ref{cycle:test:step1} satisfies (\ref{expectation:bound}) with
\begin{align}\label{eff:cycle:test:eps:delta}\varepsilon(\delta) = \tau\frac{2 - \tanh(\Theta)}{1 - \tanh(\Theta)}.\end{align}
\end{proposition}

By combining Propositions  \ref{mean:comb:inf:testing:general} and \ref{test:cycle:generic:proof} we immediately conclude that if $
\cT \geq \tau\frac{3 - 2\tanh(\Theta)}{1 - \tanh(\Theta)},$
and the constraints of Proposition \ref{test:cycle:generic:proof} hold, using the output $\tilde T_{\tilde \wb}$ of Algorithm \ref{cycle:test:step1} with the test $\psi_{\varepsilon(\delta)}$ of (\ref{combinatorial:test:ising:antiferro}) with $\varepsilon(\delta)$ as in (\ref{eff:cycle:test:eps:delta}), satisfy (\ref{typeItypeII:control}). 

\section{Strongly Monotone Properties Proofs from Section \ref{generic:lower:bounds:sec}}\label{proofs:from:section:2}

In this section we give the proofs of the general results on strongly monotone properties: Theorem \ref{biclques:theorem} and Proposition \ref{simple:lower:bound:monotone}. Other proofs from Section \ref{generic:lower:bounds:sec} including the proof of Theorem \ref{ising:single:edge} can be found in Appendix \ref{bounds:ferromagnets}. Since the signal strength is uniformly equal to $\theta$ in all measures that we consider in this section, we will suppress the dependency on $\theta$ whenever that does not cause confusion. For the convenience of the reader below is a definition of $\chi^2$-divergence which we use in the proofs.

\begin{definition}[$\chi^2$-divergence]\label{def:chisq:div} For two measures $\PP$ and $\QQ$ satisfying $\PP \ll \QQ$ the $\chi^2$-divergence is defined by
$$
\textstyle D_{\chi^2}(\PP, \QQ) = \EE_{\QQ} \left(\frac{d\PP}{d \QQ} - 1\right)^2.
$$
\end{definition}

Before we prove Theorem \ref{biclques:theorem} we state a key Lemma whose proof is given in Appendix \ref{bounds:ferromagnets}.

\begin{lemma}[Low Temperature Bound]\label{biclique:lemma} Let $G = (V,E)$ be a graph, and $k, \ell \in V$ be vertices such that $(k,\ell) \not \in E$. Let $r \geq 2$ and $l$ be integers such that there exists an $l \times r$ biclique $B \trianglelefteq G$, containing $k$ and $\ell$ on its right (i.e., $r$ side). Then for values of $\theta \geq \frac{2}{l}$, $\theta \geq \frac{3}{r - 2}$ for $r > 2$ and $\theta \geq \log 2$ when $r = 2$ we have:
$$
\EE_{G} X_k X_\ell \geq 1 - \frac{2(r-1)}{\exp(\theta l) + r - 1}.
$$
\end{lemma}

\begin{proof}[Proof of Theorem \ref{biclques:theorem}] First we point out a simple implication of (\ref{theta:scaling:upper:bound:max:monotone}). If $\theta \geq \frac{\log \frac{2\kappa n r}{\log \lfloor d/(l + r) \rfloor}}{l},$ then certainly
\begin{align}\label{theta:condition}\theta \geq \frac{\log\Big[\frac{2\kappa n (r - 1)}{\log \lfloor d/(l + r) \rfloor} - (r - 1)\Big]}{l}.\end{align}
We will now argue that even if the above holds the asymptotic minimax risk is still $1$. Note that since $B \trianglelefteq H_0$ we have $|V(H_0)| = |V(B)| = l + r$. Consider a graph $G_0$ based on the union of $m = \lfloor d/(l + r)\rfloor$ disconnected copies of the graph $H_0$ (see Figure \ref{biclique:example:cycle:triangle}). Let $e_k= (u_k, v_k)$ be the local copy of the edge $e = (u,v)$ in the $k$\textsuperscript{th} copy of $H_0$. Since the property $\cP$ can be represented as a maximum over subgraphs of $G_0$, by the assumption of the theorem $\cP(G_0) = 0 $, while adding the edge $e_k$ for any $k \in [m]$ to the $k$\textsuperscript{th} copy of $H_0$ transfers $G_0$ to a graph $G_k$ satisfying $\cP(G_k) = 1$. For each graph in $i \in \{0,1,\ldots,m\}$ let $\PP_i$\footnote{I.e., $\PP_i$ is a shorthand for $\PP_{\theta, G_i}$ as defined in \eqref{ising:model:measure}.} be the measure corresponding to Ising models with graphs $\{G_0, G_1, \ldots, G_{m}\}$, and $\EE_k$ be the corresponding expectation under $\PP_k$. Define the mixture measure $\overline \PP^{\otimes n} = \frac{1}{m} \sum_{k \in [m]} \PP^{\otimes n}_{k}$. Using Lemma \ref{lemma:chi:sq} 
\begin{align*}
D_{\chi^2}(\overline \PP^{\otimes n}, \PP^{\otimes n}_0) + 1 & \leq \frac{1}{m^2} \sum_{j,k \in [m]} (1 + \tanh(\theta)[\EE_{{j}} X_{u_k} X_{v_k} -  \EE_{{0}} X_{u_k} X_{v_k}])^n.
\end{align*}
By Proposition \ref{restriction:prop}, since the copies of $H_0$ are disconnected, for $k \neq j$  we have
\[
\EE_{{j}} X_{u_k} X_{v_k} -  \EE_{{0}} X_{u_k} X_{v_k} = 0,
\] 
while by Lemma \ref{biclique:lemma} if $k \equiv j$
\[
\EE_{{j}} X_{u_k} X_{v_k} -  \EE_{{0}} X_{u_k} X_{v_k} \leq 1 -  \EE_{{0}} X_{u_k} X_{v_k} \leq \frac{2(r - 1)}{\exp(\theta l) + r - 1}.
\] 
We conclude that
\begin{align*}
D_{\chi^2}(\overline \PP^{\otimes n}, \PP^{\otimes n}_0) + 1 & \leq \frac{m - 1}{m} + \frac{1}{m} \Big(1 + \tanh(\theta)\frac{2(r - 1)}{\exp(\theta l) + r - 1}\Big)^n \\
& \leq \frac{m - 1}{m} + \frac{1}{m} \Big(1 + \frac{2(r - 1)}{\exp(\theta l) + r - 1}\Big)^n\\
& \leq \frac{m - 1}{m} + \frac{\exp\Big(\frac{2n(r - 1)}{\exp(\theta l) + r - 1}\Big)}{m}.
\end{align*}
Using \eqref{theta:condition} a simple calculation shows that
$$
\limsup_{n \rightarrow \infty} D_{\chi^2}(\overline \PP^{\otimes n}, \PP^{\otimes n}_0) = 0.
$$

Recall that by Le Cam's lemma we have the bound
\begin{align}\label{le:cams:ineq:main:text}
R_n(\cP, \theta) \geq \inf_{\psi} \Big[\PP^{\otimes n}_0(\psi = 1) + \overline \PP^{\otimes n}(\psi = 0)\Big] \geq 1 - \frac{1}{2}\sqrt{D_{\chi^2}(\overline \PP^{\otimes n}, \PP^{\otimes n}_0)},
\end{align}
which completes the proof.
\end{proof}

Below we prove Proposition \ref{simple:lower:bound:monotone}. The proof utilizes a similar construction to the one used in the proof of Theorem \ref{biclques:theorem}.

\begin{proof}[Proof of Proposition \ref{simple:lower:bound:monotone}] Construct a null graph $G_0$ by repeating $H_0$ $\lfloor d/m\rfloor$ times. Let $\overline \PP^{\otimes n}$ be the mixture of measures $\PP^{\otimes n}_j$, where $\PP_j$ is the measure corresponding to adding an edge to one of the ``clones'' of $H_0$, thus transferring $G_0$ to a graph $G_j$ such that $\cP(G_j) = 1$, and let $\PP^{}_0$ be the Ising measure under the uniform signal model with graph $G_0$. Let $\EE_j$ and $\EE_0$ be the expectations under $\PP_j$ and $\PP_0$ respectively. Let $e_j = (u_j, v_j) = E(G_j) \setminus E(G_0)$ be the edge that distinguishes $G_j$ from $G_0$. An example of $G_0$ and one of the alternative graphs $G_j$ is given on Figure \ref{biclique:example:cycle:triangle}. 
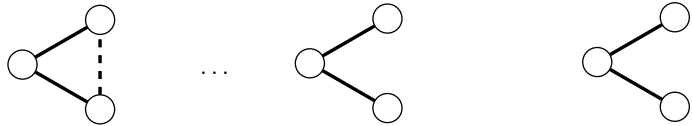
\begin{figure}[H] 
\centering
\begin{tikzpicture}[scale=.55]
\SetVertexNormal[Shape  = circle,
                  FillColor  = white,
                  MinSize = 11pt,
                  InnerSep=0pt,
                  LineWidth = .5pt]
   \SetVertexNoLabel
   \tikzset{EdgeStyle/.style= {thin,
                                double          = black,
                                double distance = .5pt}}
                                \begin{scope}[rotate=60, shift={(0cm, 0cm)}]
                                \grEmptyCycle[prefix=a,RA=1.25]{3}{1}\end{scope}
                                                                                \begin{scope}[rotate=60, shift={(-7cm, 12cm)}]                \grEmptyCycle[prefix=c,RA=1.25]{3}{1}\end{scope}
                                                                                
                                                                              \node[] at (-10.5,-.25) {$\ldots$};
                                \begin{scope}[rotate=60, shift={(-3.5cm, 6cm)}]\grEmptyCycle[prefix=b,RA=1.25]{3}{1}\end{scope} 
                                                                    \Edge(a1)(a2)
                                                                    \Edge(a1)(a0)
                                                                    \Edge(b1)(b2)
                                                                    \Edge(b1)(b0)
                                                                    \Edge(c1)(c2)
                                                                    \Edge(c1)(c0)

    \tikzset{EdgeStyle/.style= {dashed,thin,
                                double          = black,
                                double distance = .5pt}}
\Edge(c2)(c0)
    \tikzset{LabelStyle/.style = {below, fill = white, text = black, fill opacity=0, text opacity = 1}}
\end{tikzpicture}
\caption{For this figure let $\cP$ be cycle testing. The figure shows an example of $d/3$ incomplete triangle graphs ($G_0$), and takes a mixture of distributions adding one edge to complete the triangles one at a time.}\label{biclique:example:cycle:triangle}
\end{figure}
Using Lemma \ref{lemma:chi:sq} we can evaluate the divergence $D_{\chi^2}$: 
\begin{align*}
D_{\chi^2}(\overline \PP^{\otimes n}, \PP^{\otimes n}_0) & = \frac{1}{\lfloor d/m \rfloor^2} \sum_{j,k \in [\lfloor d/m \rfloor]} (1 + \tanh(\theta)[\EE_jX_{u_k} X_{v_k} - \EE_0 X_{u_k} X_{v_k} ])^n  - 1\\
& = \frac{1}{\lfloor d/m \rfloor} (1 + \tanh(\theta)[\EE_jX_{u_j} X_{v_j} - \EE_0 X_{u_j} X_{v_j} ])^n - \frac{1}{\lfloor d/m \rfloor}.
\end{align*}
By Lemma 4 of \cite{tandon2014information} we know that $\EE_jX_{u_j} X_{v_j} - \EE_0 X_{u_j} X_{v_j} \leq \tanh(\theta)$. Therefore
$$
D_{\chi^2}(\overline \PP^{\otimes n}, \PP^{\otimes n}_0) \leq \frac{1}{\lfloor d/m \rfloor} (1 + \tanh^2(\theta))^n \leq \frac{2^n}{\lfloor d/m \rfloor},
$$
whereby by the first inequality if $\tanh(\theta) < \kappa  \sqrt{\frac{\log \lfloor d/m \rfloor}{n}}$ for some $\kappa < 1$ the above $\rightarrow 0$. The second inequality proves the last implication of the Proposition after an application of Le Cam's argument \eqref{le:cams:ineq:main:text}. 
\end{proof}

\section{Discussion}\label{discussion:sec}

In this manuscript we formalized necessary and sufficient conditions for property testing in Ising models. Specifically, we showed lower and upper information-theoretic bounds on the temperature for ferromagnetic models. Furthermore, we argued that greedy correlation screening works well at high temperature regimes, and can also be useful in low  temperature regimes for certain properties. We also demonstrated that testing strongly monotone properties over the class of general Ising models is strictly more difficult than testing in ferromagnets. We discussed generic property tests based on correlation decoding, and developed a computationally efficient cycle test. 

Important problems that we plan to investigate in future work include --- searching for more sophisticated algorithms than correlation screening which will work at low temperature regimes for testing any property in ferromagnets; relating the temperature of the system to the information-theoretic limits of strongly monotone property testing in models with antiferromagnetic interactions; utilizing computationally efficient structure recovery algorithms, such as those of \cite{bresler2015efficiently}, to obtain tractable algorithms for property testing in general models. Finally several further problems merit further work: can one test for connectivity in the regime $\log d \gtrsim n$ if $\tanh(\theta) < 1$ in finite samples in ferromagnets; are there algorithms matching the information-theoretic limitations when testing over general models; is there a more general version of Theorem \ref{biclques:theorem} which works for general properties not only strongly monotone properties.

\section*{Acknowledgments} The authors would like to express their gratitude to Professors Michael Aizenman, Ramon van Handel and Allan Sly, for thought provoking discussions, their insights and encouragements. Special thanks are due to Ramon, who suggested the interpolating measure $\PP_{G\oplus t G'}$ towards the proof of Theorem \ref{most:important:theorem}. Further thanks are due to the Editor, Associate Editor and two referees whose comments and constructive suggestions led to significant improvements of the manuscript. 

\bibliographystyle{ims}
\bibliography{combinatorial_inference_ising}
\newpage
\appendix

\title{Supplementary Material to ``Property Testing in High Dimensional Ising Models''}

The supplementary material is organized as follows:

\begin{itemize}

\item Appendix \ref{aux:app:sec} collects several auxiliary results which we utilize in the later sections.

\item Appendix \ref{bounds:ferromagnets} contains the main minimax risk lower bound proofs for ferromagnets. This includes the proof of Theorem \ref{ising:single:edge}. 

\item Appendix \ref{corr:screening:proofs} contains all proofs for the correlation screening algorithm in ferromagnets. 

\item Appendix \ref{bounds:with:ferromagnets} contains proofs for the limitations in general models. 

\item Appendix \ref{correlation:testing:general:models} contains proofs for the correlation testing algorithms for general models, including the proof of the computationally efficient cycle test. 

\end{itemize}

\section{Auxiliary Results}\label{aux:app:sec}

\begin{lemma}[\cite{fisher1967critical}]\label{fisher:bound} Let $G$ be a graph and $\PP_{\theta, G}$ be the measure given by \eqref{ising:model:measure}. Denote by $N_{uv}(k)$ the number of self avoiding walks of length $k$ between $u$ and $v$ on $G$. Then we have the bound:
$$
\EE_{\theta, G} X_u X_v \leq \sum_{k = d_G(u,v)}^{\infty} \tanh(\theta)^k N_{uv}(k).
$$
\end{lemma}


\begin{theorem}[Griffith's inequality, \cite{lebowitz1974ghs}]\label{FKG:ineq} For a zero-field ferromagnetic model, and two sets $A, B \subset [d]$ we have 
$$
\EE \prod_{i \in A} X_i \prod_{i \in B} X_i \geq \EE \prod_{i \in A} X_i \EE \prod_{i \in B} X_i.
$$

\end{theorem}

In order to formalize the next result we need some notation. Let $\PP_1$ and $\PP_2$ be two Gibbs measures. Dobrushin's interaction matrix $\Cb$ is defined element-wise as
$$
C_{r m} = \sup_{\xb = \by \mbox{ \tiny off }r } \{D_{\TV}(\PP_{1}(\cdot | \bX_{-m} = \xb), \PP_{1}(\cdot | \bX_{-m} = \yb))\}\footnotemark,
$$
\footnotetext{Recall that for two probability measures $\PP$ and $\QQ$ on a probability space $(\Omega,\Sigma)$, we have $D_{\TV}(\PP, \QQ) := \sup_{A \in \Sigma} |\PP(A) - \QQ(A)|$.}
where equality off $r$ means except on the $r$th position. Let $\Db$ be
$$
\Db := \sum_{i = 0}^\infty \Cb^i.
$$
Define the vector $\bb$ element-wise as
$$
b_k = \EE D_{\TV}(\PP_1(\cdot | \bX_{-m}), \PP_2(\cdot | \bX_{-m})).
$$
Suppose $f$ is a function whose expectation one wants to control.  Define
$$
\delta_i(f) := \sup \{|f(\xb) - f(\yb)| : \xb = \yb \mbox{ off } i \}.
$$
We have
\begin{theorem}[Dobrushin's Comparison Theorem \citep{follmer1988random}]\label{dobrushins:comparison:thm} Suppose that $\lim_{l\rightarrow \infty}\allowbreak \|\Cb^l\|_1 \allowbreak = 0$. Then
$$
|\EE_{\PP_1} f -  \EE_{\PP_2} f | \leq \sum_{i} [\bb^{\top}\Db]_i\delta_i(f).
$$ 
\end{theorem}

In this section we will further spell out the details of some auxiliary results. For the next result assume we have $m$ distinct graphs $G_j$ for $j \in [m]$ which differ from a graph $G_0$ in a single edge, i.e. $G_j = (V,E_j)$, where $|E_j\setminus E_0| = 1$. For each $j$, denote the only edge in the set $E_j \setminus E_0$ by $e_j = (u_j, v_j)$. Define for brevity $\PP_{j} = \PP_{\theta, G_j}$ and $\PP_0 = \PP_{\theta, G_0}$ where the measures $\PP_{\theta, G_0}, \PP_{\theta, G_j}$ are defined in (\ref{ising:model:measure}). Denote the corresponding expectations with $\EE_j$ and $\EE_0$. 

\begin{lemma}[Edge Addition]\label{lemma:chi:sq} Suppose one has ferromagnetic models as in \eqref{ising:model:measure}. The following bound holds
\begin{align*}
\EE_0 \frac{\PP_j}{\PP_0}\frac{\PP_k}{\PP_0} \leq 1 + \tanh(\theta)[\EE_{j} X_{u_k} X_{v_k} -  \EE_{0} X_{u_k} X_{v_k}].\footnotemark
\end{align*}
\end{lemma}

\footnotetext{We omit writing the argument $\bX$ in the ratio $\frac{\PP_j(\bX)}{\PP_0(\bX)}\frac{\PP_k(\bX)}{\PP_0(\bX)}$ for brevity.}

\begin{remark} By Griffith's inequality (Theorem \ref{FKG:ineq}) we always have $\EE_{j} X_{u_k} X_{v_k} -  \EE_{0} X_{u_k} X_{v_k} \allowbreak \geq 0$, and therefore $\EE_{0} \frac{\PP_j}{\PP_0}\frac{\PP_k}{\PP_0}  \geq 1$. Furthermore, if we are interested in product measures over $n$ i.i.d. observations, i.e, $\PP_j^{\otimes n}$ and $\PP_0^{\otimes n}$ Lemma \ref{lemma:chi:sq} immediately to yields the following bound
$$
\EE_{\PP^{\otimes n}_0}\frac{\PP_j^{\otimes n}}{\PP_0^{\otimes n}}\frac{\PP_k^{\otimes n}}{\PP_0^{\otimes n}} = \bigg[\EE_0 \frac{\PP_j}{\PP_0}\frac{\PP_k}{\PP_0}\bigg]^n \leq (1 + \tanh(\theta)[\EE_{j} X_{u_k} X_{v_k} -  \EE_{0} X_{u_k} X_{v_k}])^n.
$$
\end{remark}

\begin{proof}[Proof of Lemma \ref{lemma:chi:sq}] Note that by definition we have:
$$
\frac{\PP_k(\xb)}{\PP_0(\xb)} = \frac{Z_{\theta, G}}{Z_{\theta, G_k}} \exp(\theta x_{u_k} x_{v_k}).
$$
Hence 
$$
\frac{Z_{\theta, G_k}}{Z_{\theta, G}} = \sum_{\xb \in\{\pm 1\}^d} \frac{Z_{\theta, G_k}}{Z_{\theta, G}} \PP_k(\xb) = \sum_{\xb \in\{\pm 1\}^d} \exp(\theta x_{u_k} x_{v_k}) \PP_0(\xb) = \EE_0\exp(\theta X_{u_k} X_{v_k}).
$$
Hence $\PP_0(\xb) \frac{\exp(\theta x_{u_k} x_{v_k})}{\EE_{0} \exp(\theta X_{u_k} X_{v_k})} = \PP_k(\xb)$. This shows that:
$$
\EE_{0} \frac{\PP_j}{\PP_0}\frac{\PP_k}{\PP_0} = \EE_{j} \frac{\PP_k}{\PP_0} = \frac{\EE_{j} \exp(\theta X_{u_k} X_{v_k})}{\EE_{0} \exp(\theta X_{u_k} X_{v_k})}.
$$
Using that for any $c \in \RR$ and $x \in \{\pm 1\}$ we have the identity
\begin{align}\label{cosh:exp:identity}
\exp(c x) = \cosh(c)(1 + x \tanh(c)).
\end{align}
implies that for $\QQ = \PP_0, \PP_j$
$$
\EE_{\QQ} \exp(\theta X_{u_k} X_{v_k}) = \sinh(\theta)\EE_{\QQ} X_{u_k} X_{v_k}  + \cosh(\theta),
$$
Consequently
\begin{align*}
\EE_{0} \frac{\PP_j}{\PP_0}\frac{\PP_k}{\PP_0} & =\frac{\EE_{j} \exp(\theta X_{u_k} X_{v_k})}{\EE_{0} \exp(\theta X_{u_k} X_{v_k})} = 1 + \frac{\sinh(\theta) [\EE_{j} X_{u_k} X_{v_k} -  \EE_{0} X_{u_k} X_{v_k}]}{\sinh(\theta) \EE_{0} X_{u_k} X_{v_k} + \cosh(\theta)}\\
& = 1 + \frac{\EE_{j} X_{u_k} X_{v_k} -  \EE_{0} X_{u_k} X_{v_k}}{\EE_{0} X_{u_k} X_{v_k} + \coth(\theta)} \\
& \leq 1 + \tanh(\theta)[\EE_{j} X_{u_k} X_{v_k} -  \EE_{0} X_{u_k} X_{v_k}],
\end{align*}
where for the last inequality we note that by Griffith's inequality (Theorem \ref{FKG:ineq}) $\EE_{0} X_{u_k} X_{v_k} \geq 0$. With this the proof is complete. 
\end{proof}

For a given graph $G = (V,E)$ and vertex set $W$ let $G \vert_{W} = (W, E \vert_{W})$, where $E \vert_{W} = \{(u,v) \in E ~|~ u,v \in W \}$ is the restriction of $G$ on the vertex set $W$. The next proposition states that when studying correlations it suffices to consider a potentially smaller graph. 

\begin{proposition}\label{restriction:prop} Let $G = (V,E)$ be any graph, and $\PP_{\theta, G_{\wb}}$ be the measure of an Ising model with graph $G$ as in (\ref{ising:model:measure:anti:full:generality}). Let $u, v \in V$ be two fixed vertices. Denote by $\cP^{G}_{u \rightarrow v}$ the set of all simple paths on $G$ connecting $u$ and $v$, and let $\cV_{G}$ be the set of all vertices in $\cP^{G}_{u \rightarrow v}$. Let $\tilde G = G \vert_{\cV_{G}}$ and $\PP_{\theta, \tilde G_{\wb}}$ be the measure of the Ising model with weights $\wb$ restricted to $\tilde G$. Then 
$$
\EE_{\theta, G_{\wb}}X_u X_v =  \EE_{\theta, \tilde G_{\wb}}X_u X_v.
$$
\end{proposition}

\begin{proof}[Proof of Proposition \ref{restriction:prop}] Construct the graph $\tilde G^c = (V, E \setminus E\vert_{\cV_{G}})$. Take any two distinct vertices $v_1, v_2 \in \cV_{G}$. Note that $v_1$ and $v_2$ belong to two distinct connected components in $\tilde G^c$. To see this, assume the contrary. This implies that there exists a path on $\tilde G^c$ connecting $v_1$ and $v_2$ which is a contradiction since such a path does not belong to the set $\cP^G_{u \rightarrow v}$. For each vertex $\ell \in \cV_{G}$ let $\cC_{\ell}$ denote the connected component of $\tilde G^c$ containing $\ell$, and let $V(\cC_{\ell})$ be its vertex set. Let $V^c = V \setminus \cup_{\ell \in \cV_{G}} (V(\cC_\ell) \setminus \{\ell\})$ denote the vertex set which is completely disconnected from the graph $\tilde G$. In terms of this notation we have the decomposition:
$$
V = \cup_{\ell \in \cV_{G}} (V(\cC_\ell) \setminus \{\ell\}) \cup \cV_{G} \cup V^c.
$$

Let $(X_\ell)_{\ell \in \cV_{G}}$ and  $(Y_\ell)_{\ell \in \cV_{G}}$ be two fixed configurations of spins in the vertex set $\cV_{G}$. Take any configuration of spins $(x_\ell)_{\ell \in V \setminus \cV_{G}}$, and consider the following two states
\begin{align*}
\bS_X &= ((X_\ell)_{\ell \in \cV_{G}}, (X_\ell \odot (x_k)_{k \in V(\cC_\ell) \setminus \{\ell\}})_{\ell \in \cV_{G}}, (x_\ell)_{\ell \in V^c})\\
\bS_Y &= ((Y_\ell)_{\ell \in \cV_{G}}, (Y_\ell \odot (x_k)_{k \in V(\cC_\ell) \setminus \{\ell\}})_{\ell \in \cV_{G}}, (x_\ell)_{\ell \in V^c}),
\end{align*}
where by $\odot$ we mean element-wise multiplication. Due to the construction we conclude:
\begin{align}\label{prob:ration:equality}
\frac{\PP_{\theta, G_{\wb}}(\bS_X)}{\PP_{\theta, G_{\wb}}(\bS_Y)} = \exp\Bigr(\sum_{(k,\ell) \in E\vert_{\cV_{G}}} \theta w_{k\ell}(X_kX_\ell - Y_kY_\ell)\Bigr) = \frac{\PP_{\theta, {\tilde G}_{\wb}}((X_\ell)_{\ell \in \cV_{G}})}{\PP_{\theta, {\tilde G}_{\wb}}((Y_\ell)_{\ell \in \cV_{G}})}.
\end{align}
Note that as we vary the spins $(x_\ell)_{\ell \in V \setminus \cV_{G}}$ while holding $(X_\ell)_{\ell \in \cV_{G}}$ and $(Y_\ell)_{\ell \in \cV_{G}}$ fixed, the values $(X_\ell\odot (x_k)_{k \in V(\cC_\ell) \setminus \{\ell\}})_{\ell \in \cV_{G}}, (x_\ell)_{\ell \in V^c})$ and $(Y_\ell\odot (x_k)_{k \in V(\cC_\ell) \setminus \{\ell\}})_{\ell \in \cV_{G}}, (x_\ell)_{\ell \in V^c})$ take all possible values. The last observation and (\ref{prob:ration:equality}) complete the proof.
\end{proof}

\begin{lemma}[Path Graph Correlations]\label{chain:graph} Let $G=([\ell], \{(j,j+1)\}_{j \in [\ell-1]})$ be a path graph, and $\PP_{\theta, G_{\wb}}$ be the measure of a general Ising model with the graph $G$ as in (\ref{ising:model:measure:anti:full:generality}). For any $k \leq j \in [\ell]$ we have:
$$
\EE_{\theta, G_{\wb}} X_{k + 1} X_{j} = \prod_{u = k + 1}^{j-1}\tanh(\theta w_{u,u+1}).
$$
\end{lemma}

\begin{proof}[Proof of Lemma \ref{chain:graph}] The proof follows by a standard induction argument, and is omitted. \end{proof}

\section{Bounds with Ferromagnets}\label{bounds:ferromagnets}

Since in all measures the signal strength is $\theta$, in this section we will suppress the dependency on $\theta$ whenever that does not cause confusion.

Below we restate and prove Theorem \ref{ising:single:edge} slightly more generally than in the main text.

\begin{theorem}[Theorem \ref{ising:single:edge} Restated]\label{ising:single:edge:restated} Given a binary graph property $\cP$, let $G_0 \in \cG_0(\cP)$, and the set $\cC$ be a divider set with a null base $G_0$. Suppose that $|\cC| \rightarrow \infty$ asymptotically. If for some $c < 1$ and $0 < \varepsilon < 1$ we have
$$
\textstyle \theta \leq (1 - \varepsilon)\sqrt{\frac{\log N(\cC, d_{G_0}, ((-\log c)^{-1} + \varepsilon)\log \log |\cC|)}{n}} \wedge \atanh\bigr(\frac{c}{\maxdeg(G_0) + 1}\bigr),
$$
then $\liminf_{n \rightarrow \infty} R_n(\cP, \theta) = 1.$ The original setting is given by $\varepsilon = \frac{1}{2}$ and $c = e^{-2}$.
\end{theorem}

\begin{proof}[Proof of Theorem \ref{ising:single:edge}] 
Let $S\subseteq \cC$ denote the $r$-packing set of $\cC$ with cardinality $m := |S| = N(\cC, d_{G_0}, r)$, where we have set for brevity $r = ([-\log c]^{-1} + \varepsilon)\log \log |\cC|$. To ease of notation we will use $\PP_0$ for the measure $\PP_{G_0}$ and for an edge $e_j \in S$ we let $\PP_j = \PP_{G_j}$ where $G_j = G_0 \cup \{e_j\}$. Recall the notation $^{\otimes n}$ for a product measure of $n$ i.i.d. observations. Let $\overline \PP^{\otimes n} = \frac{1}{m} \sum_{j \in [m]} \PP^{\otimes n}_j$ denote the mixture density of the graphs from the alternative. By Definition \ref{def:chisq:div} we have
\begin{align}\label{chi:square:identity}
D_{\chi^2}(\overline \PP^{\otimes n}, \PP^{\otimes n}_0) = \EE_{\PP^{\otimes n}_0} \Big (\frac{\overline \PP^{\otimes n}}{\PP^{\otimes n}_0} - 1\Big)^2 = \frac{1}{m^2} \sum_{i,j \in [m]}\EE_{\PP^{\otimes n}_0} \frac{\PP^{\otimes n}_i}{\PP^{\otimes n}_0}\frac{\PP^{\otimes n}_j}{\PP^{\otimes n}_0} - 1. 
\end{align}
We now handle the terms $\EE_{\PP^{\otimes n}_0} \frac{\PP^{\otimes n}_i}{\PP^{\otimes n}_0}\frac{\PP^{\otimes n}_j}{\PP^{\otimes n}_0}$ one by one, and we distinguish two cases. \\

\noindent{\bf Case I.} First assume  that $i \neq j$. We will show that the following bound holds.
\begin{lemma}[High Temperature Bound]\label{key:thm:diff:ising} Let $G = (V,E)$ and $G' = (V,E')$ be two graphs such that $E \setminus E' = \{e\}$, where $e = (u,v)$. Then if $\Ab_G$ denotes the adjacency matrix of $G$, and $\|\Ab_G\|_1 \tanh(\theta) \leq c < 1$ for some small $c > 0$ the following holds 
$$
\EE_G X_k X_\ell - \EE_{G'} X_k X_\ell \leq 2 \sum_{l = 0}^{\infty} \bigr(\tanh(\theta)\bigr)^{l + 1} \bigr([\Ab_G^l]_{uk} + [\Ab_G^l]_{vk}  + [\Ab_G^l]_{u\ell} + [\Ab_G^l]_{v\ell}\bigr),
$$
where the expectations $\EE_G, \EE_{G'}$ are taken with respect to distributions as specified by (\ref{ising:model:measure}).
\end{lemma}

Let $\Lambda := \maxdeg(G_0) = \|\Ab_{G_0}\|_1$, where $\Ab_{G_0}$ is the adjacency matrix of $G_0$. If $\|\Ab_{G_0}\|_2$ denotes the operator norm of $\Ab_{G_0}$, recall that $\|\Ab_{G_0}\|_2 \leq \Lambda$ by Gershgorin's circle Theorem \citep{golub2012matrix}. By the triangle inequality $\|\Ab_{G_j}\|_2 \leq \|\Ab_{G_0}\|_2 + 1 \leq \Lambda + 1$ for all $j$. It immediately follows by our assumption that
$$
\|\Ab_{G_j}\|_1 \tanh(\theta) \leq c < 1,
$$
and hence the condition of Lemma \ref{key:thm:diff:ising} is satisfied.

Recall that for any two vertices $r$ and $t$ in $[d]$, $[\Ab_{G_j}^l]_{rt}$ equals the number of paths of length $l$ between vertex $r$ and vertex $t$. Since $[\Ab_{G_j}^l]_{rt} \leq \|\Ab_{G_j}\|^l_2\leq (\Lambda + 1)^l$, it follows that $(\Lambda + 1)^l$ is an upper bound on the number of paths on the graph $G_j$ between vertices $r$ and $t$ of length $l$. Take the two distinct edges $e_i = (k, \ell), e_j = (u, v)  \in S$. Put $\tilde \Lambda := (\Lambda + 1)$. Let $d_{G_0}(e_i, e_j) = l_{ij}$.  Applying Lemma \ref{key:thm:diff:ising} we have:
\begin{align*}
\MoveEqLeft \EE_{j} X_{k} X_{\ell} - \EE_{{0}}  X_{k} X_{\ell} \\
& \leq 2 \sum_{l = 0}^{\infty} \bigr(\tanh(\theta)\bigr)^{l + 1} \bigr([\Ab_{G_j}^l]_{uk} + [\Ab_{G_j}^l]_{vk}  + [\Ab_{G_j}^l]_{u\ell} + [\Ab_{G_j}^l]_{v\ell}\bigr)\\
&\leq \sum_{l \geq l_{ij}} 8 \tilde \Lambda^l\tanh^{l + 1}(\theta)  \leq \frac{8 \tanh(\theta)(\tilde \Lambda\tanh(\theta))^{l_{ij}}}{1 - \tilde \Lambda \tanh(\theta)}
\end{align*}
\noindent {\bf Case II.} When $i = j$, Lemma 4 of \cite{tandon2014information} gives the following bound
$$
\EE_{j}  X_{k} X_{\ell}  - \EE_{0} X_{k} X_{\ell}  \leq \tanh(\theta).
$$
Applying Lemma \ref{lemma:chi:sq}
\begin{align*}
\MoveEqLeft[0.75] D_{\chi^2}(\overline \PP^{\otimes n}, \PP^{\otimes n}_0) + 1 \\
& \leq \frac{1}{m^2} \sum_{\substack{e_i, e_j \in S\\ i \neq j}} (1 + \tanh(\theta)[\EE_{{j}} X_{k} X_{\ell} -  \EE_{0} X_{k} X_{\ell}])^n + \frac{1}{m^2} \sum_{e_j \in S} (1 + \tanh(\theta)^2)^n\\
& \leq  \frac{1}{m^2} \sum_{\substack{e_i, e_j \in S\\ i \neq j}} \bigg[1 +  \frac{8\tanh^2(\theta)(\tilde \Lambda \tanh(\theta))^{l_{ij}}}{1 - \tilde\Lambda \tanh(\theta)}\bigg]^n  + \frac{(1 + \tanh(\theta)^2)^n}{m}\\
& \leq  \frac{1}{m^2} \sum_{\substack{e_i, e_j \in S\\ i \neq j}} \bigg[1 +  \frac{8}{(1 - c)\tilde \Lambda^2} (\tilde \Lambda \tanh(\theta))^{l_{ij} + 2}\bigg]^n + \frac{(1 + \tanh(\theta)^2)^n}{m} \\
& \leq  \frac{1}{m^2} \sum_{\substack{e_i, e_j \in S\\ i \neq j}}  \exp\bigg(n \frac{8}{(1-c)\tilde\Lambda^2}  (\tilde \Lambda\tanh(\theta))^{l_{ij} + 2}\bigg) + \frac{\exp(n\tanh(\theta)^2)}{m} ,
\end{align*}
where the next to last inequality uses $\theta \leq \atanh(c\tilde \Lambda^{-1})$ and $c < 1$, and the last one uses $x + 1 \leq \exp(x)$. Since $S$ is an $r$-packing we have that $l_{ij} + 2 \geq r + 2 > 2$. Under the assumption that $\tanh(\theta) \leq \theta \leq (1 - \varepsilon)\sqrt{\frac{\log m}{n}}$, we have
\begin{align*}
\MoveEqLeft D_{\chi^2}(\overline \PP^{\otimes n}, \PP^{\otimes n}_0) + 1 \\
& \leq \frac{m- 1}{m} \exp\bigg(n \frac{8}{(1-c)\tilde \Lambda^2}  (\tilde\Lambda\tanh(\theta))^{r + 2}\bigg) + \frac{\exp(n\tanh(\theta)^2)}{m}\\
& \leq \frac{m - 1}{m} \bigg(1 + (e-1) n \frac{8}{(1-c)\tilde \Lambda^2}  (\tilde \Lambda\tanh(\theta))^{r + 2}\bigg) + o(1)\\
& \leq 1 + (e-1) n \frac{8}{(1 - c)\tilde\Lambda^2}  (\tilde\Lambda\tanh(\theta))^{r + 2} + o(1),
\end{align*}
where the next to last inequality holds when $\frac{8 n}{(1-c)\tilde\Lambda^2}  (\tilde\Lambda\tanh(\theta))^{r + 2} \leq 1$ since by convexity $e^x \leq 1 + (e - 1)x$ when $x \leq 1$. In what follows we argue that in fact $\frac{n}{(1-c)\tilde\Lambda^2}  (\tilde\Lambda\tanh(\theta))^{r + 2} = o(1)$, hence the inequality is always true in an asymptotic sense. We have
$$
\textstyle \frac{n}{(1-c)\tilde\Lambda^2}  (\tilde\Lambda\tanh(\theta))^{r + 2}  \leq  \frac{(1 - \varepsilon)^2}{1 -c } \log(m)(\tilde \Lambda\tanh(\theta))^{r} \leq  \frac{(1 - \varepsilon)^2}{1 -c } \log(|\cC|)(\tilde \Lambda\tanh(\theta))^{r}
$$
Taking a $\log$ results in
$$
\log \log |\cC| + \log((1 -\varepsilon)^2/(1-c)) + r \log(\tilde\Lambda\tanh(\theta)) \rightarrow -\infty,
$$
where the limit holds since $r = ([-\log c]^{-1} + \varepsilon) \log \log(|\cC|)$ and hence $\allowbreak r  \log(\tilde\Lambda\tanh(\theta)) \leq r \log(c) \leq -\log\log |\cC| + \varepsilon \log c \log\log |\cC|$, and $\varepsilon \log c < 0$ while $|\cC| \rightarrow \infty$ asymptotically. We have established that under the condition of the theorem 
$$
\limsup_{n \rightarrow \infty} D_{\chi^2}(\overline \PP^{\otimes n}, \PP^{\otimes n}_0) = 0.
$$

Recall that by Le Cam's lemma we have the bound
\begin{align*}
R_n(\cP, \theta) \geq \inf_{\psi} \Big[\PP^{\otimes n}_0(\psi = 1) + \overline \PP^{\otimes n}(\psi = 0)\Big] \geq 1 - \frac{1}{2}\sqrt{D_{\chi^2}(\overline \PP^{\otimes n}, \PP^{\otimes n}_0)}.
\end{align*}

Hence since the LHS of the inequality above is a lower bound on the risk, the proof is completed.
\end{proof}

\begin{proof}[Proof of Lemma \ref{key:thm:diff:ising}] 

To prove this result we make usage of Dobrushin's comparison technique (Theorem \ref{dobrushins:comparison:thm} or see Theorem (2.8) \cite[][]{follmer1988random}) which a is powerful tool for comparing discrepancies between Gibbs measures based on their local specifications. 

We first estimate Dobrushin's interaction matrix $\Cb$ which measures the influence of a vertex $r$ on vertex $m$, and is defined element-wise as
$$
C_{r m} = \sup_{\xb = \by \mbox{ \tiny off }r } \{D_{\TV}(\PP_{G,m}(\cdot | \bX_{-m} = \xb), \PP_{G,m}(\cdot | \bX_{-m} = \yb))\}.
$$
In the above, $\xb$ and $\yb$ are $d-1$ dimensional vectors, $\xb = \by \mbox{ off }r $ means that $\xb$ and $\yb$ coincide except the entries $x_r$ and $y_r$, $\PP_{G,m}$ denotes the conditional distribution of $X_m$ given the values of all remaining spins $\bX_{-m}$, and we have denoted the total variation between two measures by $D_{\TV}$. For a vertex $m$ let $\cN_{m}$ denote the set neighbors of $m$ in the graph $G$. We have
$$
\PP_{G,m}(X = x | \bX_{-m} = \xb) = g\bigg(\theta x \sum_{j \in \cN_{m} }x_j\bigg) 
$$
for $x \in \{\pm 1\}$ and $g(z) = e^z/(e^{-z} + e^z)$. It therefore immediately follows that $C_{r m} = 0$ if $r \not \in \cN_m$. Suppose $r \in \cN_m$. Let $S_{r} = \sum_{j \in \cN_{m}\setminus\{r\}} x_j = \sum_{j \in \cN_{m}\setminus\{r\}} y_j$. We have
\begin{align*}
C_{rm} & = \sup_{S_r}|g\bigr(\theta(S_r + 1)\bigr) - g\bigr(\theta(S_r - 1)\bigr) | = \sup_{S_r} \bigg|\frac{e^{2\theta} - e^{-2\theta}}{e^{-2 S_{r}\theta} + e^{2 S_{r}\theta} + e^{2\theta} + e^{-2\theta}}\bigg| \\
& \leq \frac{e^{2\theta} - e^{-2\theta}}{2+ e^{2\theta} + e^{-2\theta}} = \tanh(\theta).
\end{align*}
We conclude that Dobrushin's interaction matrix is $\Cb = \Ab_G \tanh(\theta)$, where $\Ab_G$ is the adjacency matrix of the graph $G$. Notice that the matrix $\Cb$ satisfies Dobrushin's condition
$$
\lim_{l \rightarrow \infty}\|\Cb^l\|_1 = 0,
$$
since $\|\Cb\|_1 = \|\Ab_G\|_1 \tanh(\theta) < 1$, and therefore $\displaystyle \lim_{l \rightarrow \infty} \|\Cb^l\|_1 \leq \lim_{l \rightarrow \infty}\|\Cb\|^l_1 = 0$. Next we will find the discrepancy between the averaged marginal conditional measures $\PP_G$ and $\PP_{G'}$. We define
$$
b_m := \sum_{\xb \in \{\pm 1\}^{d-1}} D_{\TV}(\PP_{G,m}(\cdot | \bX_{-m} = \xb), \PP_{G',m}(\cdot | \bX_{-m} = \xb)) \PP_{G'}(\bX_{-m} = \xb).
$$
Since $G$ and $G'$ differ only on the edge $(u,v)$ we have that $b_m = 0$ unless $m = u$ or $m = v$. Let $m = u$. Put $S = \sum_{j \in \cN_m} x_j$. We have
\begin{align*}
\MoveEqLeft |\PP_{G,u}(X=x | \bX_{-u} = \xb) - \PP_{G',u}(X =x | \bX_{-u} = \xb)| \\
& = |g\bigr(\theta x S\bigr) - g\bigr(\theta x(S - x_v)\bigr)|\\
& = \frac{e^{\theta } - e^{-\theta }}{e^{2\theta S - \theta x_v }  + e^{-2\theta S + \theta x_v }  +e^{-\theta } + e^{\theta } }\\
& \leq \frac{\tanh(\theta)}{1/\cosh(\theta) + 1} \leq \tanh(\theta).
\end{align*}
Therefore we conclude that $b_u \leq \tanh(\theta)$ and hence by symmetry the same inequality also holds for $b_v$. 
Since we are interested in bounded the correlation of $X_k$ and $X_\ell$ under the two measures we define the comparison function $f : \{\pm 1\}^d \mapsto \RR$ 
$$
f(\bX) := X_k X_\ell.
$$
and the oscillation of $f$ at site $r$:
$$
\delta_r(f) := \sup \{|f(\xb) - f(\yb)| : \xb = \yb \mbox{ off } r\}.
$$
Note that $\delta_{r}(f) = 0$ if $r \not \in \{k,\ell\}$ and $\delta_k(f) = \delta_\ell(f) = 2$. Define the matrix
$$
\Db := \sum_{l = 0}^{\infty} \bigr(\tanh(\theta)\Ab_G\bigr)^l,
$$
and let $\bb$ be the vector with entries $b_m$. According to Dobrushin's comparison theorem we have
\begin{align}\label{dobrushin}
\EE_G X_k X_\ell - \EE_{G'} X_k X_\ell \leq 2 ([\bb^\top \Db]_{k} + [\bb^\top\Db]_{\ell}) .
\end{align}
Hence (\ref{dobrushin}) is equivalent to
$$
\EE_G X_k X_\ell - \EE_{G'} X_k X_\ell \leq 2 \sum_{l = 0}^{\infty} \bigr(\tanh(\theta)\bigr)^{l + 1} \bigr([\Ab_G^l]_{uk} + [\Ab_G^l]_{vk}  + [\Ab_G^l]_{u\ell} + [\Ab_G^l]_{v\ell}\bigr),
$$
which is what we wanted to show. 
\end{proof}

%


\begin{proof}[Proof of Lemma \ref{biclique:lemma}] By Griffith's inequality (Theorem \ref{FKG:ineq}) pruning edges reduces the correlations. Hence we may assume without loss of generality that $G$ is an $l \times r$-biclique. We first show the case when $r \geq 3$. By a direct calculation we have
$$
\frac{\PP_G(X_k X_\ell = 1)}{\PP_G(X_k X_\ell = -1)} = \frac{\sum_{m = 0}^{r-2} \sum_{j = 0}^{l} {r-2 \choose m} {l \choose j} \exp[\theta(r - 2m)(l - 2j)]}{\sum_{m = 0}^{r-2} \sum_{j = 0}^{l} {r-2 \choose m} {l \choose j} \exp[\theta(r - 2m - 2)(l - 2j)]}.
$$
To this end note that
\begin{align*}
\textstyle \sum_{j = 0}^{l} {l \choose j} \exp[\theta(r - 2m)(l - 2j)]  = 2^{l} \cosh(\theta(r - 2m))^{l}.
\end{align*}
Similarly,
\begin{align*}
\textstyle \sum_{j = 0}^{l} {l \choose j} \exp[\theta(r - 2m - 2)(l - 2j)] = 2^{l} \cosh(\theta(r - 2m - 2))^{l}.
\end{align*}
Hence we obtain the identity
\begin{align}\label{biclique:identity}
\frac{\PP_G(X_k X_\ell = 1)}{\PP_G(X_k X_\ell = -1)} = \frac{\sum_{m= 0 }^{r-2} {r-2 \choose m} \cosh(\theta(r - 2m))^{l}}{\sum_{m= 0 }^{r-2} {r-2 \choose m} \cosh(\theta(r - 2m - 2))^{l}}.
\end{align}
To obtain a bound on (\ref{biclique:identity}) we will search for the $\argmax$ of the denominator. We will first argue that if $m^*$ is the $\argmax$ of the denominator them $m^*$ satisfies $0 \leq m^* \leq \frac{r-2}{e}$. Recall that $\cosh$ is an even function and hence by symmetry ($m \leftrightarrow r- 2 - m$):
$$
{r-2 \choose m} = {r-2 \choose r - 2 - m}, ~~~ \cosh(\theta(r - 2m - 2)) = \cosh(\theta(r - 2(r - 2 - m) - 2)).
$$ 
Therefore, to find a maximum in the denominator we only need to focus on terms satisfying $m \leq  \lfloor (r - 2)/2\rfloor$. To show that $0 \leq m^* \leq \frac{r-2}{e}$, we begin by comparing the term at $m = 0$ with any term at $m$ satisfying $(r-2)/e < m \leq \lceil(r - 2)/2\rceil$, i.e., we will compare $\cosh(\theta(r - 2))^{l}$ to ${r-2 \choose m} \cosh(\theta(r - 2m - 2))^{l}$. First we record a well known bound on the binomial coefficients
$$
{r-2 \choose m} \leq \Big(\frac{(r-2)e}{m}\Big)^m. 
$$
We will now argue that
\begin{align}\label{technical:ineq}
\frac{\cosh(\theta(r - 2))}{\cosh(\theta(r - 2m - 2))} \geq \exp(m \theta ).
\end{align}
Direct calculation shows that the above is equivalent to
\begin{align}\label{second:technical:cond}
\exp(\theta(2r - 4 - 4m)) \geq \frac{1 - \exp(-3m \theta)}{\exp(m \theta ) - 1}.
\end{align}
Note that the function $\frac{1 - e^{-3x}}{e^x - 1}$ is decreasing (its derivative is $-e^{-3x}(2 e^x + e^{2x} + 3) < 0$), and hence reaches its maximum at small values of $\theta m$. We have:
$$
\theta m > \frac{3}{r-2}\frac{r-2}{e} = \frac{3}{e}. 
$$
It is simple to check that $\frac{1 - e^{-9/e}}{e^{3/e} - 1} < 1$. Meanwhile since $2r - 4 - 4m \geq 0$ the left hand side of (\ref{second:technical:cond}) is at least $1$ and consequently (\ref{technical:ineq}) holds. Therefore when 
$$
 \Big(\frac{(r-2)e}{m}\Big)^m < \exp(\theta m l),
$$
i.e., when $m > \frac{r-2}{e}$ (using $\theta l > 2$) we have 
$$\cosh(\theta(r - 2))^{l} \geq {r-2 \choose m} \cosh(\theta(r - 2m - 2))^{l}.$$
Hence the maximum is reached at $m^*$ such that $0 \leq m^* \leq \frac{r-2}{e}$. We therefore have the bound:
$$
\textstyle \frac{\PP_G(X_k X_\ell = 1)}{\PP_G(X_k X_\ell = -1)} \geq \frac{ {r-2 \choose m^*} \cosh(\theta(r - 2m^*))^{l}}{(r -1){r-2 \choose m^*} \cosh(\theta(r- 2m^* - 2))^{l}} = \frac{ \cosh(\theta(r - 2m^*))^{l}}{(r -1) \cosh(\theta(r - 2m^* - 2))^{l}}.
$$
Finally we will argue that
$$
\frac{ \cosh(\theta(r - 2m^*))}{\cosh(\theta(r - 2m^* - 2))}\geq \exp(\theta). 
$$
Similarly to before we need to verify:
\begin{align}\label{third:technical:cond}
\exp(\theta(2 r - 4 - 4m^*)) \geq \frac{1 - \exp(-3 \theta)}{\exp( \theta ) - 1},
\end{align}
We have $\frac{1 - \exp(-3 \theta)}{\exp( \theta ) - 1} \leq \lim_{x \rightarrow 0} \frac{1 - \exp(-3 x)}{\exp( x ) - 1} = 3$, and thus (\ref{third:technical:cond}) is implied when $\exp(\theta(2 r - 4 - 4m^*)) \geq 3$. Using the fact that $m^* \leq \frac{r-2}{e}$ we have:
$$
\exp(\theta(2r - 4 - 4m^*)) \geq \exp(\theta(2 - 4/e)(r - 2)) \geq \exp(6 - 12/e) > 3,
$$
where we used the fact that $\theta (r-2) \geq 3$. Compiling these results yields:
\begin{align}\label{prob:ratio:ineq}
\frac{\PP_G(X_k X_\ell = 1)}{\PP_G(X_k X_\ell = -1)} \geq \frac{\exp(\theta l)}{r-1}
\end{align}
Hence
$$
\EE_{G} X_k X_\ell = \PP_G(X_k X_\ell = 1) - \PP_G(X_k X_\ell = -1) \geq 1 - \frac{2(r-1)}{\exp(\theta l) + r - 1},
$$
as we claimed. This completes the proof when $r > 2$. In the case when $r = 2$, a direct calculation shows
$$
\frac{\PP_G(X_k X_\ell = 1)}{\PP_G(X_k X_\ell = -1)} = \cosh(2 \theta)^{l}.
$$
Clearly then when $\theta \geq \log 2$, we have $\cosh(2 \theta) \geq \exp(\theta)$, and the proof can proceed in the same way as before. With this the result is complete. 
\end{proof}

\begin{remark} Theorem \ref{biclques:theorem} demonstrates that some strongly monotone properties have upper bound signal strength limitations. In fact it is evident from the proof that the scaling on $\theta$ in (\ref{theta:scaling:upper:bound:max:monotone}) can be slightly improved. It turns out that under the same conditions, it suffices that if 
$$
\frac{2 \tanh(\theta)(r - 1)}{\exp(\theta l) + r - 1} = o\bigg(\log \frac{d}{|l + r|}\bigg),
$$
holds instead of (\ref{theta:scaling:upper:bound:max:monotone}) we still have that the minimax risk goes to $1$ asymptotically.
\end{remark}

\begin{proof}[Proof of Example \ref{conn:example} cont'd] 

Fix $0 < \alpha < 1$ such that $\log d \geq n (\log 2)/\alpha $. For simplicity suppose that the quantities $\tilde d = d/2$, $\tilde d^{\alpha}$ and $\tilde d^{1-\alpha}$ are integer. If they are not, one just simply needs to round them and the proof goes through.  We will construct a graph $G_0$ under the null hypothesis consisting of two equal paths with $\tilde d$ vertices each (see Figure \ref{path:graph:conn:under:h0}). Label the vertices on the upper path as $1,2,\ldots, \tilde d$ and on the lower path as $\tilde d + 1, \tilde d + 2, \ldots, 2 \tilde d$. To construct alternative graphs take the set of edges $\{(1,\tilde d + 1), (\tilde d^{1 - \alpha} + 1, \tilde d + \tilde d^{1 - \alpha} + 1),  (2\tilde d^{1-\alpha} + 1, \tilde d + 2\tilde d^{1-\alpha} + 1), \ldots, ((\tilde d^{\alpha} - 1)\tilde d^{1-\alpha} + 1, \tilde d + (\tilde d^{\alpha} - 1)\tilde d^{1-\alpha} + 1)\}$, which consists of $\tilde d^{\alpha}$ edges. Let $\PP_0$ be the mixture of the uniform $\theta$ signal Ising model with graph $G_0$ and $\overline \PP$ be the uniform mixture of adding any of the aforementioned edges to $G_0$. By Lemma \ref{lemma:chi:sq} we have the following bounds:

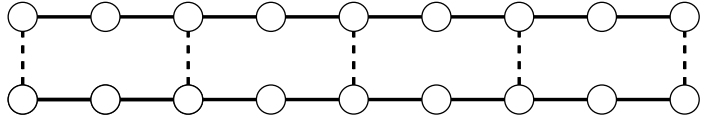
\begin{figure}[H] \centering
\begin{tikzpicture}[scale=.55]
\SetVertexNormal[Shape      = circle,
                  FillColor  = white,
                  MinSize = 11pt,
                  InnerSep=0pt,
                  LineWidth = .5pt]
   \SetVertexNoLabel
   \tikzset{LabelStyle/.style = {below, fill = white, text = black, fill opacity=0, text opacity = 1}}
   \tikzset{EdgeStyle/.style= {thin,
                                double          = black,
                                double distance = .5pt}}
    \begin{scope}\grPath[prefix=a,RA=2]{3}\end{scope}
   \tikzset{EdgeStyle/.style= { thin,
                                double          = black,
                                double distance = .5pt}}
     \begin{scope}[shift={(0,2)}]\grPath[prefix=b,RA=2]{9}\end{scope}
         \begin{scope}\grPath[,RA=2]{9}\end{scope}
         
            \tikzset{EdgeStyle/.style= {dashed,thin,
                                double          = black,
                                double distance = .5pt}}

          \Edge(a0)(b0)
          \Edge(a2)(b2)
          \Edge(a4)(b4)
          \Edge(a6)(b6)
          \Edge(a8)(b8)
\end{tikzpicture}\caption{A graph under the null hypothesis with solid edges; graphs under the alternative hypothesis are produced by adding any of the dashed edges to the solid edges.}\label{path:graph:conn:under:h0}
\end{figure}
\vspace{-1.2cm}
\begin{align*}
D_{\chi^2}(\overline \PP^{\otimes n}, \PP^{\otimes n}_0) &= \frac{1}{\tilde d^{2\alpha}}\sum_{k,j = 1}^{\tilde d^\alpha}\bigg(1 + \tanh(\theta) [\EE_{j} X_{u_k} X_{v_k} - \EE_{0} X_{u_k} X_{v_k}]\bigg)^n - 1\\
& \leq \frac{(1 + \tanh^2(\theta))^n}{\tilde d^\alpha} + (1 + [\tanh(\theta)]^{2\tilde d^{1-\alpha} + 2})^{n} - 1
\end{align*}
where the last inequality follows by Lemma \ref{chain:graph} and Proposition \ref{restriction:prop}, and the fact that any two non-coinciding edges are at least $2\tilde d^{1-\alpha} + 1$ apart. Using $\log d \geq C n$ (for $C \geq (\log 2)/\alpha$) we continue the bound
$$
D_{\chi^2}(\overline \PP^{\otimes n}, \PP^{\otimes n}_0) \leq \frac{2^n}{\tilde d^\alpha} + (1 + [\tanh(\theta)]^{2\exp((1-\alpha)Cn)})^{n} - 1.
$$
Supposing that $C \geq \log2/\alpha$ we have $\frac{2^n}{\tilde d^\alpha} \rightarrow 0$. Next we show that the second term is $1 + o(1)$ unless $\tanh(\theta) = 1$ asymptotically. Suppose that for any $\varepsilon > 0$ we have $\tanh(\theta) < 1 - \varepsilon$ asymptotically. We have:
$$
(1 + [\tanh(\theta)]^{2\exp((1-\alpha)Cn)})^{n} - 1 \leq \exp(n[\tanh(\theta)]^{2\exp((1-\alpha)Cn)}) - 1 = o(1).
$$
Thus we complete the proof. 
\end{proof}

\section{Correlation Screening for Ferromagnets}\label{corr:screening:proofs}

\begin{lemma}\label{simple:concentration:lemma} Given $n$ samples from an Ising model satisfying (\ref{ising:model:measure:anti:full:generality}). We have that
\begin{align}\label{union:bound:cov:conc}
\PP_{\theta, G_{\wb}}\bigg(\max_{u,v\in [d]}|\hat \EE X_u X_v - \EE_{\theta, G_{\wb}} X_u X_v| \geq \varepsilon\bigg) \leq 2 {d \choose 2} e^{-n\varepsilon^2/2}.
\end{align}
Note that since (\ref{ferromagnet:def}) is a special case of (\ref{ising:model:measure:anti:full:generality}), the same conclusion is also valid for ferromagnets satisfying (\ref{ferromagnet:def}).
\end{lemma}

\begin{proof}[Proof of Lemma \ref{simple:concentration:lemma}] This is a direct consequence of Hoeffding's inequality \citep{boucheron2013concentration} and the union bound.
\end{proof}

\begin{proof}[Proof of Theorem \ref{generic:property:test:thm}] Recall the definitions of $\cT$ (\ref{min:corr:under:alt}) and $\tau$ (\ref{tau:univ:thresh}). Using Lemma \ref{simple:concentration:lemma} we can guarantee that if $G_{\wb} \in \cG_1 \cap \cR$
$$
\min_{e \in \hat G} M_e \geq \cT - \tau \geq \underline{\cT} - \tau,
$$
with probability at least $1 - \delta$. This immediately shows that
$$
\sup_{G_{\wb} \in \cG_1 \cap \cR} \PP_{\theta, G_{\wb}}(\psi = 0) \leq \delta.
$$
Next since $\hat G \in \cW$, we have that if $\cP(G) = 0$
$$
\min_{e \in \hat G} M_e \leq \cQ + \tau,
$$
on an event of probability at least $1- \delta$. When (\ref{abstract:corr:algo:assumption}) holds we have
$$
\min_{e \in \hat G} M_e < \underline{\cT} - \tau,
$$
and thus Algorithm \ref{generic:property:test} will return $\psi = 0$. This completes the proof. 
\end{proof}

\begin{lemma}[Generic No Edge Correlation Upper Bound]\label{generic:upper:bound:lemma} Assume we have a ferromagnetic Ising model as specified by (\ref{ferromagnet:def}) where all weights $w_{uv}$ satisfy $0 < w_{uv} \leq \Theta/\theta$. Then if the maximum degree of $G$ is bounded by $s \geq 3$, for any two vertices $k$ and $\ell$ such that $(k,\ell) \not \in E(G)$ we have
\begin{align}\label{generic:no:edge:bound}
\textstyle \EE_{\theta, G_{\wb}} X_k X_\ell \leq \frac{\cosh(2s\Theta) + 2se^{-2(s-1)\Theta}(\cosh(2(s-1)\Theta) - \cosh(2\Theta)) - 1}{\cosh(2s\Theta) + 2se^{-2(s-1)\Theta}(\cosh(2(s-1)\Theta) + \cosh(2\Theta)) + 1} = \frac{R(s,\Theta) - 1}{R(s,\Theta) + 1}.
\end{align}
If $s = 2$ we have 
$$
\textstyle \EE_{\theta, G_{\wb}} X_k X_\ell \leq \frac{\cosh(4\Theta) - 1}{\cosh(4\Theta) + 3}.
$$
Finally if $s = 1$, $\EE_{\theta, G_{\wb}} X_k X_\ell  = 0$.
\end{lemma}

\begin{proof}[Proof of Lemma \ref{generic:upper:bound:lemma}] First, by Griffith's inequality (Theorem \ref{FKG:ineq}) it follows that since all $0\leq w_{uv} \leq\Theta /\theta$ the correlation $\EE_{\theta, G_{\wb}} X_k X_\ell$ will only increase if we were to set all $w_{uv} = \Theta/\theta$. Hence we may assume that all $w_{uv} = \Theta/\theta$ for all $u,v$. 

Let $\cN_k$ and $\cN_\ell$ denote the neighbors of vertices $k$ and $\ell$ respectively. Again, due to Griffith's inequality adding edges increases correlations so we can assume that $|\cN_k| = |\cN_\ell| = s$. 
We have
\begin{align}\label{master:prob:ration:id}
& \textstyle \frac{ \PP_{\theta, G_{\wb}}(X_k X_\ell = 1)}{\PP_{\theta, G_{\wb}}(X_k X_\ell = -1)}=  \\
& \textstyle \frac{\sum_{\xb \in \{\pm 1\}^d}\sum_{\xi \in \{\pm 1\}} \exp(\xi\Theta (\sum_{u \in \cN_k} x_u + \sum_{u \in \cN_\ell} x_u) + \Theta\sum_{(u,v) \in E(G),  k,\ell \not\in \{u, v\}} x_ux_v)}{\sum_{\xb \in \{\pm 1\}^d}\sum_{\xi \in \{\pm 1\}} \exp(\xi\Theta (\sum_{u \in \cN_k} x_u - \sum_{u \in \cN_\ell} x_u) + \Theta \sum_{(u,v) \in E(G),  k,\ell \not\in \{u, v\}} x_ux_v)}. \nonumber
\end{align}
Let $\cN = \cN_k \cap \cN_\ell$ and suppose $|\cN| = m$ for $0 \leq m \leq s$. Put $\overline \cN_k := \cN_k \setminus \cN$ and similarly let $\overline \cN_\ell := \cN_\ell \setminus \cN$. For a vector $\xb \in \RR^d$ and a set $S \subseteq [d]$ let $\xb_{S} = (x_i)_{i \in S}$. For $\xi \in \{\pm 1\}$ define the sets 
\begin{align*}
 S_1 & := \textstyle \{\xb \in \{\pm 1\}^{d} ~|~ |\sum_{u \in \cN_k} x_u + \sum_{u \in \cN_\ell} x_u| = 2s\},\\ 
  S_2 & := \textstyle \{\xb \in \{\pm 1\}^{d} ~|~ |\sum_{u \in \cN_k} x_u + \sum_{u \in \cN_\ell} x_u| \leq 2(s-2) \},
\end{align*} 
and let $S_3 := \{\pm 1\}^{d} \setminus (S_1 \cup S_2)$. We now record several identities for the three sets. For any $\xb \in S_1$ we have
\begin{align*}
\textstyle \sum\limits_{\xi \in \{\pm 1\}}\exp(\xi\Theta (\sum\limits_{u \in \cN_k} x_u + \sum\limits_{u \in \cN_\ell} x_u)) & = 2 \cosh(2s\Theta), \\
\textstyle \sum\limits_{\xi \in \{\pm 1\}}\exp(\xi\Theta (\sum\limits_{u \in \cN_k} x_u - \sum\limits_{u \in \cN_\ell} x_u)) & = 2.
\end{align*}
The first identity holds by the definition of $S_1$, while the second holds since if $\xb \in S_1$ it follows that $\sum_{u \in \cN_k} x_u - \sum_{u \in \cN_\ell} x_u = 0$. For any $\xb \in S_2$:
\begin{align*}
\textstyle \sum\limits_{\xi \in \{\pm 1\}}\exp(\xi\Theta (\sum\limits_{u \in \cN_k} x_u + \sum\limits_{u \in \cN_\ell} x_u)) & \leq 2 \cosh(2(s-2)\Theta),\\
\textstyle  \sum\limits_{\xi \in \{\pm 1\}}\exp(\xi\Theta (\sum\limits_{u \in \cN_k} x_u - \sum\limits_{u \in \cN_\ell} x_u)) & \geq 2.
\end{align*}
The first inequality follows by the definition of $S_2$, and the second is true since since for any $x \in \RR$ $e^{x} + e^{-x} \geq 2$. Finally for any $\xb \in S_3$:
\begin{align*}
\textstyle \sum\limits_{\xi \in \{\pm 1\}}\exp(\xi\Theta (\sum\limits_{u \in \cN_k} x_u + \sum\limits_{u \in \cN_\ell} x_u)) & = 2 \cosh(2(s-1)\Theta), \\
\textstyle \sum\limits_{\xi \in \{\pm 1\}}\exp(\xi\Theta (\sum\limits_{u \in \cN_k} x_u - \sum\limits_{u \in \cN_\ell} x_u)) & \geq 2\cosh(2\Theta).
\end{align*}
To see why the first identity holds, note that the sum $\sum_{u \in \cN_k} x_u + \sum_{u \in \cN_\ell} x_u$ contains $2s$ odd terms (equal to $\pm 1$), and is therefore even. Hence since $2(s-2) < \bigr|\sum_{u \in \cN_k} x_u + \sum_{u \in \cN_\ell} x_u\bigr| < 2s$ we must have $\bigr|\sum_{u \in \cN_k} x_u + \sum_{u \in \cN_\ell} x_u\bigr| = 2(s-1)$. To show why the inequality holds suppose the contrary, i.e., suppose $\bigr|\sum_{u \in \cN_k} x_u - \sum_{u \in \cN_\ell} x_u\bigr| \leq 1$. Since $\sum_{u \in \cN_k} x_u - \sum_{u \in \cN_\ell} x_u = \sum_{u \in \overline \cN_k} x_u - \sum_{u \in \overline \cN_\ell} x_u$ and $|\overline \cN_k \cup \overline \cN_\ell| = 2(s - m)$, the sum $\sum_{u \in \overline \cN_k} x_u - \sum_{u \in \overline \cN_\ell} x_u$ is even. Hence the only possibility is $\sum_{u \in \overline \cN_k} x_u = \sum_{u \in \overline \cN_\ell} x_u$, and therefore $2 \bigr| \sum_{\cN_k} x_u\bigr| = 2(s-1)$. The latter is impossible since $\sum_{\cN_k} x_u \equiv s~ (\operatorname{mod} 2)$. 

We first show the result when $s \geq 3$. Let $\mathbf{1} \in \RR^d$ denote a vector of $1$'s. Observe the following inclusions
\begin{align*}
\textstyle \overline S_1:= & \{\xb \in \{\pm 1\}^{d} ~|~ \exists\xi \in \{\pm 1\}, \xb_{\cN_k \cup \cN_\ell} = \xi \mathbf{1}_{\cN_k \cup \cN_\ell}\} = S_1 ,\\ 
\textstyle \overline S_2 := & \{\xb \in \{\pm 1\}^{d} ~|~  \exists u \in \cN: \xb_{\cN_k \cup \cN_\ell \setminus \{u\}} = -x_u \mathbf{1}_{\cN_k \cup \cN_\ell \setminus \{u\}}\} \subseteq S_2,\\
\textstyle \overline S_3 := & \{\xb \in \{\pm 1\}^{d} ~|~ \exists u \in \overline \cN_k \cup \overline \cN_\ell: \xb_{\cN_k \cup \cN_\ell \setminus \{u\}} = -x_u \mathbf{1}_{\cN_k \cup \cN_\ell \setminus \{u\}} \} \subseteq S_3.
\end{align*}
Notice that when $s \geq 3$: $|\overline S_2| \geq m|S_1|$ and $|\overline S_3| \geq 2(s - m)|S_1|$. For $i \in [3]$ define the three sums $\Sigma_i :=\sum_{\xb \in S_i} \exp(\Theta\sum_{(u,v) \in E(G),  k,\ell \not\in \{u, v\}} x_ux_v)$.
Using (\ref{master:prob:ration:id}) and the identities we recorded above we conclude 
\begin{align}\label{first:ineq:prob:ratio}
\textstyle \frac{ \PP_{\theta, G_{\wb}}(X_k X_\ell = 1)}{\PP_{\theta, G_{\wb}}(X_k X_\ell = -1)} \leq \frac{\cosh(2s\Theta) \Sigma_1 + \cosh(2(s-2)\Theta)\Sigma_2 + \cosh(2(s-1)\Theta)\Sigma_3}{\Sigma_1 + \Sigma_2 + \cosh(2\Theta)\Sigma_3}.
\end{align}
We now use the following elementary inequality which can be checked via cross multiplication: for positive real numbers $a,b,c,d,x,y >0 $ such that $\frac{a}{c} \geq \frac{b}{d}$ and $x \geq y$ it holds that
\begin{align}\label{elementary:fraq:ineq}
\frac{a + b x}{c + d x} \leq \frac{a + b y}{c + d y}.
\end{align}
Taking $x = \infty, y = 1$ in (\ref{elementary:fraq:ineq}) implies that
$$
\textstyle \frac{\cosh(2s\Theta) \Sigma_1 + \cosh(2(s-1)\Theta)\Sigma_3}{\Sigma_1 +  \cosh(2\Theta)\Sigma_3} \geq \cosh(2s\Theta) \wedge \frac{\cosh(2(s-1)\Theta)}{\cosh(2\Theta)} \geq \cosh(2(s-2)\Theta),
$$
where the last inequality follows by simple algebra and the fact that $\cosh(x)$ is increasing for $x \geq 0$. We now record the inequality
$$
\textstyle \frac{\Sigma_2}{\Sigma_1} \geq \frac{\sum_{\xb \in \overline S_2} \exp(\Theta\sum_{(u,v) \in E(G),  k,\ell \not\in \{u, v\}} x_ux_v)}{ \sum_{\xb \in S_1} \exp(\Theta\sum_{(u,v) \in E(G),  k,\ell \not\in \{u, v\}} x_ux_v)} \geq m e^{-2(s-2)\Theta},
$$
which follows by the fact that $|\overline S_2| \geq m|S_1|$ and that all vertices $u \in \cN$ are connected to at most $s-2$ vertices in the set $[d] \setminus \{k, \ell\}$. Combining the last two inequalities with (\ref{first:ineq:prob:ratio}) and (\ref{elementary:fraq:ineq}) we obtain 
\begin{align}\label{second:ineq:prob:ratio}
\textstyle \frac{ \PP_{\theta, G_{\wb}}(X_k X_\ell = 1)}{\PP_{\theta, G_{\wb}}(X_k X_\ell = -1)} \leq \frac{ \cosh(2s\Theta)  + \cosh(2(s-2)\Theta)m \exp(-2(s-2)\Theta) + \cosh(2(s-1)\Theta)\frac{\Sigma_3}{\Sigma_1} }{1 + m \exp(-2(s-2)\Theta) + \cosh(2\Theta)\frac{\Sigma_3}{\Sigma_1} }.
\end{align}
We now consider two cases.\\
\noindent {\bf Case I.}  Suppose
\begin{align}\label{caseI:working:assump}
\frac{\cosh(2s\Theta)  + \cosh(2(s-2)\Theta)m e^{-2(s-2)\Theta} }{1 + m e^{-2(s-2)\Theta} } \leq \frac{\cosh(2(s-1)\Theta) }{ \cosh(2\Theta)}.
\end{align}
Taking $y = 0$ in (\ref{elementary:fraq:ineq}), and using (\ref{second:ineq:prob:ratio}) yields
\begin{align}\label{caseI:ineq:no:edge:corr}
\frac{ \PP_{\theta, G_{\wb}}(X_k X_\ell = 1)}{\PP_{\theta, G_{\wb}}(X_k X_\ell = -1)} \leq \frac{\cosh(2(s-1)\Theta) }{ \cosh(2\Theta)}.
\end{align}
\noindent {\bf Case II.} Assume that (\ref{caseI:working:assump}) holds in the opposite direction. We have
$$
\textstyle \frac{\Sigma_3}{\Sigma_1} \geq \frac{\sum_{\xb \in \overline S_3} \exp(\Theta\sum_{(u,v) \in E(G),  k,\ell \not\in \{u, v\}} x_ux_v)}{\sum_{\xb \in S_1} \exp(\Theta\sum_{(u,v) \in E(G),  k,\ell \not\in \{u, v\}} x_ux_v)} \geq 2(s-m) e^{-2(s-1)\Theta},
$$
where we used that $|\overline S_3| \geq 2(s-m)|S_1|$ and that for any $\xb \in \overline S_3$ the vertex $u \in \overline \cN_k \cup \overline \cN_\ell$ is connected to at most $s-1$ vertices in the sed $[d] \setminus \{k, \ell\}$. (\ref{elementary:fraq:ineq}) and (\ref{second:ineq:prob:ratio}) imply
$$
\textstyle \frac{ \PP_{\theta, G_{\wb}}(X_k X_\ell = 1)}{\PP_{\theta, G_{\wb}}(X_k X_\ell = -1)} \leq \frac{\cosh(2s\Theta)  + \cosh(2(s-2)\Theta)m e^{-2(s-2)\Theta} + \cosh(2(s-1)\Theta)2(s-m) e^{-2(s-1)\Theta}}{1 + me^{-2(s-2)\Theta} + \cosh(2\Theta)2(s-m) e^{-2(s-1)\Theta}}.
$$
Taking the supremum over $0 \leq m \leq s$ on the RHS shows that the maximum is reached at $m \equiv 0$. The latter can be verified via comparing consecutive values of $m$ and arguing that the RHS is decreasing in $m$; we omit this lengthly calculation. Finally a simple comparison between the RHS of the last bound evaluated at $m = 0$ and the RHS of (\ref{caseI:ineq:no:edge:corr}) shows that the last bound is larger. Putting everything together we conclude that for $s \geq 3$
$$
\textstyle \frac{ \PP_{\theta, G_{\wb}}(X_k X_\ell = 1)}{\PP_{\theta, G_{\wb}}(X_k X_\ell = -1)} \leq \frac{\cosh(2s\Theta) + 2se^{-2(s-1)\Theta} \cosh(2(s-1)\Theta)}{1 + 2se^{-2(s-1)\Theta}\cosh(2\Theta)}.
$$

For the special case $s = 2$, the major difference is that when $m = 2$, $|\overline S_2| \geq |S_1|$ and not $|\overline S_2| \geq 2|S_1|$ as before. Using the same ideas as in the proof of the case $s \geq 3$ one can show that when $s \equiv 2$:
$$
\textstyle \frac{ \PP_{\theta, G_{\wb}}(X_k X_\ell = 1)}{\PP_{\theta, G_{\wb}}(X_k X_\ell = -1)} \leq  \frac{\cosh(4\Theta) + 1}{2}.
$$
The last two inequalities combined with the fact $$\EE_{\theta, G_{\wb}} X_k X_\ell = 1 - 2 \PP_{\theta, G_{\wb}}(X_k X_\ell = -1)$$ complete the proof.
\end{proof}

The following simple lemma gives a sharper correlation bound when no edge is present, in a regime where the maximum parameter $\Theta$ is sufficiently small. 

\begin{lemma}[High Temperature No Edge Correlation Upper Bound] \label{hight:no:edge:corr} Let the assumptions of Lemma \ref{generic:upper:bound:lemma} hold, and let additionally $(s-1) \tanh(\Theta) < 1$. Then we have
$$
\textstyle \EE_{\theta, G_{\wb}} X_k X_\ell \leq \frac{s\tanh^2(\Theta)}{1 - (s-1) \tanh(\Theta)}.
$$
\end{lemma}

\begin{proof}[Proof of Lemma \ref{hight:no:edge:corr}] As in the proof of Lemma \ref{generic:upper:bound:lemma}, we may assume that all $w_{uv} = \Theta/\theta$. Next, notice that each path $P$ connecting $k$ and $\ell$ has length $m \geq 2$. Let $G_P$ be the restriction of the graph $G$ to the path graph $P$. By Lemma \ref{chain:graph} we know that $E_{\theta, G_P}X_iX_j = \tanh^m(\Theta)$, and furthermore there are at most $s (s-1)^{m-2}$ such paths, due to the fact that each vertex has at most $s$ neighbors. The conclusion now follows from a result by \cite{fisher1967critical} (stated under Lemma \ref{fisher:bound}) which bounds the correlation between two nodes with a sum over the correlations of all path graphs between the two vertices, and the formula for summing converging geometric progressions.
\end{proof}

\begin{proof}[Proof of Proposition \ref{generic:property:test:cor}] Note that if $G_{\wb} \in \cG_0 \cap \cR$ for any  $G' \in \cW$ there exists an edge $e \in E(G')$ with a zero entry in $\wb$.
Fact i. now follows since by Lemma \ref{generic:upper:bound:lemma} and the definition of $R(s,\Theta)$ and we have for $s \geq 3$
 $$
\cQ \leq \frac{R(s,\Theta) - 1}{R(s,\Theta) + 1},
$$
and fact ii. follows by noting that when $(s-1) \tanh(\Theta) < 1$ by Lemma \ref{hight:no:edge:corr}
$$
\cQ \leq  \frac{s\tanh^2(\Theta)}{1 - (s-1) \tanh(\Theta)}.
$$
\end{proof}

\begin{proof}[Proof of Corollary \ref{conn:test:consistency}] We will now argue that if $G$ is under the alternative, i.e., if $G$ is connected, then for any spanning tree $T \subseteq G$ we have
\begin{align}\label{cov:inequality}
\min_{(u,v) \in T} \EE_{\theta, G_{\wb}} X_u X_v \geq \tanh(\theta).
\end{align}
By Griffith's inequality (Theorem \ref{FKG:ineq}) we have 
$$\min_{(u,v) \in T} \EE_{\theta, G_{\wb}} X_u X_v \geq \min_{(u,v) \in T} \EE_{\theta, T_{\wb}} X_u X_v,$$ 
where $T_{\wb}$ is the restriction of $G_\wb$ on $T$, i.e., setting all entries of $\wb$ outside of $E(T)$ to $0$ while retaining the values of all entries belonging to $T$. Using Proposition \ref{restriction:prop}, Lemma \ref{chain:graph} and the fact that on a tree $T$ the only path between vertices $u$ and $v$ connected by an edge $(u,v)$ is the edge we conclude that
$$
\EE_{\theta, T_{\wb}} X_u X_v = \tanh(\theta w_{uv}).
$$
Hence
$$
\min_{(u,v) \in T} \EE_{\theta, G_{\wb}} X_u X_v \geq \min_{(u,v) \in T} \EE_{\theta, T_{\wb}} X_u X_v = \min_{(u,v) \in T} \tanh(\theta w_{uv}) \geq \tanh(\theta),
$$
as we claimed. 
On the other hand, when $G$ is under the null for any tree $T$, we have
$$
\min_{(u,v) \in T} \EE_{\theta, G_{\wb}} X_u X_v = 0. 
$$
Hence an application of Theorem \ref{generic:property:test:thm} completes the proof. 
\end{proof}

\begin{proof}[Proof of Corollary \ref{cycle:test:consistency}] Similarly to the proof of Corollary \ref{conn:test:consistency} we can convince ourselves with the help of Griffith's inequality (see Theorem \ref{FKG:ineq}), Proposition \ref{restriction:prop} and Lemma \ref{chain:graph} that under the alternative there exists at least one cycle $C \subseteq G$ so that
$$
\min_{(u,v) \in C} \EE_{\theta, G_{\wb}} X_uX_v \geq \min_{(u,v) \in C} \EE_{\theta, C_{\wb}} X_uX_v \geq \min_{(u,v) \in C} \tanh(\theta w_{uv}) \geq \tanh(\theta),
$$
where $C_{\wb}$ is the restriction of $G_\wb$ on $C$. 
Next, note that under the null hypothesis for any given cycle $C$ we will have at least one edge $(u,v) \in C$ such that $w_{uv} = 0$. Since under the null the graph is a forest by Proposition \ref{restriction:prop} and Lemma \ref{chain:graph} we know that
$$
\min_{(u,v) \in C}\EE_{\theta, G_{\wb}} X_uX_v \leq \prod_{(k,l) \in \cP^{G}_{u \rightarrow v}} \tanh(\theta w_{kl}) \leq \tanh^2(\Theta), 
$$
where $\cP^G_{u \rightarrow v}$ is the direct path that connects $u$ to $v$. Applying Theorem \ref{generic:property:test:thm} completes the proof. 
\end{proof}

\begin{proof}[Proof of Corollary \ref{clique:test:consistency}] Take any model with a graph $G$ containing an $m$-clique. By Griffith's inequality (Theorem \ref{FKG:ineq}) we know that the correlation for any edge $(u,v) \in E(G)$: $\EE_{\theta, G_{\wb}} X_u X_v$ will only decrease if we set all parameters to their lower bound $\theta$, and if we remove edges so that we reduce the graph $G$ to an $m$-clique. We recognize that under such manipulation we obtain a Curie-Weiss model\footnotemark \footnotetext{i.e., a fully connected $m$-clique with inverse temperature $\theta$.} (albeit under non-standard parametrization) for $m$ of the variables, and remaining $d-m$ variables are independent and have Rademacher distributions. Hence we only need to calculate $\PP_{\theta, G}(X_kX_\ell = \xi)$\footnote{Here we omit the weights since they are all equal to $1$ on the graph $G$.} for $\xi \in \{\pm 1\}$ in an $m$-dimensional Curie-Weiss model.  This is a standard calculation which we nevertheless include for completeness.  We have
\begin{align}
 \PP_{\theta, G}(X_kX_\ell = 1) &=\textstyle \frac{ e^{\theta -  \theta (m-2)/2}}{Z_{m}(\theta)}S(2\theta, \frac{\theta}{2}),  \PP_{\theta,G}(X_kX_\ell = -1) = \frac{ e^{-\theta -  \theta (m-2)/2}}{Z_{m}(\theta)}S(0,\frac{\theta}{2}), \label{probs:def}
\end{align}
where $Z_m(\theta)$ is the partition constant of an $m$-dimensional Curie-Weiss magnet with inverse temperature $\theta$ and no external magnetic field, i.e.,
\begin{align}\label{Curie:Weiss:Magnet:Part}
Z_m(\theta) = \sum_{\xb \in \{\pm 1\}^m} \exp\bigr (\theta \sum_{u < v} x_u x_v\bigr),
\end{align}  
and the function $S(h,\mu) = S(-h,\mu)$ given by
$$
S(h, \mu) := \sum_{\xb \in \{\pm 1\}^{m-2}} \exp\bigg( h \sum_{u \in [m-2]} x_u + \mu \bigg(\sum_{u \in [m-2]} x_u\bigg)^2 \bigg).
$$
Clearly since $\PP_{\theta, G}(X_kX_\ell = 1) + \PP_{\theta, G}(X_kX_\ell = -1) = 1$, one can express $Z_{m}(\theta)$ in terms of  $S(2\theta, \theta/2)$ and $S(0, \theta/2)$. We set out to find a closed form expression of the more general quantity $S(h,\mu)$.
Using the identity
$$
e^{x^2} = \frac{1}{\sqrt{2\pi}} \int_{-\infty}^\infty e^{-y^2/2 + \sqrt{2}xy} d y,
$$
we can rewrite $S(h, \mu)$ as 
\begin{align}
S(h,\mu) & = \frac{1}{\sqrt{2\pi}}\sum_{\xb \in \{\pm 1\}^{m-2}} \int_{-\infty}^{\infty} e^{-y^2/2 + (y \sqrt{2 \mu} +h) \sum_{u \in [m-2]} x_u } dy \nonumber \\
& = 2^{m-2} \EE_{Z} \cosh^{m-2}(\sqrt{2\mu} Z + h).\label{Shmu:calc}
\end{align}
Putting (\ref{probs:def}) and (\ref{Shmu:calc}) together we conclude
$$
\EE_{\theta, G} X_k X_\ell = \PP_{\theta, G}(X_k X_\ell = 1) - \PP_{\theta, G}(X_k X_\ell = -1) = \frac{r(m,\theta) - 1}{r(m, \theta) + 1}.
$$
This implies that under the alternative
$$
\min_{(u,v) \in C} \EE_{\theta, G_{\wb}} X_u X_v \geq \frac{r(m,\theta) - 1}{r(m, \theta) + 1},
$$
for any $m$-clique $C \subseteq G$, which is what we wanted to show.  The remaining part of the Corollary is a direct consequence of Proposition \ref{generic:property:test:cor} and Theorem \ref{generic:property:test:thm}. 
%

\end{proof}

\begin{remark}\label{clique:testing:remark} In this remark we will first show a low temperature expansion of $\frac{r(m,\theta) - 1}{r(m, \theta) + 1}$. Note the equivalent formulation 
$$
\frac{r(m,\theta) - 1}{r(m, \theta) + 1} = \frac{d \log Z_m(\theta)/d \theta}{{m \choose 2}},
$$
where $Z_m(\theta)$ is defined in (\ref{Curie:Weiss:Magnet:Part}). Now observe that 
$$
Z_m(\theta) = 2 e^{{m \choose 2} \theta}\bigr(1 + m e^{-2(m-1)\theta} + {m \choose 2} e^{-4(m-2)\theta} + O(m^3 e^{-6(m-3) \theta})\bigr),
$$
where the first term corresponds to all spins being $1$ or $-1$, the second term corresponds to the $m$ terms where one spin is $-1$ and all remaining  spins are $1$ and vice versa, and so on. Hence
\begin{align*}
& \log Z_m(\theta) \\
& = \log 2 + {m \choose 2} \theta + \log\bigr(1 + m e^{-2(m-1)\theta} + {m \choose 2} e^{-4(m-2)\theta} + O(m^3 e^{-6(m-3) \theta})\bigr)\\
& =  \log 2 + {m \choose 2} \theta + m e^{-2(m-1)\theta} + {m \choose 2} e^{-4(m-2)\theta} + O(m^3 e^{-6(m-3) \theta})
\end{align*}
Under the assumption $m^2e^{-2m\theta} = o(1)$, we have the approximation
$$
\cT = \frac{r(m,\theta) - 1}{r(m, \theta) + 1} = 1 - 4 e^{-2(m-1)\theta} - 4 (m-2) e^{-4(m-2)\theta} + o(1).
$$
Now recalling bound (\ref{generic:no:edge:bound}), we can equivalently rewrite its RHS as
$$
\cQ \leq 1 - \frac{4s e^{-2(s-2)\Theta} + 4s e^{-2s\Theta} + 4}{2 + e^{2s\Theta} + e^{-2s\Theta} + 2s(1 + e^{-4(s-1)\Theta} + e^{-2(s-2)\Theta} + e^{-2s\Theta})}.
$$
Under the assumption $e^{2s\Theta} \gg s$ the above expression is asymptotically equivalent to
$$
1 - 4 e^{-2s\Theta}- 4s e^{-4(s-1)\Theta} - 4s e^{-4s\Theta} + o(1).
$$
Therefore, when $m = s + 1$, and $\theta = \Theta$ the asymptotic difference is of magnitude at least $4e^{-4(s-1)\theta}$. Hence $\cT - \cQ$ is asymptotically $4e^{-4(s-1)\theta}$ and thus if 
$$
2e^{-4(s-1)\theta} \geq  \sqrt{\frac{4 \log d + \log \delta^{-1}}{n}},
$$
the test is successful. This is equivalent to $\theta \leq \frac{\log{\frac{n}{ \log d + \log \delta^{-1}/4}}}{8(s-1)}$, which matches bound (\ref{reader:friendly:upper:bound:clique}) up to scalars.
\end{remark}

\section{Bounds for General Models}\label{bounds:with:ferromagnets}

The next two results are dedicated to the proof of the main result of Section \ref{antiferromagnetic:bounds}. To ease the presentation we define the following notation. For a set $S \subset \RR$, a set $V \subset \NN$ by $\xb \in S^V$ we understand $\xb = (x_k)_{k \in V}$ and $x_k \in S$. For a vector $\xb \in \RR^W$, and a set $V$ where $V \subseteq W \subset \NN$ let
\begin{align}\label{shorthands:sq:sums}
S^V_{\xb} := \sum_{k \in V} x_k, \mbox{ and } SS_{\xb}^V = \sum_{(k,\ell) \in {V\choose 2} } x_k x_\ell,
\end{align}
where we remind the reader that ${V\choose 2} = \{(k,\ell) ~|~ k < \ell, ~ k,\ell \in V\}$. To relieve the sub-indeces, we further introduce a slight abuse of notation. For a fixed $\theta > 0$ and a \textit{multigraph} $G$ (i.e., a graph allowed to have more than one edge joining two vertices) we denote with $\PP_G$ the Ising measure
$$
\PP_G(\bX = \xb) \varpropto \exp\bigr(-\theta \sum_{(u,v) \in E(G)} x_u x_v\bigr),
$$
for $\theta \geq 0$. 
We proceed with the
\begin{proof}[Proof of Theorem \ref{scaling:theorem}]  The proof is concerned with a more fundamental problem than testing the strongly monotone property $\cP$. Namely we consider the problem of detecting an antiferromagnetic $s$-clique. To elaborate in detail, for a set $V$ with cardinality $|V| = s$, let $C_s(V)$ denote the $s$-clique with vertices in the set $V$. For a vertex set $V \subset [d]$, put $G_V := ([d], E(C_s(V)))$. The pdf of the Ising model $\PP_{G_V}$ is given by
$$
\PP_{G_V}(\bX = \xb) \varpropto \exp\Big(-\theta SS_{\xb}^V \Big),
$$
where $\xb \in \RR^d$ and $\theta \geq 0$. Let $\PP_{\varnothing}$ denote a $d$-dimensional Rademacher vector corresponding to the null hypothesis, i.e., the hypothesis with an empty graph, or equivalently the null hypothesis when $\theta = 0$. Given a sample of $n$ observations from the measure $\PP$, the antiferromagnetic clique detection problem aims to test
\begin{align}\label{antiferro:clique:detection}
\Hb_0: \PP = \PP_{\varnothing} \mbox{ vs } \Hb_1 : \PP = \PP_{G_V} \mbox{ for some } V\subset [d], |V| = s.
\end{align}
It is clear that any test of $\cP$, can test (\ref{antiferro:clique:detection}) with the same type I and type II error control. Let $\pi$ denote the distribution uniformly sampling the vertex set $V \subset [d]$ of an $s$-clique. Define the mixture measure
$$
\overline \PP^{\otimes n}(\bX_1 = \xb_1, \ldots, \bX_n = \xb_n) := \EE_{V \sim \pi} \PP^{\otimes n}_{G_V}(\bX_1 = \xb_1, \ldots, \bX_n = \xb_n).
$$
Using Le Cam's Lemma \citep{yu1997assouad}, we have
$$
R_n(\cP, \cR_s, \theta) \geq 1 - \frac{1}{2}\sqrt{D_{\chi^2}(\overline \PP^{\otimes n}, \PP_{\varnothing}^{\otimes n})}.
$$
Therefore it suffices to control the divergence $D_{\chi^2}(\overline \PP^{\otimes n}, \PP_{\varnothing}^{\otimes n})$ (recall Definition \ref{def:chisq:div}). 

To write the expression $D_{\chi^2}(\overline \PP^{\otimes n}, \PP_{\varnothing}^{\otimes n})$ in a convenient form, we need to introduce several quantities. To this end take the $s$-clique graph $C_s(V)$ and for a $\theta > 0$ define the partition function
$$
Z_{C_s(V)}(\theta) := \sum_{\xb \in \{\pm 1\}^V} \exp\bigr(-\theta SS_{\xb}^V\bigr). 
$$
Taking two $s$-clique graphs $C_s(V)$ and $C_s(V')$ define the partition function based on the multigraph $C_s(V) \oplus C_s(V')$ (see Figure \ref{multigraph:depiction}):
$$
Z_{C_s(V) \oplus C_s(V')}(\theta) := \sum_{\xb \in \{\pm 1\}^{V \cup V'}} \exp\bigr(-\theta \bigr[SS_{\xb}^{V} +SS_{\xb}^{V'}\bigr]\bigr).
$$
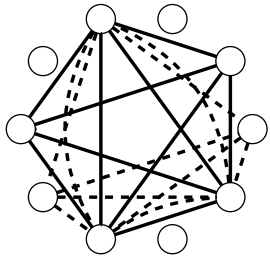
\begin{figure}
\centering
\begin{tikzpicture}[scale=.7]
\SetVertexNormal[Shape      = circle,
                  FillColor  = white,
                  MinSize = 11pt,
                  InnerSep=0pt,
                  LineWidth = .5pt]
   \SetVertexNoLabel
   \tikzset{LabelStyle/.style = {above, fill = white, text = black, fill opacity=0, text opacity = 1}}
         \tikzset{EdgeStyle/.style= {thin,dashed,
                                double          = black,
                                double distance = .5pt}}
  \begin{scope}[shift={(0cm, 0cm)}]
\grEmptyCycle[prefix=c,RA=2.2]{5}%
 \begin{scope}[rotate=36]
         \tikzset{EdgeStyle/.style= {thin,
                                double          = black,
                                double distance = .5pt}}
\grComplete[prefix=d,RA=2.2]{5}%
  \end{scope}
  \Edge(d1)(c3)
    \Edge(d1)(c0)
    \Edge(c3)(c0)
        \Edge(c3)(d4)
        \Edge(c0)(d4)
                \Edge(c3)(d3)
        \Edge(c0)(d3)
                     \tikzset{EdgeStyle/.style= {thin,dashed, bend right = 30,
                                double          = black,
                                double distance = .5pt}}
                                \Edge(d1)(d3)
                                \Edge(d4)(d1)
                         \tikzset{EdgeStyle/.style= {thin,dashed, bend left = 15,
                                double          = black,
                                double distance = .5pt}}
                                                                \Edge(d3)(d4)
\end{scope}
\end{tikzpicture}
\caption{The multigraph $C_s(V) \oplus C_s(V')$ is the graph whose adjacency matrix is the sum of the adjacency matrices of $C_s(V)$ and $C_s(V')$, i.e., $\Ab_{C_s(V) \oplus C_s(V')} = \Ab_{C_s(V)} + \Ab_{C_s(V')}$. Above is a depiction of $C_s(V) \oplus C_s(V')$, where $s = 5, d = 10$ and $|V \cap V'| = 3$; $C_s(V)$ is plotted with solid edges, while $C_s(V')$ is plotted with dashed edges.} \label{multigraph:depiction}
\end{figure}
Define the standardized version of the partition function above by 
$$\textstyle T(\theta,V \cap V') := \frac{Z_{C_s(V) \oplus C_s(V')}(\theta)}{2^{2 s - |V \cap V'|}}.$$
Note that when $V \cap V' = \varnothing$, $T(\theta,\varnothing) = \frac{Z_{C_s(V)}(\theta)Z_{C_s(V')}(\theta)}{2^{2s}}$.
A bit of algebra shows that one can write
\begin{align}\label{chi:sq:divergence:mixture}
\textstyle D_{\chi^2}(\overline \PP^{\otimes n}, \PP_{\varnothing}^{\otimes n}) = \EE_{V, V' \sim \pi}\frac{T^n(\theta, V \cap V')}{T^n(\theta, \varnothing)} - 1,
\end{align}
where $V,V' \sim \pi$ indicates a two sample i.i.d. draw from $\pi$. We now state the key result enabling us to prove Theorem \ref{scaling:theorem}, and defer its proof to after we complete this proof.
\begin{theorem} \label{most:important:theorem} For any pair of vertex sets $V, V'$ of cardinality $s$, the ratio
$$
\textstyle \frac{Z_{C_s(V) \oplus C_s(V')}(\theta)}{Z_{C_s(V)}(\theta)Z_{C_s(V')}(\theta)},
$$
is non-decreasing in $\theta$ for $\theta \geq 0$. 
\end{theorem}
Although the statement of Theorem \ref{most:important:theorem} is elementary, its proof is involved. The main difficulty stems from the fact that we are considering an antiferromagnetic model, also does not in general obey correlation inequalities in contrast to the ferromagnetic case . 

Fix two vertex sets $V,V'$ of cardinality $s$, and let $\cI = V\cap V'$ for brevity. By virtue of Theorem \ref{most:important:theorem}, we have that $\frac{T(\theta, \cI)}{T(\theta, \varnothing)}$ is increasing in $\theta$. Hence to upper bound this ratio it suffices to understand its behavior in the limit $\theta \rightarrow \infty$. 
Define  $U := S^{V \setminus \cI}_{\xb}$ and $I := S^{\cI}_{\xb}$, and rewrite $T(\theta, \cI)$ as
$$
\textstyle T(\theta, \cI) = e^{s \theta} \!\!\!\!\!\! \sum\limits_{I = -|\cI|, ~ 2 | [I + |\cI|]}^{|\cI|} \frac{{|\cI| \choose \frac{I + |\cI|}{2}}}{2^{|\cI|}}\bigg[\sum\limits_{U = -(s - |\cI|), ~ 2 | [U + (s - |\cI|)]}^{s - |\cI|} \frac{{s - |\cI| \choose \frac{U + s - |\cI|}{2}}}{2^{s-|\cI|}}e^{-\frac{\theta}{2} (U + I)^2}\bigg]^2.
$$
After a bit of algebra, depending on the parity of $s$ we have the following identity
$$
\lim_{\theta \rightarrow \infty} \frac{T(\theta, \cI)}{T(\theta, \varnothing)} = \begin{cases}\frac{2^{|\cI|}}{{s \choose \frac{s}{2}}^2}\sum_{j = 0, 2|j}^{2|\cI|}{|\cI| \choose \frac{j}{2}} {s - |\cI| \choose \frac{s - j}{2}}^2, \mbox{ if } s \mbox{ is even},\\
\frac{2^{|\cI|}}{4 {s \choose \frac{s-1}{2}}^2}\sum_{j = 0, 2|j}^{2|\cI|}{|\cI| \choose \frac{j}{2}} \bigr[{s - |\cI| \choose \frac{s -1 - j}{2}} + {s - |\cI| \choose \frac{s + 1 - j}{2}}\bigr]^2, \mbox{ if } s \mbox{ is odd.}\end{cases}
$$
We now distinguish two cases.\\\\
{\bf Case I.} First consider the case when $s$ is even. We have
$$
\textstyle \sum_{j = 0, 2|j}^{2|\cI|}{|\cI| \choose \frac{j}{2}} {s - |\cI| \choose \frac{s - j}{2}}^2 \leq \sum_{j = 0, 2|j}^{2|\cI|}{|\cI| \choose \frac{j}{2}} {s - |\cI| \choose \frac{s - j}{2}}\max_{j} {s - |\cI| \choose \frac{s - j}{2}} = {s \choose \frac{s}{2}} {s - |\cI| \choose \frac{s}{2} - \lfloor\frac{|\cI|}{2}\rfloor},
$$
where in the last equality we used Vandermonde's identity. We conclude that
\begin{align}\label{binom:bound}
\textstyle \sup_{\theta \geq 0}\frac{T(\theta, \cI)}{T(\theta, \varnothing)}\leq \frac{2^{|\cI|}{s - |\cI| \choose s/2 - \lfloor|\cI|/2\rfloor}}{{s \choose s/2}} =: F(|\cI|).  
\end{align}
A bit of algebra shows that the ratio $\frac{F(|\cI|)}{F(|\cI| + 2)} \leq 1$ for all values of $|\cI|$. Hence the maximum is achieved at $|\cI| \in \{s, s-1\}$ (and a direct check shows that it is achieved at $|\cI| = s$). Therefore 
$$
\textstyle \sup_{\theta \geq 0}\frac{T(\theta, \cI)}{T(\theta, \varnothing)}\leq \frac{2^s}{{s \choose s/2}} \leq \sqrt{2s},
$$
where the last inequality follows by Lemma 17.5.1 of \cite{cover2012elements}. Moreover, by expanding the binomials on the right hand side (RHS) of (\ref{binom:bound}) and some algebra we have
$$
\textstyle \sup_{\theta \geq 0}\frac{T(\theta, \cI)}{T(\theta, \varnothing)} \leq (1 + \frac{1}{s - 2\lfloor |\cI|/2\rfloor + 1})^{\lfloor |\cI|/2\rfloor}.
$$
{\bf Case II.}  Now consider the case when $s$ is odd. Using the same ideas as in case I. shows that the following bound holds
$$
\textstyle \sup_{\theta \geq 0}\frac{T(\theta, \cI)}{T(\theta, \varnothing)} \leq (1 + \frac{1}{s - 2\lceil |\cI|/2\rceil + 2})^{\lceil |\cI|/2\rceil} \wedge \sqrt{2s}.
$$
Compiling all inequalities yields
\begin{align*}
\MoveEqLeft \textstyle \sup_{\theta \geq 0}\frac{T^n(\theta, \cI)}{T^n(\theta, \varnothing)} \leq \textstyle (1 + \frac{1}{s - 2\lfloor |\cI|/2\rfloor + 1})^{n\lceil |\cI|/2\rceil} \wedge (\sqrt{2s})^{n} \\
& \leq \textstyle \exp(\frac{n\lceil |\cI|/2\rceil}{s - 2\lfloor |\cI|/2\rfloor + 1}) \wedge (\sqrt{2s})^{n} \leq \exp(\frac{2n|\cI|}{s}) + \mathbbm{1}(|\cI| \geq s/2)(\sqrt{2s})^{n}.
\end{align*}
The final step of our proof is to control (\ref{chi:sq:divergence:mixture}). Taking expectation with respect to $V, V' \sim \pi$ and subtracting $1$ from the preceding display yields:
\begin{align}
\textstyle \MoveEqLeft D_{\chi^2}(\overline \PP^{\otimes n}, \PP_0^{\otimes n}) + 1 \nonumber \\
& \textstyle \leq \EE_{V, V' \sim \pi} \exp(\frac{2n|V \cap V'|}{s}) +  \EE_{V, V' \sim \pi} \mathbbm{1}(|V \cap V'| \geq \frac{s}{2})(\sqrt{2s})^{n} \nonumber \\ 
& \textstyle = \EE_{V' \sim \pi}[\exp(\frac{2n|V \cap V'|}{s}) | V = [s]] +  \EE_{V' \sim \pi}[ \mathbbm{1}(|V \cap V'| \geq \frac{s}{2}) | V = [s]](\sqrt{2s})^{n},\label{chi:sq:div:bound}
\end{align}
where in the final identity we used the symmetry of the problem to condition on the event $V = [s]$. In order to handle (\ref{chi:sq:div:bound}) we first recognize that given $V = [s]$, the distribution of $|V \cap V'|$ is hypergeometric: the number of red balls out of $s$ balls drawn without replacement from an urn consisting of $s$ red balls and $d-s$ blue balls. Using a standard bound of the tail probability of a hypergeometric distribution \citep{chvatal1979tail} we have
$$
\textstyle \EE_{V' \sim \pi}[ \mathbbm{1}(|V \cap V'| \geq \frac{s}{2}) | V = [s]] \leq 2^s \bigr(\frac{s}{d}\bigr)^{s} \bigr(\frac{d-s}{d}\bigr)^{s} \leq 2^s \bigr(\frac{s}{d}\bigr)^{s}.
$$
To deal with the first term of the RHS of (\ref{chi:sq:div:bound}), we use the representation $|V \cap V'| = \sum_{v \in V'} \mathbbm{1}(v \in V)$, and the fact that $\{\mathbbm{1}(v \in V)\}_{v \in V'}$ are negatively associated \citep[see, e.g., for a proper definition]{joag1983negative}. We have
\begin{align*}
\textstyle \EE_{V' \sim \pi}[\exp(\frac{2n|V \cap V'|}{s}) | V = [s]] - 1 & \leq \prod_{v \in V'} \bigr(\PP(v \in V) e^{2n/s} + 1 -\PP(v \in V)\bigr) \\
& \leq \bigr(1 + e^{2n/s}\frac{s}{d}\bigr)^s \leq \exp\bigr(\frac{s^2}{d} e^{2n/s}\bigr).
\end{align*}
Hence, continuing the bounds in (\ref{chi:sq:div:bound}) we finally have 
$$
D_{\chi^2}(\overline \PP^{\otimes n}, \PP_{\varnothing}^{\otimes n}) \leq \exp\bigr(\frac{s^2}{d} e^{2n/s}\bigr) - 1 +  \bigr(\frac{2s}{d}\bigr)^{s}(\sqrt{2s})^{n}.
$$
It can be checked that under the sufficient conditions (\ref{sufficient:conditions:clique:detection}) the above expression goes to $0$ asymptotically. With this our proof is complete.
\end{proof}

\begin{proof}[Proof of Theorem \ref{most:important:theorem}]
For brevity we will use the shorthand notation $G := C_s(V)$ and $G' := C_s(V')$. Recall the definitions of $S_{\xb}^V$ and $SS_{\xb}^V$ in (\ref{shorthands:sq:sums}). The statement is equivalent to showing that 
\begin{align*}
\textstyle  \frac{d}{d \theta} \log \frac{Z_{G \oplus G'}(\theta)}{Z_{G}(\theta)Z_{G'}(\theta)}  =\EE_{G}\bigr[ SS^{V}_{\bX}\bigr] + \EE_{G'}\bigr[ SS^{V'}_{\bX} \bigr] -\EE_{G \oplus G'}\bigr[ SS^{V}_{\bX}+SS^{V'}_{\bX} \bigr] \geq 0.
\end{align*}
The latter is implied if we show that
\begin{align}\label{expectation:inequality}
\textstyle \EE_{G}\bigr[ SS^{V}_{\bX} \bigr] \geq \EE_{G \oplus G'}\bigr[ SS^{V}_{\bX} \bigr].
\end{align}
To this end, for any $0\leq t \leq 1$ define the model
$$
\PP_{G\oplus t G'}(\bX = \xb) \varpropto \exp\bigr(-\theta \bigr[SS^{V}_{\xb} + t SS^{V'}_{\xb} \bigr]\bigr).
$$
Clearly, $\PP_{G \oplus t G'}$ interpolates between the measures $\PP_{G}$ and $\PP_{G \oplus  G'}$, and (\ref{expectation:inequality}) will be implied if we showed that $\frac{d}{dt}\EE_{G\oplus tG'}\bigr[ SS^{V}_{\xb} \bigr] \leq 0.$
A direct calculation yields 
\begin{align*}
\MoveEqLeft \frac{d}{dt}\EE_{G\oplus tG'}\bigr[ SS^{V}_{\bX} \bigr] = -\theta \bigr(\EE_{G\oplus tG'}\bigr[SS^{V}_{\bX}SS^{V'}_{\bX} \bigr] - \EE_{G\oplus tG'}\bigr[SS^{V}_{\bX}\bigr] \EE_{G\oplus tG'}\bigr[SS^{V'}_{\bX}\bigr] \bigr).
\end{align*}
Thus it suffices to show the following positive correlation inequality, for all $0 \leq t \leq 1$: $
\EE_{G\oplus tG'}\bigr[SS^{V}_{\bX}SS^{V'}_{\bX} \bigr] \geq \EE_{G\oplus tG'}\bigr[SS^{V}_{\bX}\bigr] \EE_{G\oplus tG'}\bigr[SS^{V'}_{\bX}\bigr].$
Notice that for any $V$ and $\xb \in \{\pm 1\}^V$ we have the identity $(S_{\xb}^V)^2 = 2SS_{\xb}^V + |V|$. Hence the above is precisely equivalent to
\begin{align}\label{square:corr:ineq}
\EE_{G\oplus tG'}\bigr[\bigr(S^{V}_{\bX}\bigr)^2\bigr(S^{V'}_{\bX}\bigr)^2\bigr] \geq  \EE_{G\oplus tG'}\bigr(S^{V}_{\bX}\bigr)^2 \EE_{G\oplus tG'}\bigr(S^{V'}_{\bX}\bigr)^2.
\end{align}
If we were in the ferromagnetic case (\ref{square:corr:ineq}) would have followed by Griffiths inequality \citep{griffiths1967correlations}, which would have completed the proof. However, in antiferromagnetic cases the situation is  more challenging, as such positive correlation inequalities are not expected to hold in general. Therefore below we utilize the special structure of our problem. 

According to the definition of $\PP_{G \oplus t G'}$, the random vectors $\bX_{V \setminus (V \cap V')} = (X_i)_{i \in V \setminus (V \cap V')}$ and $\bX_{V' \setminus (V \cap V')} = (X_i)_{i \in V' \setminus (V \cap V')}$ are conditionally independent given $S^{V \cap V'}_{\bX}$. Let $\tilde \bX$ be an independent copy of $\bX$. Put for brevity
\begin{align*}
\cH_G(\bX, \tilde \bX) &:= \EE_{G\oplus tG'}\bigr[\bigr(S^{V}_{\bX}\bigr)^2 \bigr | S^{V \cap V'}_{\bX}\bigr] - \EE_{G\oplus tG'}\bigr[\bigr(S^{V}_{\tilde \bX}\bigr)^2 \bigr | S^{V \cap V'}_{\tilde \bX}\bigr]\\
\cH_{G'}(\bX, \tilde \bX) &:= \EE_{G\oplus tG'}\bigr[\bigr(S^{V'}_{\bX}\bigr)^2 \bigr | S^{V \cap V'}_{\bX}\bigr] - \EE_{G\oplus tG'}\bigr[\bigr(S^{V'}_{\tilde \bX}\bigr)^2 \bigr | S^{V \cap V'}_{\tilde \bX}\bigr],
\end{align*}
and note that (\ref{square:corr:ineq}) can be equivalently expressed as
$$
\EE_{G \oplus t G'} \cH_G(\bX, \tilde \bX) \cH_{G'}(\bX, \tilde \bX) \geq 0.
$$
The above will be implied if for any feasible $h, \tilde h$ we showed that the following two numbers
\begin{align*}
\EE_{G\oplus tG'}\bigr[\bigr(S^{V}_{\bX}\bigr)^2 \bigr | S^{V \cap V'}_{\bX} = h\bigr] - \EE_{G\oplus tG'}\bigr[\bigr(S^{V}_{\tilde \bX}\bigr)^2 \bigr | S^{V \cap V'}_{\tilde \bX} = \tilde h\bigr], \\
\EE_{G\oplus tG'}\bigr[\bigr(S^{V'}_{\bX}\bigr)^2 \bigr | S^{V \cap V'}_{\bX} = h\bigr] - \EE_{G\oplus tG'}\bigr[\bigr(S^{V'}_{\tilde \bX}\bigr)^2 \bigr | S^{V \cap V'}_{\tilde \bX} = \tilde h\bigr],
\end{align*}
have the same sign. On the other hand, due to sign symmetry in $\bX$, the above is implied if the following functions are increasing in $h$ on the set $\NN_0 \cap \{-|V \cap V'|, -|V \cap V'| + 2, \ldots, |V \cap V'|-2,|V \cap V'|\}$, 
$$
\EE_{G\oplus tG'}\bigr[\bigr(S^{V}_{\bX}\bigr)^2 \bigr | S^{V \cap V'}_{\bX} = h\bigr], ~~~ \EE_{G\oplus tG'}\bigr[\bigr(S^{V'}_{\bX}\bigr)^2 \bigr | S^{V \cap V'}_{\bX} = h\bigr].
$$
By a simple calculation one can check that for any $\xb \in \{\pm 1\}^{|V \setminus (V \cap V')|}, $
$$
\PP_{G \oplus t G'} \bigr(\bX_{V \setminus (V \cap V') } = \xb|  S^{V \cap V'}_{\bX} = h\bigr)\varpropto \exp\bigg(-\frac{\theta}{2} \bigr(S^{V\setminus(V \cap V')}_{\xb}+ h\bigr)^2 \bigg),
$$
and similarly for any $\xb \in \{\pm 1\}^{|V' \setminus (V \cap V')|}$
$$
\PP_{G \oplus t G'} \bigr(\bX_{V' \setminus (V \cap V') } = \xb|  S^{V \cap V'}_{\bX} = h\bigr)\varpropto \exp\bigg(-\frac{\theta t}{2} \bigr(S^{V'\setminus(V \cap V')}_{\xb}+ h\bigr)^2 \bigg).
$$
Hence, due to this symmetry and the form of the distribution, it is sufficient to show that for any $k$-clique graph $H$ with $|V(H)| = k$, $E = {V(H) \choose 2}$, and any $h \in \NN_0$
\begin{align}\label{expectation:difference}
\EE_{H,  h + 2}\bigr(S^{V(H)}_{\bX} + h + 2\bigr)^2 - \EE_{H, h}\bigr(S^{V(H)}_{\bX} + h\bigr)^2 \geq 0,
\end{align}
where
$$
\PP_{H,h}(\bX = \xb) \varpropto \exp\bigg(-\frac{\theta}{2}\bigr( S^{V(H)}_{\xb}+ h \bigr)^2\bigg).
$$
If $Z_{H,h}(\theta)$ denotes the partition function of $\PP_{H,h}$, one can verify that the sign of the LHS of (\ref{expectation:difference}) coincides with the sign of the derivative
$$
\frac{d}{d \theta} \log \frac{Z_{H, h}(\theta)}{Z_{H, h + 2}(\theta)},
$$
and therefore if suffices to show that for all $\theta \geq \theta' \geq 0$ and $h \in \NN_0$ we have
\begin{align}\label{part:func:ratio:ineq}
\frac{Z_{H, h}(2\theta')}{Z_{H, h + 2}(2\theta')} \leq \frac{Z_{H, h}(2\theta)}{Z_{H, h + 2}(2\theta)},
\end{align}
where we scaled the parameters $\theta$ and $\theta'$ by $2$ for convenience. For two vectors $\bX$ and $\bY$ put for brevity $S_X = S^{V(H)}_{\bX} = \sum_{i \in V(H)} X_i$ and $S_Y = S^{V(H)}_{\bY}  = \sum_{i \in V(H)} Y_i$, and let $V(H) = \{v_1,\ldots, v_k\}$. (\ref{part:func:ratio:ineq}) is equivalent to 
$$
\EE_{\varnothing} \exp(-\theta (S_Y + h + 2)^2 - \theta'(S_X + h)^2) \leq \EE_{\varnothing} \exp(-  \theta'(S_Y + h + 2)^2 -\theta(S_X + h)^2),
$$
where $\EE_\varnothing$ denotes the expectation with respect to uniformly drawing $(X_{v_1},\ldots,\allowbreak X_{v_k}, \allowbreak Y_{v_1}, \ldots, Y_{v_k})$ from the set $\{\pm 1\}^{2k}$. The random variables inside the expectations are discrete and non-negative and therefore the inequality from the preceding display can be written as
\begin{align*}
\MoveEqLeft \int_{0}^{\infty} \PP_{\varnothing}\bigg(\exp\bigg(-\theta (S_Y + h + 2)^2 - \theta'(S_X + h)^2\bigg) > t\bigg) dt \\
& \leq \int_{0}^{\infty} \PP_{\varnothing}\bigg(\exp\bigg(-\theta' (S_Y + h + 2)^2 - \theta(S_X + h)^2\bigg) > t\bigg) dt,
\end{align*}
where $\PP_{\varnothing}$ is the uniform measure on $\{\pm 1\}^{2k}$. Hence to prove (\ref{expectation:difference}) it suffices to show that for all $\theta \geq \theta' \geq 0$ the following stochastic dominance holds
\begin{align}\label{stoch:dominance}
\textstyle \PP_{\varnothing}(\theta (S_Y + h + 2)^2 + \theta' (S_X + h)^2 < t ) & \leq \PP_{\varnothing}(\theta'(S_Y + h + 2)^2 + \theta(S_X + h)^2 < t)
\end{align}
where $t \geq 0$. We show this in the following
\begin{proposition}\label{key:proposition} Let $\{X_i\}_{i \in [k]}$ and $\{Y_i\}_{i \in [k]}$ are i.i.d. Rademacher random variables, $h \in \NN_0$ and $\theta \geq \theta' \geq 0$. Then (\ref{stoch:dominance}) holds. 
\end{proposition}
The proof of Proposition \ref{key:proposition} is technical and we prove it below. With this our proof is complete.
\end{proof}

\begin{proof}[Proof of Proposition \ref{key:proposition}] Define the two ellipses
\begin{align*}
E_1 &= \{(x,y) \in \RR^2 : \theta (x + h)^2 + \theta'(y + h + 2)^2 < t\}, \mbox{ and } \\
E_2 &= \{(x,y) \in \RR^2 : \theta' (x + h)^2 + \theta(y + h + 2)^2 < t\}.
\end{align*}
Clearly both ellipses are centered at the point $(-h, - (h + 2))$. Let us begin by calculating the intersection points of the contours of the two ellipses in the case when $\theta \neq \theta'$ (the case $\theta = \theta'$ is trivial, as (\ref{stoch:dominance}) holds with an equality). The contours intersect when 
$$
\theta (x + h)^2 + \theta'(y + h + 2)^2 = \theta' (x + h)^2 + \theta(y + h + 2)^2 = t.
$$
Hence at the intersection, we necessarily have $(x - y - 2)(x + y + 2h + 2) = 0$. Note that the equations $x - y - 2 = 0$, and $x + y + 2h + 2 = 0$ define two lines, which pass through the center $(-h,-(h+2))$, and partition the $x,y$-plane into four regions. Denote those regions as 
\begin{align*}
\cE = \{(x,y) \in \RR^2: x - y - 2 \geq 0, x + y + 2h + 2 \geq 0\},\\
\cN = \{(x,y) \in \RR^2: x - y - 2 \leq 0, x + y + 2h + 2 \geq 0\},\\
\cW =  \{(x,y) \in \RR^2: x - y - 2 \leq 0, x + y + 2h + 2 \leq 0\},\\
\cS =  \{(x,y) \in \RR^2 : x - y - 2 \geq 0, x + y + 2h + 2 \leq 0\}.
\end{align*}
To facilitate the readability of the proof, we plot two examples of $E_1$ and $E_2$, along with the two lines and four regions in Figure \ref{fig:two:ellipses}.
\begin{figure}[t]
\centering
\begin{subfigure}{.5\textwidth}
  \centering
  \includegraphics[width=1\linewidth]{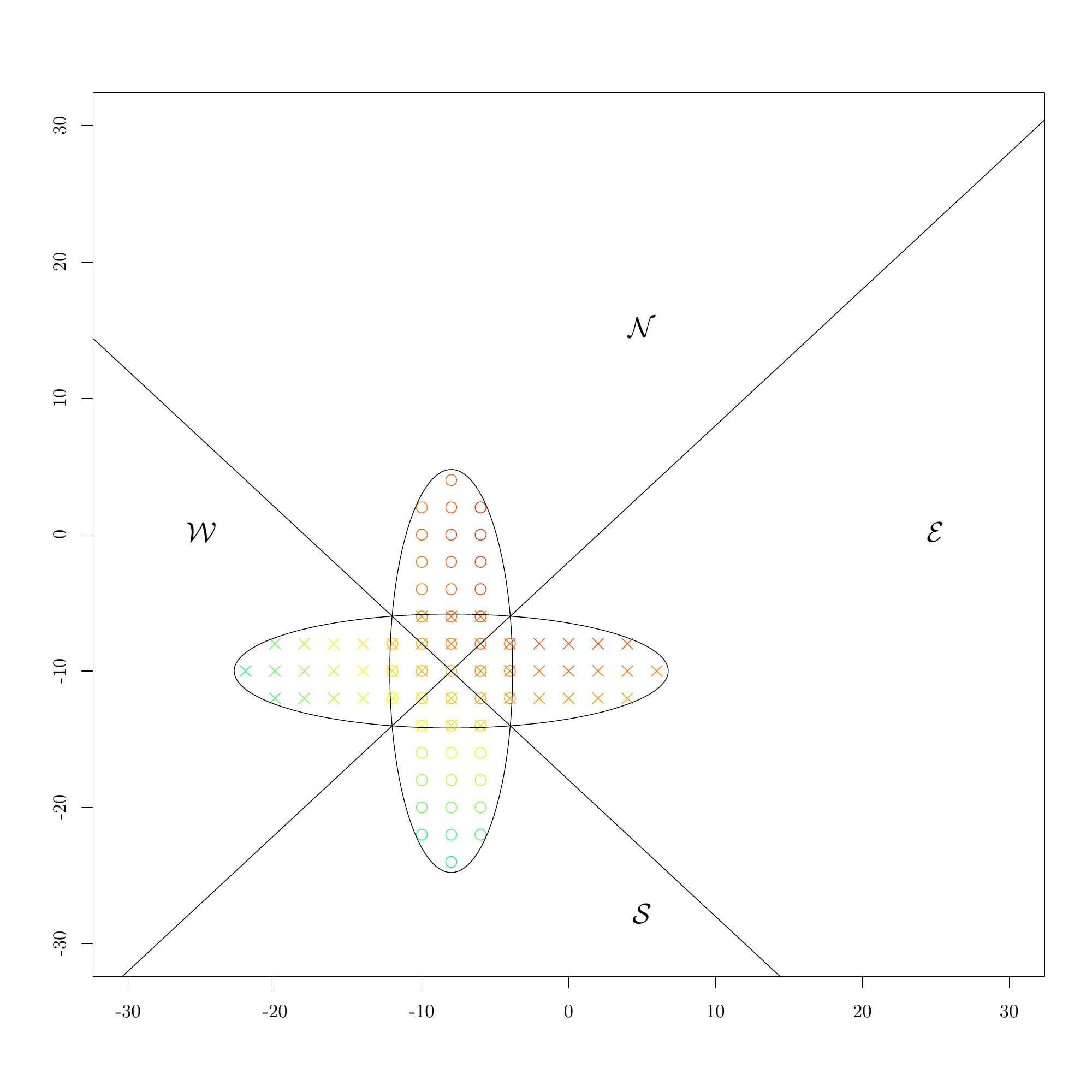}
\end{subfigure}%
\begin{subfigure}{.5\textwidth}
  \centering
  \includegraphics[width=1\linewidth]{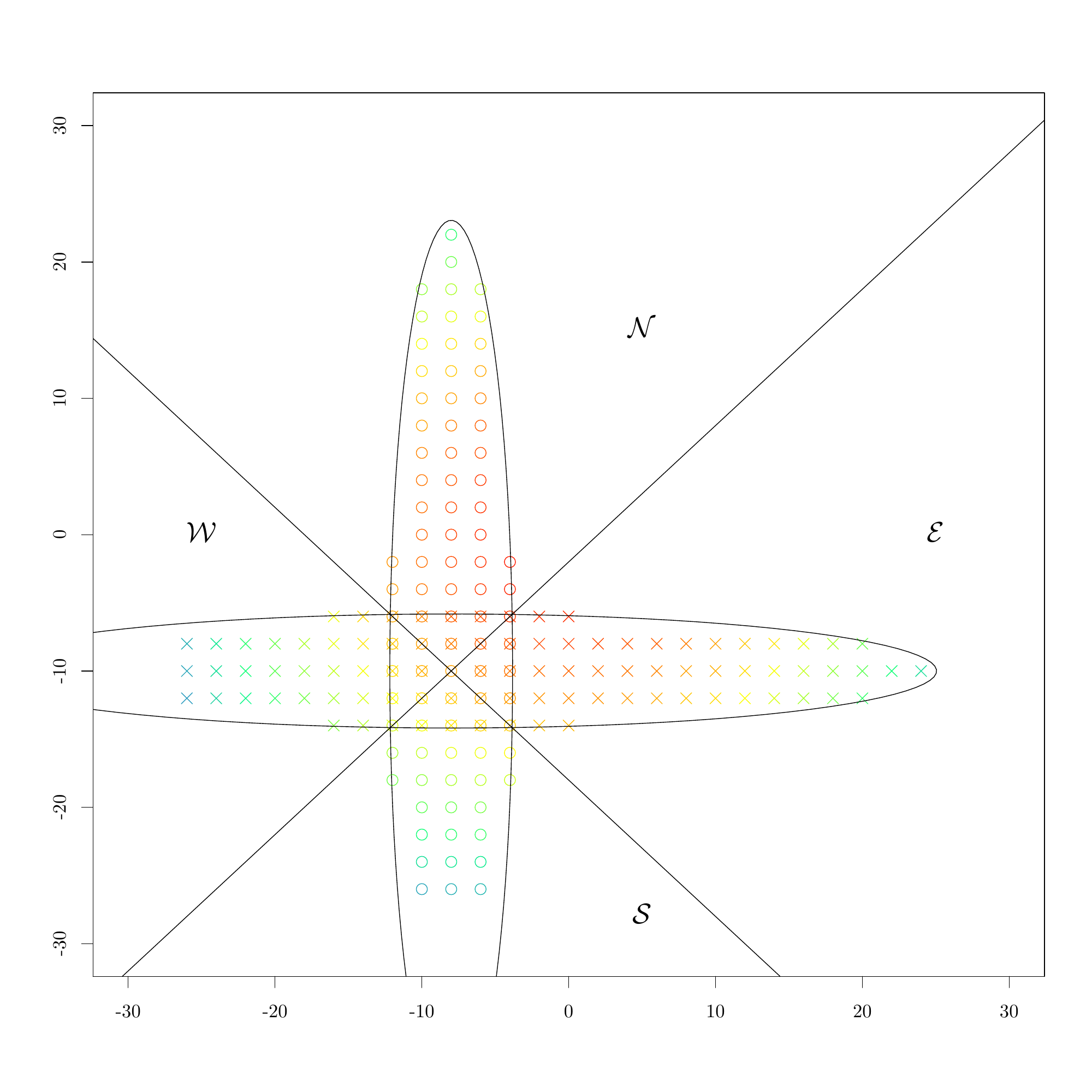}
\end{subfigure}
\caption{An example of the two ellipses $E_1$ and $E_2$ with the same values of $\theta$ but two different values of $\theta'$. The points inside the ellipses represent points from the set $S$. The point of intersection of the two lines is the center $(-h,-h-2)$, where for this specific example we have set $h = 8$, $k = 26$. Points ``o'' form the set $E_1\cap S$, and points ``x" form the set $E_2\cap S$. The four regions $\cE, \cN, \cW, \cS$ are also depicted. Points of warmer colors are more likely to occur than points with cooler colors.}
\label{fig:two:ellipses}
\end{figure}
By the above definitions we know that $E_1 \setminus E_2 \subset \cN \cup \cS$, and $E_2 \setminus E_1 \subset \cE \cup \cW$. Note that the pairs $(S_X, S_Y)$ can only take integer values, and in such a way so that $-k \leq S_X \leq k$, $-k \leq S_Y \leq k$, and $S_X  \equiv S_Y  \equiv k ~(\operatorname{mod} 2)$. Define
$$
S =  \{(x,y) \in  (\NN_0)^2 | -k \leq x \leq k, -k \leq y \leq k, x  \equiv y  \equiv k ~(\operatorname{mod} 2)\}
$$ 
In terms of this notation (\ref{stoch:dominance}) can be restated as
$$
\PP_{\varnothing}((S_X, S_Y) \in E_2 \cap S) \leq \PP_{\varnothing}((S_X, S_Y) \in E_1 \cap S),
$$
or equivalently
\begin{align}\label{prob:ineq}
\PP_{\varnothing}((S_X, S_Y) \in (E_2 \setminus E_1)\cap S) \leq \PP_{\varnothing}((S_X, S_Y) \in (E_1\setminus E_2) \cap S).
\end{align}

In the following, our proof plan is to show (\ref{prob:ineq}) by taking pairs of points from the set $(E_1\setminus E_2) \cap S$ and matching them in a 1-1 manner with pairs of points from the set $(E_2 \setminus E_1)\cap S$, while ensuring that the former pairs are more likely to occur than the latter. Sometimes one or more of the points in the two pairs will lie outside of the set $S$ (see, e.g., the right panel of Figure \ref{fig:two:ellipses}), but we will make sure that this does not alter the conclusion. 

Take a point $(i,j) \in (E_1 \setminus E_2) \cap S \cap \cN$. Define the transformation $(\hat i, \hat j) = (-j - 2h-2, -i - 2h -2)$, and consider the four points $(i,j), (j + 2, i - 2), (\hat j + 2, \hat i - 2), (\hat i, \hat j)$. Since $(i,j) \in (E_1 \setminus E_2) \cap \cN$ it is simple to check that $(j + 2, i - 2) \in (E_2\setminus E_1) \cap \cE$, $(\hat i, \hat j) \in (E_2\setminus E_1) \cap \cW$ and $(\hat j + 2, \hat i - 2) \in (E_1\setminus E_2) \cap \cS$. We will now argue that 
\begin{align}\label{very:important:ineq}
\PP_\varnothing \bigr((S_X, S_Y) \in \{(i,j), (\hat j + 2, \hat i -2)\}\bigr) \geq \PP_\varnothing \bigr((S_X, S_Y) \in \{(j + 2,i - 2), (\hat i, \hat j)\}\bigr). 
\end{align}
Note that for any point $(m,\ell) \in S$
$$
\PP_\varnothing((S_X, S_Y) = (m,\ell)) = {k \choose \frac{k + m}{2}}{k \choose \frac{k + \ell}{2}}2^{-2k}.
$$
To this end we consider several cases. 
\begin{itemize}
\item[i.] First assume that all of the above four points belong to $S$. Therefore, we need to show that
$$
{k \choose \frac{k + i}{2}}{k \choose \frac{k + j}{2}} - {k \choose \frac{k +  i - 2}{2}}{k \choose \frac{k +  j + 2}{2}} \geq {k \choose \frac{k + \hat i}{2}}{k \choose \frac{k + \hat j}{2}} - {k \choose \frac{k +  \hat i - 2}{2}}{k \choose \frac{k +  \hat j + 2}{2}}.
$$
A direct calculation yields that the above is equivalent to
$$
\frac{\bigr(\frac{j - i}{2}\bigr)(k + 1) + (k + 1)}{\bigr(\frac{k + i}{2}\bigr)!\bigr(\frac{k - i}{2} + 1\bigr)! \bigr(\frac{k + j}{2} + 1\bigr)! \bigr(\frac{k - j}{2}\bigr)!} \geq \frac{\bigr(\frac{\hat j - \hat i}{2}\bigr)(k + 1) + (k + 1)}{\bigr(\frac{k + \hat i}{2}\bigr)!\bigr(\frac{k -\hat i}{2} + 1\bigr)! \bigr(\frac{k +\hat j}{2} + 1\bigr)! \bigr(\frac{k -\hat j}{2}\bigr)!}.
$$
We note that $\hat j - \hat i = j - i$, and furthermore the numerator $\bigr(\frac{j - i}{2}\bigr)(k + 1) + (k + 1) \geq 0$, since $j - i \geq -2$ by the fact that $(i,j) \in \cN$. Rearaging and cancelling terms shows that the above is equivalent to showing that
\begin{align*}
 \MoveEqLeft \textstyle \bigr(\frac{k + j}{2} + 2\bigr)\ldots \bigr(\frac{k + j + 2h + 2}{2} + 1\bigr)  \bigr(\frac{k + i}{2} + 1\bigr) \ldots \bigr(\frac{k + i + 2h + 2}{2}\bigr) \geq\\
& \textstyle \bigr(\frac{k - i - 2h -2}{2} + 2\bigr) \ldots \bigr(\frac{k - i}{2} + 1\bigr) \bigr(\frac{k - j - 2h -2}{2} + 1\bigr)\ldots \bigr(\frac{k - j}{2}\bigr).
\end{align*}
This inequality holds true since $(i,j) \in \cN$ and hence $i + j + 2h + 2 \geq 0$.
\item[ii.] In the case when $(\hat j + 2, \hat i - 2) \in S$ but either or both of the points $(\hat i, \hat j) \not \in S$ or $(j + 2, i - 2) \not \in S$ then (\ref{very:important:ineq}) continues to hold by i.
\item[iii.]  Assume $(\hat j + 2, \hat i - 2) \not \in S$ but $(\hat i, \hat j) \in S$ and $(j + 2, i - 2) \in S$. By the fact that $(\hat j + 2, \hat i - 2) \not \in S$ but $(\hat i, \hat j) \in S$ we have that either $\hat j = k$ or $\hat i = -k$. The former is impossible since otherwise $i = -(k + 2h + 2) < -k$. Hence the only viable option is to have $\hat i = -k$, in which case $j = k - (2h + 2)$. We need to show that
$$
\frac{k! \bigr[\bigr(\frac{j - i}{2}\bigr)(k + 1) + (k + 1)\bigr]}{\bigr(\frac{k + i}{2}\bigr)!\bigr(\frac{k - i}{2} + 1\bigr)! \bigr(\frac{k + j}{2} + 1\bigr)! \bigr(\frac{k - j}{2}\bigr)!} \geq \frac{1}{\bigr(\frac{k + \hat j}{2}\bigr)!\bigr(\frac{k - \hat j}{2}\bigr)!}.
$$
Since $(i,j) \in E_1 \setminus E_2$ we cannot have $j = i - 2$ (which can only happen in the intersection of the two ellipses), and thus $j \geq i$. Therefore it suffices to show 
\begin{align*}
\MoveEqLeft \textstyle \bigr(\frac{j - i}{2} + 1\bigr)(k + 1)!\bigr(\frac{k + i}{2} + 1\bigr)\ldots \bigr(\frac{k + i + 2h + 2}{2}\bigr) \\
& \textstyle \geq \bigr(\frac{k -i-2h -2}{2} + 1\bigr)\ldots \bigr(\frac{k - i}{2} + 1\bigr) \bigr(\frac{k + j}{2} + 1\bigr)! \bigr(\frac{k - j}{2}\bigr)!.
\end{align*}
We first observe that the following holds
$$
\textstyle \bigr(\frac{k + i}{2} + 1\bigr)\ldots \bigr(\frac{k + i + 2h + 2}{2}\bigr) \geq \bigr(\frac{k - j}{2}\bigr)!
$$
This is true since: first $j = k - (2 h + 2)$  and hence the RHS factorial contains $h + 1$ terms which is the same number of terms contained on the LHS; and second by the fact that $(i,j) \in \cN$ we have $i + j + 2h + 2 \geq 0$ and therefore each of the terms in the product on the LHS is greater than or equal to a corresponding term in the product of the RHS. 
Next, note that $\frac{j - i}{2} + 1 = \frac{k -i-2h -2}{2} + 1$ since $j = k - (2h + 2)$. Hence it is enough to show 
$$
\textstyle (k + 1)! \geq \bigr(\frac{k -i-2h -2}{2} + 2\bigr)\ldots \bigr(\frac{k - i}{2} + 1\bigr) \bigr(\frac{k + j}{2} + 1\bigr)!
$$
The latter is equivalent to
$$
\textstyle \bigr(\frac{k + j}{2} + 2\bigr)\ldots (k + 1) \geq \bigr(\frac{k -i-2h -2}{2} + 2\bigr)\ldots \bigr(\frac{k - i}{2} + 1\bigr) ,
$$
which holds true since the number of terms on both sides is the same, and $i \geq -k$.
\item[iv.] Assume $(\hat j + 2, \hat i - 2) \not \in S$ and $(\hat i, \hat j) \in S$, but $(j + 2, i -2) \not \in S$. This case follows by iii.
 \item[v.] Assume  $(\hat j + 2, \hat i - 2) \not \in S$ and  $(j + 2, i -2) \in S$ but $(\hat i, \hat j) \not \in S$. This case follows by case i. since its calculation implies 
 $$
{k \choose \frac{k + i}{2}}{k \choose \frac{k + j}{2}} - {k \choose \frac{k +  i - 2}{2}}{k \choose \frac{k +  j + 2}{2}} \geq 0.
 $$
 \end{itemize}

We have now shown that (\ref{very:important:ineq}) holds. The proof will be completed if we showed that by iterating over points in the set $(E_1 \setminus E_2) \cap S \cap \cN$  exhausts all points from the remaining three sets ---  $(E_1 \setminus E_2) \cap S \cap \cS$, ~ $(E_2 \setminus E_1) \cap S \cap \cW$ and $(E_2 \setminus E_1) \cap S \cap \cE$. We show this below.

First assume there is a point $(i,j) \in (E_2 \setminus E_1) \cap S \cap \cE$ so that $(j + 2, i - 2) \not \in (E_1 \setminus E_2) \cap S \cap \cN$. Since $(i,j) \in (E_2 \setminus E_1) \cap \cE$ we necessarily have $(j + 2, i - 2) \in (E_1 \setminus E_2) \cap \cN$. Therefore by our assumption $(j + 2, i - 2) \not \in S$. By the fact that $(i,j) \in \cE$ we know that $i \geq j + 2$ and therefore $k \geq j + 2 \geq j \geq -k$. Moreover $i = -k$ implies that $j \leq -k - 2$ contradicting $(i,j) \in S$, and therefore $i \geq -k + 2$. This implies that $(j + 2, i - 2) \in S$ which is a contradiction.

Next, let $(i,j) \in (E_2 \setminus E_1) \cap S \cap \cW$ so that $(-j - 2h -2, -i - 2h - 2) \not \in (E_1 \setminus E_2) \cap S \cap \cN$. The fact that $(i,j) \in (E_2 \setminus E_1) \cap \cW$ ensures $(-j - 2h -2, -i - 2h - 2) \in (E_1 \setminus E_2) \cap \cN$ and thus $(-j - 2h -2, -i - 2h - 2) \not \in S$. Since $(i, j) \in \cW$ we have that $i \leq -j -2h -2$ and $j \leq -i - 2h -2$ hence since $(i, j) \in S$ we also have  $-i - 2h -2 \geq -k$ and $-j - 2h -2 \geq -k$. Also $-i - 2h - 2 \leq k$, as otherwise we would have $i \leq -k -2h - 2 < -k$ which would be a contradiction. Similarly $-j - 2h - 2 \leq k$. Thus $(-j - 2h -2, -i - 2h - 2) \in S$ which is a contradiction. 

Finally assume that $(i,j) \in (E_1 \setminus E_2) \cap S \cap \cS$ but $(-i - 2h, -j - 2h - 4) \not \in (E_1 \setminus E_2) \cap S \cap \cN$. Since $(i,j) \in (E_1 \setminus E_2) \cap \cS$ it can be checked that $(-i - 2h, -j - 2h - 4) \in (E_1 \setminus E_2) \cap \cN$ and thus by assumption $(-i - 2h, -j - 2h - 4) \not \in S$. Note that $-i - 2h \leq k$ and $-j - 2h - 4\leq k$ or otherwise $i < -k - 2h \leq k$ and $j \leq -k -2h-4 < -k$ which contradicts our assumption that $(i,j) \in S$. We also have that $- k + 2 \leq j+ 2 \leq -i - 2h$ by the assumption that $(i,j) \in \cS\cap S$. Moreover $-i- 2h - (-j - 2h - 4) = 4 + j - i \leq 2$ and therefore $(-j - 2h - 4) \geq -k + 2 - 2 = -k$. Hence $(-i - 2h, -j- 2h - 4) \in S$, which is a contradiction.

With this the proof is complete.
\end{proof}

\begin{proof}[Proof of Proposition \ref{connectivity:low:temp}] To compare the null and alternative hypothesis we construct the following graphs. Let $G_0 = ([d],E_0)$, where 
\begin{align*}
 E_0 := \textstyle &\bigcup_{j = 0}^{\lfloor \frac{d}{s}\rfloor - 2}\{(u,v)\}_{s j + 1 \leq u < v \leq s (j +1)} \bigcup \bigcup_{j = 1}^{\lfloor \frac{d}{s}\rfloor - 2} \{(s j, s j + 1)\} \bigcup\\
& \textstyle \bigcup_{s(\lfloor \frac{d}{s}\rfloor - 1) + 1 \leq j \leq d-1} \{(j,j+1)\}.
\end{align*}
which is a union of $\lfloor \frac{d}{s}\rfloor - 2$, $s$-cliques connected via a path, and a disconnected path graph. See Figure \ref{proof:figure:conn} for a visualization.
\begin{figure}
\centering
\begin{tikzpicture}[scale=.65]
\SetVertexNormal[Shape      = circle,
                  FillColor  = white,
                  MinSize = 11pt,
                  InnerSep=0pt,
                  LineWidth = .5pt]
   \SetVertexNoLabel
   \tikzset{LabelStyle/.style = {below, fill = white, text = black, fill opacity=0, text opacity = 1}}
   \tikzset{EdgeStyle/.style= {thin,
                                double          = black,
                                double distance = .5pt}}
    \begin{scope}\grComplete[RA=1.5]{7}\end{scope}
    \begin{scope}[shift={(4,0)}]\grComplete[prefix=a,RA=1.5]{7}\end{scope}
    \begin{scope}[shift={(8,0)}]\grComplete[prefix=b,RA=1.5]{7}\end{scope}
    \begin{scope}[shift={(0,0)}]\grComplete[prefix=c,RA=1.5]{7}\end{scope}
    \begin{scope}[shift={(-.25,-3)}]\grPath[Math,prefix=u,RA=1.5,RS=0]{7}\end{scope}
        \Edge(a0)(b3)
        \Edge(c0)(a3)
                                                                           \tikzset{LabelStyle/.style = {right, fill = white, text = black, fill opacity=0, text opacity = 1}}
        \tikzset{EdgeStyle/.append style = {dashed, thin,  bend right=15}}
    \Edge[label=$\theta$](u0)(c6)
    \tikzset{EdgeStyle/.append style = {thin, dashed, bend left=15}}
                                                                               \tikzset{LabelStyle/.style = {below left, fill = white, text = black, fill opacity=0, text opacity = 1}}
    \Edge[label=$-\theta$](u0)(c4)
\end{tikzpicture}\caption{The graph $G_0$ is depicted with solid edges. A particular $G_{j,k}$ can be constructed if we add the two dashed edges to $G_0$.}\label{proof:figure:conn}
\end{figure}
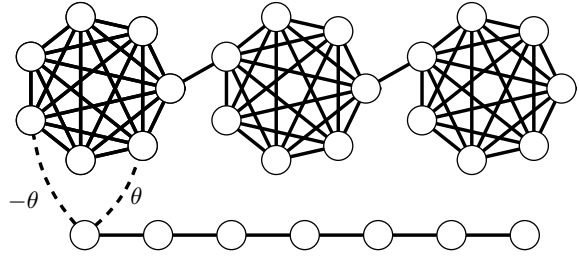
	
Based on $G_0$, we construct the ferromagnetic Ising measure $\PP_{G_0}$ which attaches an interaction of magnitude $\theta > 0$ to each edge in $G_0$ so that:
\begin{align}\label{conn:null:measure}
\PP_{G_0}(\bX) \varpropto \exp\biggr(\theta \sum_{(u,v) \in E_0} X_u X_v\biggr).
\end{align}
Next construct the connected graphs $G_{j,k} = ([d], E_{j,k})$ for $1 \leq j \leq \lfloor \frac{d}{s}\rfloor - 2$ and $3 \leq k \leq s - 1$, where
$$
E_{j,k} := E_0 \bigcup \bigg\{\bigg(s(\lfloor \frac{d}{s}\rfloor - 1) + j, sj + 2\bigg), \bigg(s(\lfloor \frac{d}{s}\rfloor - 1) + j, sj + k\bigg)\bigg\}.
$$
We define the Ising measure associated with $\PP_{G_{j,k}}$ by
\begin{align}\label{conn:alt:measures}
\PP_{G_{j,k}}(\bX) \varpropto \exp\biggr(\theta \sum_{(u,v) \in E_0} X_u X_v + \theta X_{s(\lfloor \frac{d}{s}\rfloor - 1) + j} ( X_{sj + 2} - X_{sj + k})\biggr),
\end{align}
which clearly belongs to model class (\ref{ising:model:measure:anti}) but not to (\ref{ferromagnet:def}). For two alternative graphs $G_{j_1, k_1}$ and $G_{j_2, k_2}$ we evaluate the ratio:
$$
\EE_{{G_0}} \frac{\PP_{G_{j_1, k_1}}}{\PP_{G_0}}\frac{\PP_{G_{j_2, k_2}}}{\PP_{G_0}} =  \frac{\EE_{{G_{j_1, k_1}}} \exp( \theta X_{\bar j_2}(X_{sj_2 + 2} - X_{sj_2 + k_2}))}{\EE_{{G_0}}\exp( \theta X_{\bar j_2}( X_{sj_2 + 2} -  X_{sj_2 + k_2}))}
$$
where for brevity we set $\bar j_2 := s(\lfloor \frac{d}{s}\rfloor - 1) + j_2$. Using identity (\ref{cosh:exp:identity}) we have:
\begin{align*}
\MoveEqLeft \EE_{{G_0}} \frac{\PP_{G_{j_1, k_1}}}{\PP_{G_0}}\frac{\PP_{G_{j_2, k_2}}}{\PP_{G_0}}  = \frac{\EE_{{G_{j_1, k_1}}} (1 + \tanh(\theta) X_{\bar j_2}X_{sj_2 + 2})(1 - \tanh(\theta) X_{\bar j_2}X_{sj_2 + k_2})}{\EE_{{G_0}}(1 + \tanh(\theta) X_{\bar j_2}X_{sj_2 + 2})(1 - \tanh(\theta) X_{\bar j_2}X_{sj_2 + k_2})} \\
& = 1 + \frac{\tanh(\theta)\EE_{{G_{j_1, k_1}}}[X_{\bar j_2}X_{sj_2 + 2} - X_{\bar j_2}X_{sj_2 + k_2}]}{1 - \tanh^2(\theta) \EE_{{G_{0}}} X_{sj_2 + 2}X_{sj_2 + k_2}}\\
& + \frac{\tanh^2(\theta)[\EE_{{G_{0}}} X_{sj_2 + 2}X_{sj_2 + k_2} - \EE_{{G_{j_1, k_1}}} X_{sj_2 + 2}X_{sj_2 + k_2}]}{1 - \tanh^2(\theta) \EE_{{G_{0}}} X_{sj_2 + 2}X_{sj_2 + k_2}}.
\end{align*}
where we used $\EE_{{G_{0}}}X_{\bar j_2}X_{sj_2 + k_2} = \EE_{{G_{0}}}X_{\bar j_2}X_{sj_2 + 2}= 0$. We distinguish two cases. \\

\noindent {\bf Case I.} First assume that $j_1 \neq j_2$. Since there are no simple paths connecting $sj_2 + 2$ with $sj_2 + k_2$ which pass through the vertex $\bar j_1 = s(\lfloor \frac{d}{s}\rfloor - 1) + j_1$, by Proposition \ref{restriction:prop} we have that:
$$
\EE_{{G_{0}}} X_{sj_2 + 2}X_{sj_2 + k_2} = \EE_{{G_{j_1, k_1}}} X_{sj_2 + 2}X_{sj_2 + k_2}.
$$
What is more, due to the symmetry of the construction we also have that:
$$
\EE_{{G_{j_1, k_1}}}X_{\bar j_2}X_{sj_2 + 2} = \EE_{{G_{j_1, k_1}}}X_{\bar j_2}X_{sj_2 + k_2}.
$$
Hence when $j_1 \neq j_2$ we have:
$$
\EE_{{G_0}} \frac{\PP_{G_{j_1, k_1}}}{\PP_{G_0}}\frac{\PP_{G_{j_2, k_2}}}{\PP_{G_0}} =  1.
$$
\noindent {\bf Case II.} Now consider the case $j_1 = j_2 =: j$. Below we control the terms $\EE_{{G_{j, k_1}}}[X_{\bar j}X_{sj + 2} - X_{\bar j}X_{sj + k_2}]$ (here $\bar j = s(\lfloor \frac{d}{s}\rfloor - 1) + j$) and $\EE_{{G_{0}}} X_{sj + 2}X_{sj + k_2} - \EE_{{G_{j, k_1}}}X_{sj + 2}X_{sj + k_2}$, starting from the former. The simple identity $|a -b| = 1 - ab$ for any $a,b \in \{\pm 1\}$ yields the estimate
$$
\EE_{{G_{j, k_1}}}[X_{\bar j}X_{sj + 2} - X_{\bar j}X_{sj + k_2}] \leq 1 -  \EE_{{G_{j, k_1}}}X_{sj + 2}X_{sj + k_2}.
$$
We have
\begin{align*}
\frac{\PP_{G_{j, k_1}}(X_{sj + 2}X_{sj+ k_2} = 1)}{\PP_{G_{j, k_1}}(X_{sj + 2}X_{sj + k_2} = -1)} & \geq \min_{\xi \in \{\pm 1\}}\frac{\PP_{G_{j, k_1}}(X_{sj + 2}X_{sj + k_2} = 1 | X_{\bar j} = \xi)}{\PP_{G_{j, k_1}}(X_{sj + 2}X_{sj + k_2} = -1| X_{\bar j} = \xi)} \\
& \geq \frac{\PP_{G_{0}}(X_{sj + 2}X_{sj + k_2} = 1)}{\exp(4\theta)\PP_{G_{0}}(X_{sj + 2}X_{sj + k_2} = -1)} \\
& \geq \frac{\exp(2\theta \lfloor s/4 \rfloor )}{\exp(4\theta)(2 \lfloor s/4 \rfloor - 1)},
\end{align*}
where the last inequality follows by (\ref{prob:ratio:ineq}), since the vertices $sj + 2$ and $sj + k_2$ belong to the same $s$-clique in the graph $G_0$, and hence they are on the same side of a $2 \lfloor s/4 \rfloor \times 2 \lfloor s/4 \rfloor$-biclique. Putting the last two inequalities together with the identity $ \EE_{{G_{j, k_1}}}X_{sj + 2}X_{sj + k_2}  = \PP_{G_{j, k_1}}(X_{sj + 2}X_{sj + k_2} = 1) - \PP_{G_{j, k_1}}(X_{sj + 2}X_{sj + k_2} = -1)$ gives:
\begin{align*}
\MoveEqLeft \EE_{{G_{j, k_1}}}[X_{\bar j}X_{sj + 2} - X_{\bar j}X_{sj + k_2}] \leq 1 -  \EE_{{G_{j, k_1}}}X_{sj + 2}X_{sj + k_2} \\
& \leq \frac{2(2 \lfloor s/4 \rfloor - 1)}{\exp((2 \lfloor s/4 \rfloor - 4)\theta ) + (2 \lfloor s/4 \rfloor - 1)}.
\end{align*}
Now we focus on the term $\EE_{{G_{0}}} X_{sj + 2}X_{sj + k_2} - \EE_{{G_{j, k_1}}}X_{sj + 2}X_{sj + k_2}$. Note that
$$
\EE_{{G_{0}}} X_{sj + 2}X_{sj + k_2} - \EE_{{G_{j, k_1}}}X_{sj + 2}X_{sj + k_2} \leq 1 - \EE_{{G_{j, k_1}}}X_{sj + 2}X_{sj + k_2},
$$
and therefore the previous bound applies. We have established
\begin{align*}
\EE_{{G_0}} \frac{\PP_{G_{j, k_1}}}{\PP_{G_0}}\frac{\PP_{G_{j, k_2}}}{\PP_{G_0}} & \leq 1+ \frac{\tanh(\theta) + \tanh^2(\theta)}{1 - \tanh^2(\theta)} \frac{2(2 \lfloor s/4 \rfloor - 1)}{\exp((2 \lfloor s/4 \rfloor - 4)\theta ) + (2 \lfloor s/4 \rfloor - 1)} \\
& \leq 1+ \frac{4(2 \lfloor s/4 \rfloor - 1)}{(1 - \tanh^2(\theta))(\exp((2 \lfloor s/4 \rfloor - 4)\theta ) + (2 \lfloor s/4 \rfloor - 1))}.
\end{align*}
Let $\overline \PP^{\otimes n}$ be the measure such that $\overline \PP^{\otimes n} = \frac{1}{(\lfloor\frac{d}{s}\rfloor - 2)(s-3)}\sum_{j, k} \PP^{\otimes n}_{G_{j,k}}$. We have:
\begin{align*}
\MoveEqLeft D_{\chi^2}(\overline \PP^{\otimes n}, \PP^{\otimes n}_{G_0})  = \frac{1}{[(\lfloor\frac{d}{s}\rfloor - 2)(s-3)]^2}\sum_{j_1, j_2, k_1,k_2}\EE_{\PP^{\otimes n}_{G_0}} \frac{\PP^{\otimes n}_{G_{j_1,k_1}}}{\PP^{\otimes n}_{G_0}}\frac{\PP^{\otimes n}_{G_{j_2,k_2}}}{\PP^{\otimes n}_{G_0}}\\
& \leq \frac{(\lfloor\frac{d}{s}\rfloor - 2)^2 - (\lfloor\frac{d}{s}\rfloor - 2)}{(\lfloor\frac{d}{s}\rfloor - 2)^2} + \frac{\bigg[1 + \frac{4(2 \lfloor s/4 \rfloor - 1)}{(1 - \tanh^2(\theta))(\exp((2 \lfloor s/4 \rfloor - 4)\theta ) + (2 \lfloor s/4 \rfloor - 1))}\bigg]^n}{(\lfloor\frac{d}{s}\rfloor - 2)(s-3)^2}\\
& \leq 1 + \frac{\exp\bigg[n \frac{4(2 \lfloor s/4 \rfloor - 1)}{(1 - \tanh^2(\theta))(\exp((2 \lfloor s/4 \rfloor - 4)\theta ) + (2 \lfloor s/4 \rfloor - 1))}\bigg] - (s-3)^2}{(\lfloor\frac{d}{s}\rfloor - 2)(s-3)^2}.
\end{align*}
Hence if the last quantity is $1 + o(1)$ the risk is $1$. The latter is ensured under condition (\ref{conn:testing:cond:anti}) from the statement, and hence the proof is complete. 

\end{proof}



\section{Correlation Testing for General Models}\label{correlation:testing:general:models}

\begin{proof}[Proof of Proposition \ref{mean:comb:inf:testing:general}] 
First consider the case when the true graph $G_{\wb} \in \cG_0(\cP, \theta, \Theta)$. Then by (\ref{expectation:bound}) we have that 
$$
\max_{u, v} |\hat \EE X_u X_v - \EE_{\theta, \tilde G_{\tilde \wb}} X_u X_v| \leq \varepsilon(\delta),
$$
with probability at least $1 - \delta$. Next, suppose that $G_{\wb} \in \cG_1(\cP, \theta, \Theta)$. 
Let 
$$
(u^*, v^*) =  \argmax_{u,v} |\EE_{\theta, \tilde G_{\tilde \wb}}X_{u} X_{v}- \EE_{\theta,  G_{\wb}}X_u X_v|
$$
We have
\begin{align*}
\MoveEqLeft \max_{u, v} |\hat \EE X_{u} X_{v} - \EE_{ \theta, \tilde G_{\tilde \wb}}X_{u} X_{ v}|  \geq |\hat \EE X_{u^*} X_{v^*} - \EE_{ \theta, \tilde G_{\tilde \wb}} X_{u^*} X_{v^*} | \\
& \geq  |\EE_{ \theta, \tilde G_{\tilde \wb}} X_{u^*} X_{v^*} - \EE_{ \theta, G_{\wb}} X_{u^*} X_{v^*}| - |\hat \EE X_{u^*} X_{v^*} - \EE_{ \theta, G_{\wb}}X_{u^*} X_{v^*}| \\
& \geq   |\EE_{ \theta, \tilde G_{\tilde \wb}} X_{u^*} X_{v^*} - \EE_{ \theta, G_{\wb}} X_{u^*} X_{v^*}| - \tau\\
& \geq \frac{\sinh^2(\theta/4)}{2s\Theta(3 \exp(2s\Theta) + 1)} -  \tau,
\end{align*}
where the next to last inequality holds with probability at least $1 - \delta$ by Lemma \ref{simple:concentration:lemma}, and the last inequality holds by Lemma 6 of \cite{santhanam2012information} since $\tilde G_{\tilde \wb} \in \cG_0(\cP, \theta, \Theta)$ and hence $G \neq \tilde G$. This completes the proof. 
\end{proof}

\begin{proof}[Proof of Lemma \ref{simple:lemma:F:scan}] The proof is a direct consequence of Lemma \ref{simple:concentration:lemma}.
\end{proof}

\begin{proof}[Proof of Proposition \ref{test:cycle:generic:proof}] Clearly the output of Algorithm \ref{cycle:test:step1} satisfies $\tilde T_{\tilde \wb} \in \cG_0(\cP, \theta, \Theta)$ and thus for $\tilde T$ we have $\cP(\tilde T) = 0$ by design. Hence it remains to argue that (\ref{expectation:bound}) holds with the specified $\varepsilon(\delta)$.  Using Lemma \ref{simple:concentration:lemma} we have that 
\begin{align}\label{expectation:difference:appendix}
\max_{u,v} |\hat \EE X_u X_v - \EE_{\theta, G_{\wb}} X_u X_v | \leq \tau.
\end{align}
with probability at least $1 - \delta$. Below we assume this event holds, so that all statements we make should be understood to hold with probability at least $1 - \delta$.

We begin by arguing that $\tilde T = G$. First we will show that the set $E(\tilde T) \setminus E(G) = \varnothing$. Suppose the contrary, and consider the first step, say $m$, at which an edge $e = (u,v)$ from the set $E(\tilde T) \setminus E(G)$ was added to $\tilde T$. Let $P$ (possibly $P = \varnothing$) denote the path connecting $u$ with $v$ in the graph $G$. First suppose $P = \varnothing$. Then $\EE_{\theta, G_{\wb}}X_u X_v = 0$, so $|\hat\EE X_u X_v| \leq \tau$ and since $\tanh \theta - \tau \geq \tau$ this edge is pruned at the end of the procedure which is a contradiction. So $P \neq \varnothing$ and thus $|P| \geq 2$ (if $P$ was a single edge it follows that $e \in E(G)$ which is a contradiction). By Proposition \ref{restriction:prop} and Lemma \ref{chain:graph} we know
$$
\EE_{\theta, G_{\wb}} X_u X_v = \prod_{(k, \ell) \in P} \EE_{\theta, G_{\wb}} X_k X_\ell,
$$
and therefore $|\hat \EE X_u X_v| \leq  \prod_{(k, \ell) \in P} |\EE_{\theta, G_{\wb}} X_k X_\ell| + \tau$. Since we have added $e$ at the $m$\textsuperscript{th} step of building the MST, and $e \not \in P$, then there exists an edge from $e' \in P$, say $e' = (u',v')$, which does not belong to $\tilde T$, since otherwise adding $e$ would make a cycle in $\tilde T$. Adding edge $e$ over $e'$ means that the inequality $|\hat \EE X_{u'}X_{v'}| \leq |\hat \EE X_u X_v|$ holds. Hence
\begin{align*}
|\EE_{\theta, G_{\wb}} X_{u'}X_{v'}| - \tau & \leq |\hat \EE X_{u'}X_{v'}| \leq |\hat \EE X_u X_v| \leq  \prod_{(k, \ell) \in P} |\EE_{\theta, G_{\wb}} X_k X_\ell| + \tau \\
& \leq |\EE_{\theta, G_{\wb}} X_{u'} X_{v'}|\tanh(\Theta) + \tau.
\end{align*}
However, $|\EE_{\theta, G_{\wb}} X_{u'}X_{v'}| \geq \tanh(\theta)$ and since $\tanh(\theta)(1 - \tanh(\Theta)) > 2\tau$ this contradicts the inequality in the preceding display. We conclude that $E(\tilde T) \subseteq E(G)$. If there exists an edge $e \in  E(G) \setminus E(\tilde T)$, then when building the MST this edge was either pruned at the end of the procedure, or a no-edge was preferred to it. However these two scenarios are also impossible using the same logic as before. Hence $\tilde T = G$.

By our construction we have that for $(u,v) \in E(G)$ we either have $\EE_{\theta, \tilde T_{\tilde \wb}} X_u X_v = \hat \EE X_u X_v$, or $|\hat \EE X_u X_v| > \tanh(\Theta)$ or $|\hat \EE X_u X_v| < \tanh(\theta)$ in which case $|\EE_{\theta, \tilde T_{\tilde \wb}} X_u X_v| = \tanh (\Theta)$ or $|\EE_{\theta, \tilde T_{\tilde \wb}} X_u X_v| = \tanh(\theta)$ respectively. Putting these two cases together we conclude that
$$
\max_{(u,v) \in E(G)} |\hat \EE X_u X_v - \EE_{\theta, \tilde T_{\tilde \wb}} X_u X_v | \leq \tau.
$$

Clearly if two vertices $u$ and $v$ are disconnected in $G$, we also have $|\hat \EE X_u X_v - \EE_{\theta, \tilde T_{\tilde \wb}} X_u X_v | \leq \tau$. Next take two connected vertices $u, v$ and let $P$ denote the (unique) path between them. Consider
\begin{align*}
\MoveEqLeft |\hat \EE X_u X_v - \EE_{\theta, \tilde T_{\tilde \wb}} X_u X_v|  =  \bigr|\hat \EE X_u X_v - \prod_{(k,\ell) \in P} \EE_{\theta, \tilde T_{\tilde \wb}} X_k X_\ell\bigr| \\
& \leq \bigr|\hat \EE X_u X_v - \EE_{\theta, G_{\wb}} X_u X_v | + |\prod_{(k,\ell) \in P}  \EE_{\theta, G_{\wb}} X_k X_\ell - \prod_{(k,\ell) \in P}  \EE_{\theta, \tilde T_{\tilde \wb}} X_k X_\ell\bigr| \\
& \leq  \tau +  |\prod_{(k,\ell) \in P}  \EE_{\theta, G_{\wb}} X_k X_\ell - \prod_{(k,\ell) \in P}  \EE_{\theta, \tilde T_{\tilde \wb}} X_k X_\ell\bigr|.
\end{align*}
Suppose that the length of $P$ is $l$.  For brevity put $\tau_{uv} := \EE_{\theta, G_{\wb}} X_u X_v - \EE_{\theta, \tilde T_{\tilde \wb}} X_u X_v$.  By (\ref{expectation:difference:appendix}), and the fact that our corrections of $\hat \EE X_uX_v$ to $\EE_{\theta, \tilde T_{\tilde \wb}} X_u X_v$ will only result in better estimation of $\EE_{\theta,  G_{\wb}} X_u X_v$ we know that $|\tau_{uv}| \leq \tau$ for all $(u,v) \in E(G) = E(\tilde T)$. 

Note that condition $\tanh(\theta)(1 - \tanh(\Theta)) > 2\tau$ implies that $\tanh(\Theta) < 1$ and $1 > 2\tau + \tanh(\Theta)$. A simple expansion gives 
\begin{align*}
\MoveEqLeft |\prod_{(k,\ell) \in P}  \EE_{\theta,G_{\wb}} X_k X_\ell - \prod_{(k,\ell) \in P}  \EE_{\theta, \tilde T_{\tilde \wb}} X_k X_\ell\bigr|\\
&= \bigr|\prod_{(k,\ell) \in P}  (\EE_{\theta, \tilde T_{\tilde \wb}} X_k X_\ell + \tau_{k\ell}) - \prod_{(k,\ell) \in P}  \EE_{\theta, \tilde T_{\tilde \wb}} X_k X_\ell\bigr| \\
& \leq \sum_{\substack{S \cup R = P\\ S \cap R = \varnothing, |S| < l}} \prod_{(k,\ell) \in S} |\EE_{\theta, \tilde T_{\tilde \wb}}  X_{k}X_{\ell}|\prod_{(k,\ell) \in R} |\tau_{k\ell}|\\
& \leq (\tau + \tanh(\Theta))^l - \tanh(\Theta)^l \\
& \leq \sum_{j = 1}^l \tau (\tau + \tanh(\Theta))^{l - j}\tanh(\Theta)^{j - 1}\\
& \leq \frac{\tau }{1 - \tanh(\Theta)},
\end{align*}
where the next to last inequality follows easily by induction on $l$, and the last inequality is due to the fact that $\tau + \tanh(\Theta) \leq 1$ and $\tanh(\Theta) < 1$. Hence we have successfully argued that
$$
\max_{u,v \in [d]} |\hat \EE X_u X_v - \EE_{\theta, \tilde T_{\tilde \wb}} X_u X_v| \leq \tau \frac{2 - \tanh(\Theta)}{1 - \tanh(\Theta)},
$$
which completes the proof. 
\end{proof}

\end{document}